\documentclass[12pt]{article}%letter,
\usepackage[title]{appendix}
\usepackage{algpseudocode}
\usepackage{algorithm}
\usepackage{url} 
\usepackage{amsmath}
\usepackage{amsthm}

\usepackage{amssymb}

\usepackage{mathptmx}

\usepackage{caption}%, subcaption

\usepackage{xparse} % for {{}}
\usepackage{stmaryrd} % for [[]]

\usepackage{multirow}

\usepackage{geometry}
\usepackage{bm}

\usepackage{enumitem}
\usepackage{extarrows}
\usepackage{psfrag}
\usepackage{listings}
\usepackage{grffile}
\usepackage{amsfonts}
\usepackage{float}
\usepackage{booktabs}

\renewcommand{\thealgorithm}{} %algorithm编号置为零

\usepackage{color}
\usepackage{graphicx}
\usepackage{subfigure}

\makeatletter
\def\widebreve{\mathpalette\wide@breve}
\def\wide@breve#1#2{\sbox\z@{$#1#2$}%
	\mathop{\vbox{\m@th\ialign{##\crcr
				\kern0.08em\brevefill#1{0.8\wd\z@}\crcr\noalign{\nointerlineskip}%
				$\hss#1#2\hss$\crcr}}}\limits}
\def\brevefill#1#2{$\m@th\sbox\tw@{$#1($}%
	\hss\resizebox{#2}{\wd\tw@}{\rotatebox[origin=c]{90}{\upshape(}}\hss$}
\makeatletter

\usepackage{hyperref} 
\usepackage{xcolor}
\hypersetup{
	colorlinks,
	linkcolor={green!50!black},%{red!50!black},
	citecolor={green!50!black}%,
}

\allowdisplaybreaks[3]
\def\vec#1{{\bf #1}}

\newcommand{\figref}[1]{Fig.~\ref{#1}}

\newtheorem{theorem}{Theorem}[section]
\newtheorem{proposition}{Proposition}[section]
\newtheorem{lemma}{Lemma}[section]

\newtheorem{remark}{Remark}[section]

\newtheorem{expl}{Example}[section]
\newtheorem{algo}{Algorithm}[section]
\providecommand{\keywords}[1]
{
	\small
	\textbf{\textit{Keywords:}} #1
}

\title{{\large \bf A Physical-Constraint-Preserving Finite Volume WENO Method for Special Relativistic Hydrodynamics on Unstructured Meshes}\footnote{The work of Y.~Chen is partially supported by National Natural Science Foundation of China (grant No.~11901460). The work of K.~Wu is partially supported by National Natural Science Foundation of China (grant No.~12171227).}} % Triangular

\author{
	Yaping Chen\thanks{
		School of Mathematics and Statistics,
		Xi’an Key Laboratory of Scientific Computation and Applied Statistics, NPU-UoG International Cooperative Lab for Computation and Application in Cardiology,
		Northwestern Polytechnical University, Xi’an 710129, Shaanxi
		Province, P.R.~China.  ({\tt ypchen@nwpu.edu.cn}). },
	\quad 
	Kailiang Wu\thanks{Corresponding Author. Department of Mathematics \& SUSTech International Center for Mathematics, Southern University of Science and Technology, and National Center for Applied Mathematics Shenzhen (NCAMS), Shenzhen, Guangdong 518055, China.  ({\tt wukl@sustech.edu.cn}).  }
}

\date{}

\begin{document}

	\maketitle

	\vspace{-3mm}

	\begin{abstract}
		This paper presents a highly robust third-order accurate finite volume weighted essentially non-oscillatory (WENO) method for special relativistic hydrodynamics on unstructured triangular meshes. We rigorously prove that the proposed method 
		is physical-constraint-preserving (PCP), namely, 
		always preserves the positivity of the pressure and the rest-mass density as well as the subluminal constraint on the fluid velocity. The method is built on a highly efficient compact WENO reconstruction on unstructured meshes, a simple PCP limiter, the provably 
		PCP 
		 property of the Harten--Lax--van Leer flux, and third-order strong-stability-preserving time discretization. Due to the relativistic effects, the primitive variables (namely, the rest-mass density, velocity, and pressure) are highly nonlinear implicit functions in terms of the  conservative variables, making the design and analysis of our method nontrivial. To address the difficulties arising from the strong nonlinearity, we adopt a novel quasilinear technique for the theoretical proof of the PCP property. 
		Three provable convergence-guaranteed iterative algorithms  are also introduced for the robust recovery of primitive quantities from admissible conservative variables. 
		We also propose a slight modification to an existing WENO reconstruction 
		to ensure the scaling invariance of the nonlinear weights and thus to 
		accommodate the homogeneity of the evolution operator, leading to the advantages of the modified WENO reconstruction in  resolving multi-scale wave structures.     
		Extensive  
		numerical examples are presented to demonstrate the robustness, expected accuracy, and high resolution of the proposed method.

		\vspace{2mm}
		\noindent
		\keywords{Physical-constraint-preserving,
			special relativistic hydrodynamics, WENO,
			finite volume, high-order accuracy, unstructured mesh}

	\end{abstract}

	%This is the first time that such a constraint-preserving property is rigorously proven for multidimensional RMHD schemes.

	\newpage

	\section{Introduction}
\label{sec:intro}

Relativistic hydrodynamics (RHD) plays a key role in high-energy astrophysical phenomena and laboratory plasma experiments when the fluid moves close to the speed of light
	or/and 
its internal energy is comparable to the rest-mass density. 
The strong nonlinearity of the RHD system makes it extremely difficult to obtain its analytical solutions. Hence numerical simulation has become a powerful and primary tool to solve and understand it.  
In the past several decades, many high-resolution and high-order accurate numerical methods have been developed for the RHD equations, including but not limited to 
finite volume methods  (e.g.~\cite{mignone2005hllc,tchekhovskoy2007wham,BalsaraKim2016,chen2021second}), finite difference methods (e.g.~\cite{dolezal1995,del2002efficient,radice2012thc,WuTang2015}), and discontinuous Galerkin (DG) methods (e.g.~\cite{radice2011discontinuous,zhao2013runge,KIDDER201784,teukolsky2016formulation}). %,Zanotti2015solving,bugner2016solving,fambri2018ader
Adaptive mesh refinement \cite{2006Zhang} and adaptive moving mesh \cite{he2012adaptive1} were used to further improve the resolution of discontinuities and complicated RHD flow structures. %,HeTang2012RMHD
The interested readers are referred to the review  \cite{marti2003numerical,Marti2015}, the textbook  \cite{rezzolla2013relativistic}, and a limited list of some recent works \cite{endeve2019thornado,WuShu2019SISC,marquina2019capturing,mewes2020numerical} as well as references therein.

Most of the existing numerical schemes for RHD equations were designed on structured meshes, and there are only a few works \cite{dumbser2009very, duffell2011tess} done on unstructured meshes which are highly desirable for some  applications in problems with complex geometries. 
We are interested in developing robust high-order finite volume schemes for RHD on unstructured meshes. 
An important building block for high-order finite volume methods is the reconstruction of the variables inside the computational cell from the cell averages. For example, 
the well-known weighted essentially non-oscillatory (WENO)
reconstruction \cite{liu1994weighted,jiang1996},  stemming from the essentially non-oscillatory (ENO) reconstruction \cite{HARTEN1987231}, has become one of the most popular reconstruction techniques on structured meshes. 
The strategies of ENO or WENO reconstructions were extended to unstructured triangular meshes in, for example, \cite{abgrall1994,hu1999,liu2013robust}. 
To make the  
order of WENO accuracy higher than that of the reconstruction on each smaller stencil, the linear weights 
in the classic WENO reconstruction should be carefully designed \cite{hu1999}. 
This requirement makes the extension and implementation of the classic WENO reconstruction 
on unstructured meshes difficult and complicated, 
as the desired optimal linear weights depend on the quadrature points and the topological  structure of the mesh,  and moreover, the linear weights could easily become negative which need some special treatment \cite{shi2002technique}.  
Such difficulty may be avoided if one decreases the WENO  
accuracy order 
on the combined large stencil to
the highest accuracy order among the small candidate stencils  
%if one does not require 
%the WENO accuracy order higher than that of the reconstruction on each smaller stencils  
(see, e.g., \cite{friedrich1998weighted,dumbser2007arbitrary,Dumbser2007,zhu2016new,balsara2020}), %,cheng2008third
so that the linear weights could be chosen equally or even rather arbitrarily as long as their summation equals one. %,semplice2016adaptive,levy1999central,levy2000compact,
Recently, Zhu and Qiu \cite{zhu2016new,zhu2017new} proposed a new type of WENO reconstruction, 
which is based on a combination of a high degree polynomial 
with several linear polynomials. 
This new WENO reconstruction is highly compact and efficient, and its linear weights can also be chosen arbitrarily. 
As a result, it has been easily extended to two-dimensional (2D) triangular meshes \cite{zhu2018new} and three-dimensional (3D) tetrahedral meshes \cite{zhu2017tetrahedral}. 
More recently, the multi-resolution WENO schemes 
with similar linear weights were proposed in \cite{zhu2018MRWENO}.%,zhu2019new 

Although these WENO schemes are stable and robust in many numerical experiments, they may fail to simulate ultra-relativistic flows with large Lorentz factor (high speed), 
low pressure, low density, and/or strong discontinuities. 
A major cause of the failure is the violation of the intrinsic physical constraints, namely, the positivity of the pressure and the rest-mass density as well as the subluminal constraint on the fluid velocity. 
In fact, if any of these constraints are numerically violated, the corresponding discrete
problem could become ill-posed as the hyperbolicity of the system is lost, which may finally lead to numerical instability or blowup of the code. 
It is therefore necessary to develop physical-constraint-preserving (PCP) numerical methods. 
In the past decade, two types of limiters were developed for constructing 
high-order bound-preserving type schemes for hyperbolic conservation laws. 
One is the simple scaling limiting procedure, which was first proposed by 
Zhang and Shu for 
scalar conservation laws \cite{zhang2010} and the compressible Euler equations \cite{zhang2010b} on structured rectangular meshes, and later extended to unstructured  triangular meshes \cite{zhang2012}. 
Another type is the flux-correction limiting procedure; see e.g.~\cite{xu2014parametrized,Hu2013,Xiong2016}. 
 We also refer the interested reader to recent thorough reviews in \cite{XuZhang2017,Shu2018},  and some recent works \cite{qin2018implicit,Wu2017a,WuShu2019} as well as references therein.  
These limiting techniques were also generalized to achieve 
PCP schemes for RHD. 
The first PCP work for RHD was done in \cite{WuTang2015}, where an explicit form of the 
admissible state set was established, 
the local Lax--Friedrichs scheme was rigorously proven to be PCP, and 
high-order PCP finite difference WENO schemes were proposed with flux-correction PCP  limiters. 
Bound-preserving DG methods were designed for the special RHD by Qin, Shu, and Yang in \cite{QinShu2016} with a provable $L^1$-stability. 
The PCP Lagrangian finite volume schemes with the HLLC flux were later developed in \cite{LingDuanTang2019}. 
Recently, a minimum principle on specific entropy and high-order accurate invariant region preserving numerical methods 
were proposed in \cite{WuMEP2021} for the special RHD. 
These works were focused on the special RHD system with an ideal equation of state. 
PCP central DG schemes were constructed in \cite{WuTang2017ApJS}, where a general equation of state was considered. 
Frameworks of designing high-order PCP methods were established in \cite{Wu2017} for  general RHD.  
The design and analysis of PCP schemes were carried out in 
 \cite{WuTangM3AS,wu2021provably} for the relativistic magnetohydrodynamics (MHD), 
which extended the positivity-preserving non-relativistic MHD schemes \cite{Wu2017a,WuShu2018,WuShu2019}. 
The analysis revealed for the first time that 
the PCP property of MHD schemes is strongly connected with a discrete divergence-free condition on the magnetic field \cite{WuTangM3AS,Wu2017a}. 
  Besides, a flux limiter was proposed in \cite{radice2014high} to enforce the positivity of the rest-mass density, and 
  a subluminal reconstruction was developed in \cite{BalsaraKim2016} to 
  ensure the subluminal bound of the fluid velocity.

The aim of this paper is to construct, analyze, and implement a robust PCP third-order finite volume method for the special RHD on unstructured triangular meshes. 
A distinctive feature of the proposed method lies in its desirable non-oscillatory property, homogeneity, and provably PCP property. 
	To achieve this goal, we will make the following efforts in this work.
	\begin{itemize} 
		\item  Due to relativistic effects, neither 
		 the primitive quantities nor the flux can be explicitly expressed by the conservative variables. 
		 This makes the design and analysis of the PCP schemes nontrivial in the RHD case. 
		To address the difficulties arising from the strong nonlinearity, we adopt a novel quasilinear technique to theoretically prove the PCP property of our method. This technique was named Geometric Quasi-Linearization (GQL) \cite{Wu2021GQL} due to its intrinsic geometric meaning; see the general GQL framework recently established in \cite{Wu2021GQL}.
		\item 
		Due to the nonlinear implicit mappings from 
		the conservative variables to the primitive quantities and flux, 
		 it requires to solve a nonlinear algebraic equation to recover the corresponding primitive quantities from the conservative variables in numerical computations. 
		 We present 
		 three provable convergence-guaranteed algorithms for the robust recovery of physically admissible primitive quantities.		
		 \item To achieve high-order accuracy in spatial discretization, we employ the compact and efficient WENO reconstruction proposed in \cite{zhu2018new} on triangular meshes.  
		 It will be observed that the nonlinear weights used in  \cite{zhu2018new} do not satisfy certain  scaling-invariance property, which seems important for simulating multi-scale problems 
		 (see Examples \ref{1Driemann3} and \ref{1DMSP} of this paper).   
		 We propose a slight modification to the nonlinear weights to 
		 ensure its scaling invariance and to accommodate the homogeneity of the evolution operator. 
		 The modified WENO reconstruction will be shown to be advantageous in resolving multi-scale wave structures. 
		\item To validate the robustness, accuracy, and effectiveness of our method, 
		we conduct extensive numerical tests on unstructured triangular meshes. 
		It will be shown that our PCP scheme is capable of simulating benchmark problems and more   
		challenging problems in regular and irregular domains successfully, such as a relativistic forward-facing step problem with initial velocities $v_1=0.999$ and shock-vortex interaction problems involving the low pressure and density of $1.78\times10^{-20}$ and $7.8\times10^{-15}$, respectively.
	\end{itemize}

This paper is organized as follows. We will introduce the governing equations of special RHD in Section \ref{sec:GovenEqn}. Section
\ref{Sec:Numerical_method} presents high-order PCP finite volume method for the RHD equations, including the outline and key ingredients of our method in Subsection \ref{sec:outline}, high-order characteristic WENO reconstruction on unstructured triangular mesh with a modification of the nonlinear weights to be scaling invariant in Subsection \ref{sec:RCST}, 
a PCP limiting operator in Subsection \ref{sec:constraint-preserving}, three convergence-guaranteed algorithms for primitive variables recovery in Subsection \ref{sec:recovery}, and the rigorous proof of the PCP property in Subsection \ref{sec:CPP}. 
In Section \ref{sec:examples}, we provide extensive one-dimensional (1D) and 2D numerical tests 
on demanding RHD problems to validate the effectiveness of our method. Section \ref{sec:conclusions} concludes the paper.

	\section{Governing equations of special relativistic hydrodynamics}
	\label{sec:GovenEqn}
	The equations governing RHD can be formulated in the covariant form as
		\begin{equation}\label{eq:covariant}
		\begin{cases}
			\partial_{\alpha}(\rho u^{\alpha}) = 0,\\
			\partial_{\alpha} T^{\alpha\beta}= 0,
		\end{cases}
	\end{equation}
	which describe the conservation laws of the
	baryon number density and the stress-energy tensor $T^{\alpha\beta}$. Here $\rho$ represents
	the rest-mass density,  $u^{\alpha}$ stands for the four-velocities, and $\partial_{\alpha}=\partial_{x^{\alpha}}$ is the covariant derivative.
	We have employed in \eqref{eq:covariant} the Einstein summation convention over the repeated index $\alpha$, with the Greek indices running from 0 to 3.
	For an ideal fluid the stress-energy tensor takes the form of 
	$$T^{\alpha\beta}=\rho H u^{\alpha}u^{\beta}+pg^{\alpha\beta},$$
	where $p$ is the pressure, $H=1+e+\frac{p}{\rho}$ denotes the specific enthalpy,
	 $e$ represents the specific internal energy, and the
	 geometrized unit system is used so that the speed of light $c$ in
	 vacuum equals one. Equations \eqref{eq:covariant} are closed by an equation of state, e.g., $e=e(p,\rho)$. In this paper, we focus on the ideal equation of state, which reads
	 \begin{equation}\label{EOS}
	 	 e = \frac{p}{ (\Gamma-1)\rho }
	 \end{equation}
	 with the constant $\Gamma \in (1,2]$ being the ratio of specific heats; the restriction $\Gamma \le 2$ is required by compressibility assumptions and the relativistic causality (cf.~\cite{WuTang2015}).

	For the special relativity, the spacetime metric
	 $(g^{\alpha\beta})_{4\times4}$ is Minkowski's tensor $\text{diag}\{-1,1,1,1\}$. The four-dimensional space-time coordinates become $(x^{\alpha})=(t,x_1,x_2,x_3)^\top$, and the four-velocities become $(u^{\alpha})=\gamma(1,v_1,v_2,v_3)^\top$, where
	$\gamma=1/\sqrt{1-\|\bm{v}\|^2}$  is the Lorentz factor with $\| \cdot \|$ denoting $2$-norm of the fluid velocity vector ${\bm v}=(v_1,v_2,v_3)$.
	Thus, in the special RHD case, system \eqref{eq:covariant} can be rewritten as
	\begin{equation}\label{eq:RHD3D}
	\frac{\partial \vec{U}}{\partial t} + \sum^{3}_{i=1}\frac{\partial\vec F_{i}(\vec{U})}{\partial x_i} = {\bf 0},
\end{equation}
where the conservative vector ${\bf U}$ and the fluxes, $\vec F_{i}$, $1\le i \le 3$, are defined by
	\begin{align}\label{eq:UF1}
		& {\bf U} = (D, m_1, m_2, m_3, E )^\top,
		& {\bf F}_1= (D v_1, m_1 v_1 + p, m_2 v_1, m_3 v_1, m_1)^\top,
		\\ \label{eq:UF2}
		& {\bf F}_2= (D v_2, m_1 v_2 , m_2 v_2+p, m_3 v_2, m_2)^\top,
		&{\bf F}_3= (D v_3, m_1 v_3 , m_2 v_3, m_3 v_3+p, m_3)^\top,
	\end{align}
with
\begin{equation}\label{eq:UU}
D=\rho \gamma, \qquad m_i=\rho H \gamma^2 v_i, \qquad E=\rho H \gamma^2 -p
\end{equation}
denoting the mass density, momentum in $x_i$-direction, and energy, respectively.

From equations \eqref{eq:UF1}--\eqref{eq:UU}, we see that
the conservative vector ${\bf U}$ and the fluxes $\vec F_{i}$
can be explicitly expressed by using the primitive quantities
${\bf W}:=(\rho,{\bm v}, p)^\top$
in the local rest frame.
However, unlike the non-relativistic case, for RHD there are no explicit
expressions for either the fluxes $\vec F_{i}$ or the primitive vector
${\bf W}$ in terms of the conservative variables ${\bf U}$.
This poses more additional challenges for the numerical simulations of the RHD than that for the non-relativistic case.
In practice, in order to evaluate the flux ${\bf F}_i({\bf U})$ in the computations, we have to first
recover the primitive quantities ${\bf W}$ from the
conservative vector ${\bf U}$ by performing the inverse transformation
of \eqref{eq:UU}, within every mesh cell and at each time step.
Given a conservative vector ${\bf U}=(D, {\bm m}, E)^\top$, we can get the values of the corresponding $\{p({\bf U}), {\bm v} ({\bf U}), \rho ({\bf U}) \}$ as follows:
first numerically solve a nonlinear algebraic equation \cite{WuTang2015}
\begin{equation}\label{eq:Eqforp}
	\Phi_{ \bf U}  ( p ) := \frac{p}{\Gamma - 1} - E + 
	\frac{ \| {\bm m} \|^2 }{E+p}
	+ D \sqrt{  1 - \frac{ \| {\bm m} \|^2 }{ (E+p)^2 }  } =0, \qquad p \in [0, + \infty),
\end{equation}
by utilizing a root-finding algorithm to obtain the pressure $p({\bf U})$; then calculate the velocity and rest-mass density by
\begin{equation}\label{eq:vUdU}
{\bm v} ({\bf U}) =
{ \bm m }/{\big(E + p ({\bf U})\big)}, \quad \rho( {\bf U} ) = D \sqrt{ 1 - \left \| {\bm v} ({\bf U}) \right \|^2 }.
\end{equation}
We denote the above recovery procedure by the operator ${\bm {\mathcal W}}: {\bf U}\rightarrow {\bf W}$, namely,
$${\bf W} = {\bm {\mathcal W}}({\bf U})
= \Big( \rho ({\bf U}), {\bm v}({\bf U}), p ({\bf U}) \Big)^\top.$$
Let $\vec{f}_i ({\bf W})$ denote the flux $\vec{F}_i$ that is expressed as a vector function of
the primitive variables ${\bf W}$.
Then the flux function $\vec{F}_i({\bf U})$ in terms of the conservative variables ${\bf U}$ can be
expressed by
\begin{equation*}
{\bf F}_i (\vec{U}) = \vec{f}_i ( {\bm {\mathcal W}}  (\vec{U}) ) = \vec{f}_i\circ {\bm {\mathcal W}}  (\vec{U}), \qquad 1\le i \le 3.
\end{equation*}

In the following, we shall restrict our attention to the special RHD system
in the two space dimensions:
\begin{equation}\label{eq:RHD}
	\frac{\partial \vec{U}}{\partial t} + \frac{\partial\vec F_{1}(\vec{U})}{\partial x} + \frac{\partial\vec F_{2}(\vec{U})}{\partial y} = {\bf 0},
\end{equation}
where $(x,y)$ represents the spatial coordinates; the conservative vector $\vec{U}$ and the fluxes $\vec F_{i}$ reduce to
	\begin{equation*}%\label{eq:UF2D}
	 {\bf U} = (D, m_1, m_2, E )^\top,\quad~
	 {\bf F}_1= (D v_1, m_1 v_1 + p, m_2 v_1, m_1)^\top,\quad~
	 {\bf F}_2= (D v_2, m_1 v_2 , m_2 v_2+p, m_2)^\top.
\end{equation*}
In physics, the rest-mass density $\rho$ and pressure $p$ should be positive,
and, as required by the relativistic causality, the fluid velocity magnitude $\| {\bm v} \|$ must not
exceed the speed of light $c$.
That is, the primitive vector must stay in the following set
\begin{equation}\label{Gw}
	G_w := \big \{\vec{W}=(\rho,{\bm v},p)^\top \in \mathbb{R}^{4}:~\rho>0,~p>0,~\|{\bm v}\|<c=1 \big\},
\end{equation}
where the speed of light in vacuum $c=1$ as we employed the geometrized unit system.  
Accordingly, the conservative vector $\vec{U}$ must satisfy the following constraints
	\begin{equation} \label{Gu}
	G_u := \Big\{\vec{U}=(D,{\bm m},E)^\top \in \mathbb{R}^{4}:~\rho(\vec{U})>0,~p(\vec{U})>0,~\|{\bm v}(\vec{U})\|<1 \Big\}.
\end{equation}
The functions $\rho(\vec{U}), p(\vec{U})$ and $\bm{v}(\vec{U})$ in \eqref{Gu} are highly nonlinear and have no explicit expressions, as defined by \eqref{eq:Eqforp} and \eqref{eq:vUdU}. This makes the studies on PCP numerical methods for RHD nontrivial.

We refer to $G_u$ as the physically admissible state set.
The following two properties of $G_{\bf u}$ were rigorously proven in \cite{WuTang2015}.

\begin{lemma}
The admissible state set $G_u$ is a convex set.
\end{lemma}

\begin{lemma}\label{lem:Gu1}
	The admissible state set $G_u$ is exactly equivalent to the following set
\begin{equation} \label{Gu1}
	G_u^{(1)} := \Big \{\vec{U}=(D,{\bm m},E)^\top \in \mathbb{R}^{4}:~D>0,~ g(\vec{U})>0 \Big \},
\end{equation}
where the function $g(\vec{U}):=E-\sqrt{D^2+\|{\bm m}\|^2}$ is a concave function. 
\end{lemma}

The satisfaction of constraints \eqref{Gu} is necessary, not only for physical significance, but also for
the hyperbolicity and well-posedness of the special RHD system \eqref{eq:RHD}.
In fact, as long as ${\bf U}\in G_u$, the system \eqref{eq:RHD}
is strictly hyperbolic. Let $\lambda_{\bf n}^{(k)}( {\bf U})$ and ${\bf R}_{\bf n} ({\bf U})$ be the eigenvalues and the corresponding right eigenmatrix of the rotated Jacobian matrix $\vec{A}_{\bf n}({\bf U})$,  which are given in detail in Appendix \ref{eigenstructure}.
Note that for any $\zeta >0$ and any ${\bf U} \in G_u$, the following homogeneous properties hold
\begin{equation}\label{eq:HOG}
	{\bf F}_i( \zeta {\bf U} ) = \zeta {\bf F}_i(  {\bf U} ), \quad 
	g( \zeta {\bf U} ) = \zeta g(  {\bf U} ),
\end{equation}
and the following properties hold
\begin{equation}\label{eq:DML}
\lambda_{\bf n}^{(k)}( \zeta {\bf U}) =  \lambda_{\bf n}^{(k)}( {\bf U}), \quad  {\bf R}_{\bf n} (\zeta  {\bf U}) = {\bf R}_{\bf n} ({\bf U}), \quad {\bf R}_{\bf n}^{-1} (\zeta {\bf U}) = {\bf R}_{\bf n}^{-1} ({\bf U}).
\end{equation}
Based on \eqref{eq:HOG}, we can show the following properties.

\begin{lemma}\label{lem:lamG}
If ${\bf U} \in G_u$, then for any constant $\zeta >0$, $\zeta {\bf U} \in G_u$.
\end{lemma}

\begin{proof}
This directly follows from Lemma \ref{lem:Gu1} and $g(\zeta {\bf U}) = \zeta g({\bf U})>0$.
\end{proof}

\begin{proposition}\label{prop:invar}
Let ${\mathcal S}_t$ denote the exact time evolution operator of 	the RHD system \eqref{eq:RHD}, i.e. the exact solution satisfies
$$
{\bf U}({\bm x},t) =  {\mathcal S}_t \big( {\bf U}({\bm x},0) \big). 
$$
Then for any constant $\zeta>0$, we have 
$$
{\bf U}({\bm x},t) = \frac{1}{\zeta} {\mathcal S}_t \big( \zeta {\bf U}({\bm x},0) \big). 
$$
This indicates the exact time evolution operator ${\mathcal S}_t$ is homogeneous.
\end{proposition}

\begin{proof}
	Based on ${\bf F}_i( \zeta {\bf U} ) = \zeta {\bf F}_i(  {\bf U} ) $ and Lemma \ref{lem:lamG}, 
	one can verify that $\zeta {\bf U}({\bm x},t)$ is the exact solution to 
	the RHD system \eqref{eq:RHD} with initial data $\zeta {\bf U}({\bm x},0)$. This completes the proof. 
\end{proof}

	\section{Numerical method}
	\label{Sec:Numerical_method}
	In this section, we present
	a high-order PCP finite volume method,
	which always keeps numerical solutions in the
	 admissible state set $G_u$,
	 for the 2D special RHD equations \eqref{eq:RHD} on unstructured triangular meshes.

Let ${\bm x}:=(x,y)$ represents the 2D spatial coordinates.
Assume that the 2D domain $\Omega$ is partitioned into triangular control volumes $\mathcal{T}_h$. For every cell $K\in \mathcal{T}_h$,
integrating the RHD system \eqref{eq:RHD} over $K$ and then using the divergence theorem gives
\begin{equation}\label{numerical1}
	\frac{\rm d}{{\rm d}t}\iint_{K} {\bf U} {\rm d} {\bm x} +\sum_{j=1}^{3} \int_{\mathcal{E}^j_K} {\bf n}^{(j)}_K \cdot { {\bf F}( {\bf U} )}  {\rm d} s = {\bf 0},
\end{equation}
where $\mathcal{E}^j_K, 1\le j \le 3$,
stand for the three edges of triangle $K$, the real vector ${\bf n}^{(j)}_K=(n_{K,1}^{(j)},n_{K,2}^{(j)})$ denotes the unit outward normal vector of edge $\mathcal{E}_K^j$, and ${\bf n}^{(j)}_K \cdot { {\bf F}( {\bf U} )} :=
n_{K,1}^{(j)} \vec{F}_1(\vec{U})  +  n_{K,2}^{(j)} \vec{F}_2(\vec{U})$.

\subsection{Outline and key ingredients of our method}
\label{sec:outline}

Let $\overline{\bf U}_K(t)$ denotes the numerical approximation to the cell-averaged solution $\frac{1}{|K|}\iint_{K} {\bf U}( {\bm x},t) {\rm d} {\bm x}$, where $|K|$ is the area of cell $K$.
From \eqref{numerical1}, one can obtain
a semi-discrete
finite volume method for the RHD system \eqref{eq:RHD}, in the following form
\begin{equation}\label{eq:FV1}
	\frac{\rm d}{ {\rm d} t}\overline{{\bf U}}_K =- \frac{1}{ |K| } \sum_{j=1}^{3} \int_{\mathcal{E}_K^j} \widehat{{\bf F}} \left(\vec{U}_h^{\text{int(K)}},\vec{U}_h^{\text{ext(K)}};{\bf n}^{(j)}_K \right) {\rm d} s,
\end{equation}
where $\widehat{{\bf F}}(\vec{U}_h^{\text{int(K)}},\vec{U}_h^{\text{ext(K)}};{\bf n}^{(j)}_K)$ denotes the numerical flux which approximates ${\bf n}^{(j)}_K \cdot { {\bf F}( {\bf U} )}$ and will be specified later, $\vec{U}_h({\bm x},t)$
is a suitable high-order numerical approximation to the exact solution $\vec{U}({\bm x},t)$, and the superscripts ``$\text{int(K)}$'' and ``$\text{ext(K)}$'' are the associated limits of  $\vec{U}_h$ at the cell interfaces which are taken from the interior and  exterior of $K$, respectively.

To achieve $(k+1)$th-order accuracy in space, we use a piecewise polynomial vector function $\vec{U}_h({\bm x},t) \in \mathbb{V}_h^k$ to approximate the exact solution $\vec{U}(\vec{x},t)$ for any fixed $t$, where
\[
\mathbb{V}_h^k :=\left \{\vec{u}=(u_1, \cdots, u_4)^\top:~ u_{l}|_K \in \mathbb{P}^k(K),~1\le l \le 4,\forall K\in\mathcal{T}_h \right \},
\]
and $\mathbb{P}^k(K)$ is the space of polynomials of total degree up to $k$ in  cell $K$.
In our finite volume method, the approximate solution function $\vec{U}_h({\bm x},t)$ is reconstructed from the cell averages $\left \{ \overline{{\bf U}}_K: K\in\mathcal{T}_h \right \}$.
Mathematically, the reconstruction procedure can be denoted by an  operator $\mathcal{R}_h^K:\mathbb{V}_h^0\rightarrow \mathbb{V}_h^k$, which maps the cell averages $\left \{ \overline{{\bf U}}_K: K\in\mathcal{T}_h \right \}$ to the piecewise polynomial solution ${\bf U}_h$.
For notational convenience, in the following we will temporarily suppress the $t$ dependence of all quantities, if no confusion arises.

	The edge integral in \eqref{eq:FV1} cannot be analytically evaluated in general
	and should be approximated by some 1D quadrature, for example, the $Q$-point Gauss quadrature with $2Q-1 \ge k$.
	Then the semi-discrete scheme \eqref{eq:FV1} becomes
	\begin{equation}\label{eq:FV}
		\frac{ {\rm d} \overline{{\bf U}}_K }{ {\rm d} t} =- \frac{1}{ |K| } \sum_{j=1}^{3}|\mathcal{E}_K^j|\sum_{q=1}^{\text{Q}} \omega_q \widehat{{\bf F}} \left(\vec{U}_h^{\text{int}(K)}({\bm x}_K^{(jq)}),\vec{U}_h^{\text{ext}(K)}({\bm x}_K^{(jq)});{\bf n}^{(j)}_K \right)=: \vec{L}_K(\vec{U}_h),
	\end{equation}
	where $|\mathcal{E}_K^j|$ represents the length of $\mathcal{E}_K^j$,
	$\{{\bm x}_K^{(jq)}\}_{1\leq q\leq Q}$ denote the Gauss quadrature points on $\mathcal{E}_K^j$, and $\{ \omega_q\}_{1\leq q\leq Q}$ are the associated weights with $\sum_{q=1}^{\text{Q}}\omega_q=1$. For a third-order accurate scheme with $k=2$,  we take $Q=2$ and $\omega_1=\omega_2=1/2$.

	In order to define the PCP finite volume schemes, we introduce the following two subsets of $\mathbb{V}_h^k$:
	\begin{align*}
		&\overline {\mathbb{G}}_h^k:= \left \{\vec{u}\in\mathbb{V}_h^k:~ 	 \overline{\vec{u}}_K= \frac{1}{|K|} \iint_K\vec{u}({\bm x}) {\rm d} {\bm x} \in G_u,\forall K \in\mathcal{T}_h \right \}, \\
		&	\mathbb{G}_h^k:= \left \{\vec{u}\in \overline {\mathbb{G}}_h^k:~ {\bf u}_K^{ (jq) } \in G_u, 1\le j \le 3, 1\le q \le Q,  \frac{ \overline{\vec{u}}_K-\frac23\widehat{\omega}_1\sum\limits_{j=1}^3\sum\limits_{q=1}^Q\omega_{q} {\bf u}_K^{ (jq) }  }{1-2\widehat{\omega}_1}  \in G_u,\forall K \in\mathcal{T}_h \right \},
	\end{align*}
	where
	$
	{\bf u}_K^{ (jq) } := \vec{u} \big|_K(  {\bm x}_K^{(jq)} )
	$, $\overline{\vec{u}}_K$ is the cell average of ${\bf u}$ over cell $K$, and
	 $\widehat{\omega}_1 = \frac{1}{ L ( L-1 ) }$ is the first weight of
	 the $L$-point Gauss-Lobatto quadrature with $L:=\left\lceil \frac{k+3}2 \right\rceil $.

Our high-order PCP finite volume schemes on unstructured triangular meshes are built on the following
four key ingredients (KI):
\begin{description}
	\item[(KI-1)] High-order reconstruction operator $\mathcal{R}_h^k:\mathbb{V}_h^0\rightarrow \mathbb{V}_h^k$. This operator represents the reconstruction procedure which constructs the piecewise polynomial solution ${\bf U}_h$
	from the cell averages $\left \{ \overline{{\bf U}}_K: K\in\mathcal{T}_h \right \}$.
	Several high-order reconstruction techniques were developed on unstructured triangular meshes in the literature, including but not limited to \cite{wang2017compact,zhu2018new}. %Barth1998,
	In our finite volume method, the reconstruction operator should keep the conservativeness:
\begin{equation}\label{eq:rec00}
	\frac{1}{|K|}\iint_K \mathcal{R}_h^k (\vec{u}){\rm d}{\bm x}=\frac{1}{|K|}\iint_K {\bf u} {\rm d} {\bm x}, \qquad \forall K\in \mathcal{T}_h,~ \forall \vec{u}\in \mathbb{V}_h^0,
\end{equation}
	 which yields
	 \begin{equation}\label{eq:rec1}
	 	{\bf U}_h \in \overline {\mathbb{G}}_h^k,  \quad \mbox{provided that } \overline{{\bf U}}_K \in G_u~~\forall K\in\mathcal{T}_h,
	 \end{equation}
 	namely, the operator range satisfies
 	\begin{equation}\label{eq:rec2}
 		\mathcal{R}_h^k (\overline {\mathbb{G}}_h^0) \subseteq \overline {\mathbb{G}}_h^k .
 	\end{equation}
	 Most of the existing reconstruction approaches satisfy \eqref{eq:rec00}, and in particular, we employ the simple high-order WENO reconstruction recently developed in \cite{zhu2018new}.
	 We will take the case $k=2$ as an example to illustrate the WENO reconstruction procedure $\mathcal{R}_h^k$
	 %More details of the reconstruction procedure $\mathcal{R}_h^k$ for $k=2$ will be presented
	 in Section \ref{sec:RCST}.
	\item[(KI-2)] High-order PCP limiting operator $\Pi_h: \overline {\mathbb{G}}_h^k \to \mathbb{G}_h^k$. This operator denotes a simple limiter,
	which maps the reconstructed numerical solution
	$\vec{U}_h\in \overline {\mathbb{G}}_h^k$ to $\widetilde{\vec{U}}_h:=\Pi_h(\vec{U}_h)\in {\mathbb{G}}_h^k$. The limiter also maintains the high-order accuracy and the conservativeness
	\[
	\frac{1}{|K|}\iint_K{\Pi}_h(\vec{u}) {\rm d}{\bm x}=\frac{1}{|K|}\iint_K\vec{u} {\rm d} {\bm x}, \qquad \forall K\in \mathcal{T}_h,~ \forall \vec{u}\in \overline{\mathbb{G}}_h^k.
	\]
	Clearly, the limited solution $\widetilde{\vec{U}}_h:=\Pi_h(\vec{U}_h)\in {\mathbb{G}}_h^k$ satisfies
	\begin{equation*}
	 \widetilde{\vec{U}}_{jq}^{{\rm int}(K)} :=	\widetilde{\vec{U}}_h^{\text{int}(K)}({\bm x}_K^{(jq)}) \in G_u, \quad \widetilde{\vec{U}}_{jq}^{{\rm ext}(K)} := \widetilde{\vec{U}}_h^{\text{ext}(K)}({\bm x}_K^{(jq)}) \in G_u,\quad \forall j,q,K,
	\end{equation*}
	which guarantee
	 the existence and uniqueness of the positive solution to the nonlinear equation \eqref{eq:Eqforp} and therefore, theoretically ensure the unique (physically admissible) primitive
	variables
	\begin{equation}
	 \widetilde{\vec{W}}_{jq}^{{\rm int}(K)} :=	{\bm {\mathcal W}} \left( \widetilde{\vec{U}}_{jq}^{{\rm int}(K)} \right) \in G_w, \qquad
	 \widetilde{\vec{W}}_{jq}^{{\rm ext}(K)} :=	{\bm {\mathcal W}} \left( \widetilde{\vec{U}}_{jq}^{{\rm ext}(K)} \right) \in G_w.
	\end{equation}
	More details of the PCP limiter $\Pi_h$ will be presented in Section \ref{sec:constraint-preserving}.
	\item[(KI-3)] Convergence-guaranteed algorithms for recovery of primitive variables.
	Although the existence and uniqueness of the
	positive solution to the nonlinear equation \eqref{eq:Eqforp} are ensured in theory by
	the PCP limiting procedure in {\bf (KI-2)}, some root-finding algorithms, such as Newton's method, may still fail to get the unique positive solution of \eqref{eq:Eqforp}.
	 In Section \ref{sec:recovery}, we will present three effective algorithms for solving the
	  nonlinear equation \eqref{eq:Eqforp}. We will rigorously prove the proposed algorithms provably guarantee the convergence in recovering the unique primitive
	  variables ${\bf W} \in G_w$ from any given admissible ${\bf U} \in G_u$.
	\item[(KI-4)] The PCP numerical flux and provable PCP property.
	In order to preserve the cell averages $\overline{{\bf U}}_K \in G_u$ during the time evolution of the fully discrete scheme, we need to seek a spatial discretization operator $\vec{L}_K$ such that the following PCP property
	\begin{equation}\label{eq:constraint-preservingproperty}
			\overline{\bf U}_K + \Delta t \vec{L}_K( \widetilde {\bf U}_h)  \in G_u~~\forall K\in\mathcal{T}_h,  \quad \mbox{provided that } \widetilde {\bf U}_h \in \mathbb{G}_h^k,
	\end{equation}
	holds under some CFL condition on $\Delta t$, where the requirement $\widetilde {\bf U}_h = \Pi_h(\vec{U}_h) \in \mathbb{G}_h^k$ is ensured by
	the PCP limiter in {\bf (KI-2)}.
	We achieve the PCP property \eqref{eq:constraint-preservingproperty} by adopting the HLL numerical flux
\begin{equation} \label{eq:HLL}
	\widehat{\bf F}^{hll} ( {\bf U}^-, {\bf U}^+; {\bf n} )
	= \frac{ \sigma_r {\bf n} \cdot {\bf F}({\bf U}^-) 
- \sigma_l  {\bf n}  \cdot	{\bf F}({\bf U}^+) + \sigma_l \sigma_r ( {\bf U}^+- {\bf U}^- )
  } { \sigma_r - \sigma_l }
\end{equation}
with
\begin{align*}
&	\sigma_l ({\bf U}^-, {\bf U}^+; {\bf n}) := \min \left\{
\lambda_{\bf n}^{(1)}( {\bf U}^- ), \lambda_{\bf n}^{(1)}({\bf U}^+), 0
 \right\},
 \\
&  \sigma_r ({\bf U}^-, {\bf U}^+; {\bf n}) := \max \left\{
\lambda_{\bf n}^{(4)}( {\bf U}^- ), \lambda_{\bf n}^{(4)}({\bf U}^+), 0
\right\}.
\end{align*}
	In Theorem \ref{THCPP} of Section \ref{sec:CPP}, we will rigorously prove:
	 if the numerical flux $\widehat{{\bf F}}$ in \eqref{eq:FV} chosen as the above HLL flux
	$\widehat{\bf F}^{hll}$,
	 then the resulting spatial discretization operator
	 $
	 \vec{L}_K
	 $
	 satisfies the desired PCP property \eqref{eq:constraint-preservingproperty}.
	Some other numerical fluxes such as the Lax-Friedrichs flux and the HLLC flux  also meet the PCP property \eqref{eq:constraint-preservingproperty} in this framework.
\end{description}

\begin{remark}
To compute
$\widehat{\bf F}^{hll} \left( \widetilde{\vec{U}}_{jq}^{{\rm int}(K)}, \widetilde{\vec{U}}_{jq}^{{\rm ext}(K)};{\bf n}^{(j)}_K \right) $ for $\vec{L}_K( \widetilde {\bf U}_h)$ in \eqref{eq:constraint-preservingproperty},
we need to evaluate
\begin{equation}\label{eq:1112}
	{\bf n}_K^{(j)} \cdot {\bf F} \left( \widetilde{\vec{U}}_{jq}^{{\rm int}(K)} \right)   = \sum_{i=1}^2 n_{K,i}^{(j)} ~ \vec{F}_i \left( \widetilde{\vec{U}}_{jq}^{{\rm int}(K)} \right),\quad~ {\bf n}_K^{(j)} \cdot {\bf F} \left( \widetilde{\vec{U}}_{jq}^{{\rm ext}(K)} \right)  = \sum_{i=1}^2 n_{K,i}^{(j)} ~ \vec{F}_i \left( \widetilde{\vec{U}}_{jq}^{{\rm ext}(K)} \right).
\end{equation}
As mentioned in Section \ref{sec:GovenEqn}, unlike the non-relativistic case, for RHD there are no explicit
expressions for $\vec F_{i} ({\bf U})$.
To evaluate the flux ${\bf F}_i({\bf U})$ in \eqref{eq:1112}, we have to first
recover the primitive quantities $\widetilde{\vec{W}}_{jq}^{{\rm int}(K)}$ and $\widetilde{\vec{W}}_{jq}^{{\rm ext}(K)}$ from the
conservative vectors $\widetilde{\vec{U}}_{jq}^{{\rm int}(K)}$ and  $\widetilde{\vec{U}}_{jq}^{{\rm ext}(K)}$, respectively.
The recovery is theoretically ensured by {\bf (KI-2)},
and three provably convergent algorithms for practical recovery are discussed in {\bf (KI-3)} and Section \ref{sec:recovery}.
\end{remark}

Assume that the time interval is partitioned into a mesh $\{t_0=0,t_{n+1}=t_n+\Delta t_n, 0\leq n<N_t\}$ with the time step-size $\Delta t$ determined by some CFL condition.
Let $\overline {\bf U}_K^n$ denote the numerical approximation to the cell-averaged
solution on cell $K$ at $t=t_n$.
Let $\overline {\bf U}^n({\bm x})$ denote the piece-wise constant function defined by the cell averages $\{ \overline {\bf U}_K^n: K \in {\mathcal T}_h \}$.

Based on the above four key ingredients {\bf (KI-1)}--{\bf (KI-4)} and the forward Euler time discretization, we obtain
a fully discrete PCP finite volume method:
\begin{equation}\label{eq:1stconstraint-preservingFV}
\overline {\bf U}^{n+1}_K = \overline {\bf U}^{n}_K + \Delta t_n
{\bf L}_K \big( \widetilde {\bf U}_h \big) = \overline {\bf U}^{n}_K + \Delta t_n
{\bf L}_K \left( \Pi_h \mathcal{R}_h^k  \overline {\bf U}^n \right)
\end{equation}
with the initial cell averages given by
$$
\overline {\bf U}^{0}_K = \frac{1}{|K|} \iint_K {\bf U} ( {\bm x}, 0 ) {\rm d} {\bm x} \in G_u \qquad \forall K \in {\mathcal T}_h.
$$
The PCP property of the scheme \eqref{eq:1stconstraint-preservingFV} can be easily verified by induction as follows:
\begin{itemize}
	\item Thanks to the convexity of $G_u$, one has $\overline {\bf U}^{0}_K \in G_u$.
	\item Given $\overline {\bf U}^{n}_K \in G_u$, we have
	$ \overline {\bf U}^n \in   \overline {\mathbb{G}}_h^0 $,
	which implies $\mathcal{R}_h^k  \overline {\bf U}^n \in \overline {\mathbb{G}}_h^k$ by \eqref{eq:rec2}.
	Thus $\Pi_h \mathcal{R}_h^k  \overline {\bf U}^n \in  {\mathbb{G}}_h^k$ by
	the PCP limiter. This ensures $\overline {\bf U}^{n+1}_K \in G_u$
	by \eqref{eq:constraint-preservingproperty}.
\end{itemize}
For clarification, we draw the flowchart of the PCP method \eqref{eq:1stconstraint-preservingFV} in Fig. \ref{fig:flowmap}.
\begin{figure}[tbhp]
	\centering
	\includegraphics[width=0.7\textwidth]{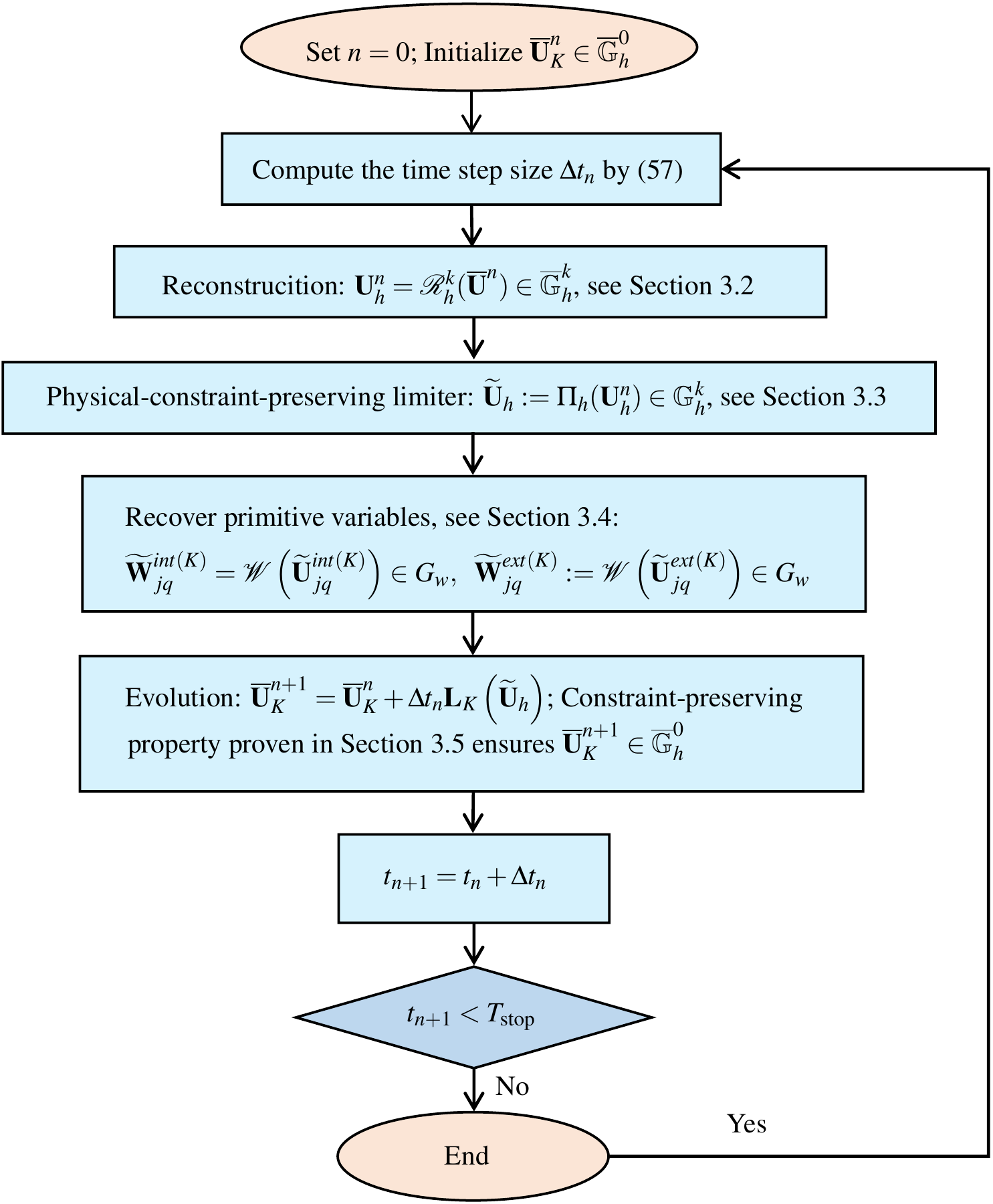}
	\caption{The flowchart of the proposed method. }
	\label{fig:flowmap}
\end{figure}

The PCP scheme is only first-order accurate in time.
	To achieve high-order accuracy in time, we can use the strong-stability-preserving (SSP) high-order methods \cite{GottliebShuTadmor2001}. 
	Since an SSP method is formally
	a convex combination of
	the forward Euler method, the PCP property remains valid due to the convexity of $G_u$.
	For example, when the third-order accurate SSP Runge-Kutta (SSP-RK) method is adopted,
	we obtain the following third-order accurate, fully discrete, PCP finite volume method:
	\begin{equation}
		\begin{cases}
		\displaystyle	\overline{\vec{U}}^{*}_K = \overline{\vec{U}}^{n}_K + \Delta t_n \vec{L}_K \left( \Pi_h \mathcal{R}_h^k \overline {\bf U}^n \right),\\[2mm]
		\displaystyle	\overline{\vec{U}}^{**}_K = \frac{3}{4}\overline{\vec{U}}^{n}_K +\frac14\left(\overline{\vec{U}}^{*}_K + \Delta t_n \vec{L}_K \left( \Pi_h \mathcal{R}_h^k \overline {\bf U}^* \right)  \right), \\[2mm]
		\displaystyle	\overline{\vec{U}}^{n+1}_K = \frac{1}{3}\overline{\vec{U}}^{n}_K +\frac23\left(\overline{\vec{U}}^{**}_K + \Delta t_n \vec{L}_K \left( \Pi_h \mathcal{R}_h^k \overline {\bf U}^{**} \right)  \right).
		\end{cases}
	\end{equation}

\begin{remark} The  HLL numerical flux $\widehat{\bf F}^{hll}$ not only meets the PCP requirement, but also has a homogeneous property as the flux function ${\bf F}({\bf U})$ in \eqref{eq:HOG}. This leads to 
the homogeneity of  the spacial discretization operator $\vec{L}_K$, namely,   
%\begin{lemma} \label{lem:fluxscal}
%	For any constant $$,
	$\vec{L}_K(\zeta\vec{U}_h)=\zeta\vec{L}_K(\vec{U}_h), \forall \zeta>0,$ as it will be shown in the proof of Theorem \ref{thrm:invarnumeric}. 
%\end{lemma} 
%Along with the results of preserving scaling invariance for  high-order reconstruction operator $\mathcal{R}^k_h$ given in Lemma \ref{lem:charactweno} in Subsection \ref{sec:RCST}, we conclude that our high-order numerical method preserves the scaling invariance property of RHD system.
If the high-order reconstruction operator $\mathcal{R}^k_h$ and the PCP limiting operator $\Pi_h$ are both 
homogeneous, namely, 
\begin{equation}\label{eq:R-PI-HOM}
	\mathcal{R}^k_h ( \zeta \overline {\bf U} ) = \zeta \mathcal{R}^k_h ( \overline {\bf U} ), \qquad \Pi_h ( \zeta \overline {\bf U} ) = \zeta \Pi_h ( \overline {\bf U} ), 
\end{equation} 
then our numerical method preserves the scaling invariance property of the RHD system; see 
Theorem \ref{thrm:invarnumeric}. 
We remark that some existing WENO reconstructions may not satisfy the above homogeneity, due to the loss of scaling invariance of the nonlinear weights; see Remark \ref{remark:weight}. 
We will show our (slightly modified) WENO reconstruction operator  $\mathcal{R}^k_h$ and the PCP limiting operator $\Pi_h$ satisfy \eqref{eq:R-PI-HOM}; see Lemma \ref{lem:charactweno}, identity \eqref{eq:PiHOM}, and Theorem \ref{thrm:invarnumeric}. 

\end{remark}

	\subsection{High-order reconstruction operator $\mathcal{R}_h^k$} \label{sec:RCST}
	This section introduces a reconstruction procedure $\mathcal{R}_h^k:\mathbb{V}_h^0\rightarrow \mathbb{V}_h^k$ on
	unstructured triangular meshes, which constructs the piecewise polynomial solution ${\bf U}_h$
	from the cell averages $\left \{ \overline{{\bf U}}_K: K\in\mathcal{T}_h \right \}$.
	 In fact, one can use any proper high-order reconstruction techniques, and the PCP property is not affected by the chosen reconstruction approach. In this paper, we employ the
	 new high-order WENO reconstruction recently developed by Zhu and Qiu  \cite{zhu2018new},
	 because it is highly compact and efficient as its linear weights can be chosen rather arbitrarily provided that their summation equals one. 
	 We observe that the nonlinear weights in \cite{zhu2018new} do not satisfy certain scaling invariance property, so that the resulting WENO reconstruction operator is generally not homogeneous. A (slight) modification to the nonlinear weights will be proposed to address this.

\subsubsection{Review of a WENO reconstruction for scalar problems}

We first briefly review the WENO reconstruction in \cite{zhu2018new} for scalar functions on triangular meshes.  We take the case $k=2$ as an example to illustrate the third-order accurate WENO reconstruction procedure.
Given the cell averages $\overline u = \{ \overline u_K: K \in {\mathcal T}_h \}$ of a scalar,
we reconstruct a quadratic polynomial as approximation on an arbitrary target cell $K_0 \in {\mathcal T}_h$:
\begin{equation}\label{WENO-scalar}
	\mathcal P_{K_0} \overline u := \varpi_0  \left(\frac{1}{\gamma_0} \phi_2 (x,y) - \sum_{\ell=1}^4\frac{\gamma_\ell }{\gamma_0} \phi_1^{(\ell)} (x,y)\right) + \sum_{\ell=1}^4 \varpi_\ell \phi_1^{(\ell)} (x,y),
\end{equation}
where
\begin{itemize}
	\item $\phi_2 (x,y)$ is a quadratic polynomial, and $\{ \phi_1^{(\ell)} (x,y), 1\le \ell \le 4 \}$ are four linear polynomials; they are reconstructed from the cell averages $\{ \overline u_K\}$ and satisfy
	\begin{equation}\label{key5643}
		\frac{1}{|K_0|} \iint_{K_0} \phi_2 (x,y) {\rm d} x {\rm d} y = \overline u_{K_0}, \qquad
		\frac{1}{|K_0|} \iint_{K_0} \phi_1^{(\ell)} (x,y) {\rm d} x {\rm d} y = \overline u_{K_0}, \quad 1\le \ell \le 4;
	\end{equation}
	\item $\{ \gamma_\ell, 0 \le \ell \le 4 \}$ are the linear weights, which are all positive, and their summation equals one;
	\item $\{ \varpi_\ell, 0 \le \ell \le 4 \}$ are the nonlinear weights for suppressing potential nonphysical oscillations in discontinuous problems.
\end{itemize}

First, we construct $\phi_2 (x,y)$.  Taking the requirement \eqref{key5643} into account, we express $\phi_2 (x,y)$ as
\begin{equation}\label{key2678}
	\phi_2(x,y) = \overline u_{K_0} +  \sum_{j=1}^5 a_j \psi_j ( x,y )
\end{equation}
with
\begin{align*}
	&  \psi_1 =  \frac{x - x_0}{\sqrt{|K_0|}},\quad \psi_2 = \frac{y - y_0}{\sqrt{|K_0|}},\quad
	\psi_3 = \frac{(x - x_0)(y-y_0)}{|K_0|} - \frac{1}{|K_0|^2} \iint_{K_0} (x - x_0)(y-y_0)  {\rm d} x {\rm d}y,
	\\
	& \psi_4 = \frac{(x - x_0)^2}{|K_0|} - \frac{1}{|K_0|^2} \iint_{K_0} (x - x_0)^2 {\rm d} x {\rm d}y, \quad
	\psi_5 = \frac{(y - y_0)^2}{|K_0|} -  \frac{1}{|K_0|^2} \iint_{K_0} (y - x_0)^2 {\rm d} x {\rm d}y,
\end{align*}
where $(x_0,y_0)$ denotes the barycenter of the target cell $K_0$. The coefficients ${\bm a}=(a_1,a_2,\dots, a_5)$ in \eqref{key2678} are determined by solving the least-squares problem
$$
\min_{ {\bm a} \in \mathbb R^5 } \sum_{ K \in \mathbb N_0 } \left| \frac{1}{ |K| } \iint_K \phi_2(x,y) {\rm d} x {\rm d} y - \overline u_{K} \right|^2
$$
where $\mathbb{N}_0 :=\{K_0, K_1, K_2, K_3, K_{11}, K_{12}, K_{21}, K_{22}, K_{31}, K_{32} \}$ denotes
the large stencil for the target cell $K_0$ as shown in Fig.~ \ref{fig:stencil}.

\begin{figure}
	\centering
	\includegraphics[width=0.99\textwidth]{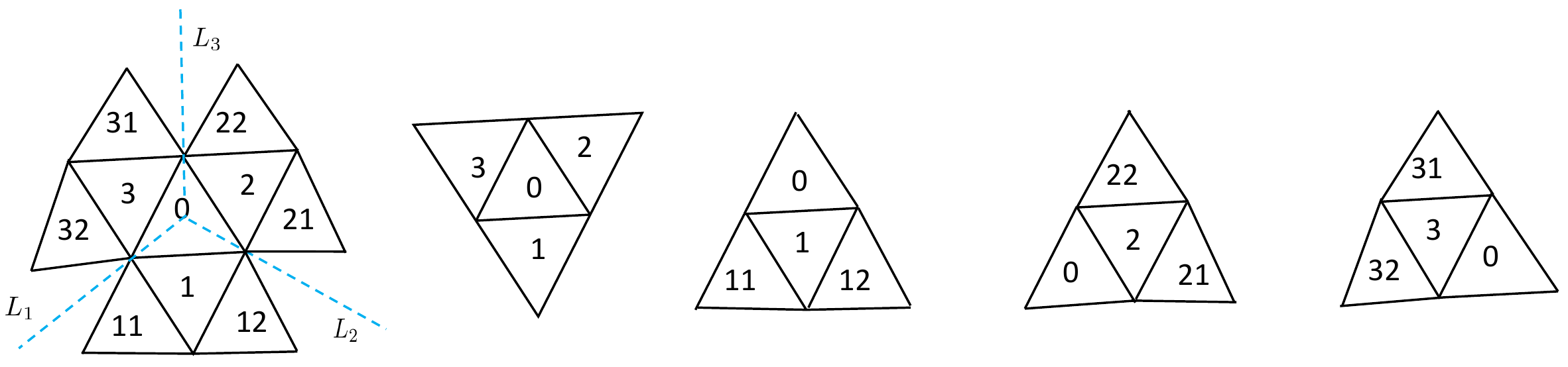}
	\caption{Reconstruction stencils $\mathbb{N}_0$ and $\mathbb{N}_1, \mathbb{N}_2, \mathbb{N}_3, \mathbb{N}_4$ (from left to right). }
	\label{fig:stencil}
\end{figure}

Next, we construct $\phi_1^{(\ell)} (x,y)$, $1\le \ell \le 4$.
With the requirement \eqref{key5643} taken into account, we express $\phi_1^{(\ell)} (x,y)$ as
\begin{equation}\label{key4333}
	\phi_1^{(\ell)}(x,y) = \overline u_{K_0} +  \sum_{j=1}^2 b_j^{ (\ell) } \psi_j ( x,y ),
\end{equation}
where the coefficients ${\bm b}^{(\ell)} = ( b_1^{ (\ell) }, b_2^{ (\ell) } )$ are determined by solving the least-squares problems
$$
\min_{ {\bm b}^{(\ell)} \in \mathbb R^2 } \sum_{ K \in \mathbb N_\ell } \left| \frac{1}{ |K| } \iint_K \phi_1^{ (\ell) }(x,y) {\rm d} x {\rm d} y - \overline u_{K} \right|^2, \qquad 1\le \ell \le 4,
$$
where
%$\mathbb{N}_1:=\{K_0,K_1,K_2,K_3\}$,  $\mathbb{N}_2:=\{K_0,K_1,K_{11},K_{12} \}$,  $\mathbb{N}_3:=\{K_0,K_2,K_{21},K_{22} \}$, and $\mathbb{N}_4 :=\{K_0,K_3,K_{31},K_{32}\}$
\begin{equation*}
	\mathbb{N}_1:=\{K_0,K_1,K_2,K_3\},  \mathbb{N}_2:=\{K_0,K_1,K_{11},K_{12} \},  \mathbb{N}_3:=\{K_0,K_2,K_{21},K_{22} \},
	\mathbb{N}_4 :=\{K_0,K_3,K_{31},K_{32}\}
\end{equation*}
are four small stencils as shown in Fig.~\ref{fig:stencil}. Specifically,
$\mathbb{N}_1$ is called the central stencil; $\mathbb{N}_2$, $\mathbb{N}_3$, and $\mathbb{N}_4$ are three sectorial stencils, each of which consists of the target cell $K_0$ and its neighboring cells whose barycenters lie in the same sector. As shown in
Fig.~\ref{fig:stencil},
 the three sectors are divided by three lines $L_1$, $L_2$, and $L_3$ connecting the centroid and three vertices of $K_0$.

Now we compute the linear weights $\gamma_\ell$ and nonlinear weights $\varpi_\ell$ in \eqref{WENO-scalar}. Following \cite{zhu2018new}, we set the linear weights as $\gamma_0=0.96, \gamma_1=\gamma_2=\gamma_3=\gamma_4=0.01$ in our computations.
However, slightly different from \cite{zhu2018new}, we take the nonlinear weights as 
\begin{equation}\label{OUR-weights}
	\varpi_\ell = \frac{\delta_\ell }{\sum_{i=0}^4\delta_i}, \qquad \delta_\ell = \gamma_\ell \left(1+\frac{\tau^2}{(\beta_\ell+\epsilon)^2}\right),\qquad  \quad \ell=0,\cdots,4,
\end{equation}
where $\epsilon$ is a small positive number used to avoid the denominator being zero. The quantity     
 $\tau$ is defined as 
\begin{equation}\label{tau}
	\tau = \frac14(|\beta_0-\beta_1|+ |\beta_0-\beta_2| + |\beta_0-\beta_3| + |\beta_0-\beta_4|) ={\mathcal O}(|K|^{\frac32}),
\end{equation}
where $\tau ={\mathcal O}(|K|^{\frac32})$ follows from \cite[Page A908, Eq.~(2.13)]{zhu2018new}, 
and $\{\beta_i, 0 \le i \le 4\}$ are the smooth indicators which are defined in a classic way as in
\cite{hu1999}:
\begin{equation*}
\beta_0= \sum_{| {\bm \alpha} |=1}^2\iint_{K}|K|^{| {\bm \alpha} |-1} \left({\mathcal D}^{ {\bm \alpha} } \phi_2(x,y) \right)^2 {\rm d} x {\rm d} y, \quad
\beta_\ell = \sum_{| {\bm \alpha}|=1}\iint_{K} \left( {\mathcal D}^{\bm \alpha} \phi_1^{(\ell)} (x,y) \right)^2 {\rm d} x {\rm d} y,~1\le \ell \le 4,
\end{equation*}
where $\bm \alpha$ is a multi-index and $\mathcal D$ is the partial derivative operator; for example, when ${\bm \alpha}=(1,1)$, then $|{\bm \alpha}|=2$ and ${\mathcal D}^{\bm \alpha} \phi(x,y)=\partial^2 \phi/\partial x\partial y$.
In this paper, we take $\epsilon = |K_0|\times \max_{K \in \mathbb N_0}\{ |\overline u_K|^2 \}$ in \eqref{OUR-weights} to make the nonlinear weights $\varpi_\ell$ scaling-invariant (see Lemma \ref{lem:Pk}), so as to achieve  
the homogeneity of $\mathcal{R}_h^k$ in \eqref{eq:R-PI-HOM} (see Lemma \ref{lem:charactweno}). 
Note the nonlinear weights in \cite{zhu2018new} are not scaling-invariant (see Remark \ref{remark:weight}).

\begin{lemma}\label{lem:Pk}
	The nonlinear weights \eqref{OUR-weights} are scaling-invariant, namely, 
	for any given constant $\zeta>0$, if we scale the cell averages $\{ \overline u_K: K \in  \mathbb N_0 \}$ to $\{ \zeta \overline u_K: K \in  \mathbb N_0 \}$, then the corresponding weights $\varpi_\ell$ remain unchanged.   
	Consequently, 
	the operator $\mathcal P_{K_0}$ is homogeneous, namely, for any constant $\zeta>0$, 
	$\mathcal P_{K_0} ( \zeta\overline u ) = \zeta \mathcal P_{K_0} \overline u$.
\end{lemma}

\begin{proof}
Let $\phi_2(x,y;\zeta)$ and $\phi^{(\ell)}_1(x,y;\zeta), 1\le \ell \le4$ be the quadratic polynomial and four linear polynomials reconstructed from the scaled cell averages $\{ \zeta \overline u_K: K \in  \mathbb N_0 \}$, and $\beta_\ell^{ (\zeta)}$ be the corresponding smooth indicators.
Denote $\tau^{(\zeta)},\delta_\ell^{(\zeta)}$, and $\varpi_\ell^{(\zeta)}$ be the values calculated from $\beta_\ell^{(\zeta)}$ by using the equations \eqref{OUR-weights} and \eqref{tau}.  Observing that $\phi_2(x,y;\zeta) = \zeta \phi_2(x,y)$ and $\phi^{(\ell)}_1(x,y;\zeta) = \zeta \phi^{(\ell)}_1(x,y)$, we obtain that $\beta_\ell^{(\zeta)} = \zeta^2 \beta_\ell, 0\le \ell \le4$ and $\tau^{(\zeta)} = \zeta^2 \tau$. Thanks to 
$\epsilon^{(\zeta)}:= |K_0|\times \max_{K \in \mathbb N_0}\{ |\zeta \overline  u_K|^2 \}=\zeta^2 \epsilon,$  
we then have 
$$\delta^{(\zeta)}_\ell = \delta_\ell, \qquad \varpi^{(\zeta)}_\ell = \varpi_\ell.$$ 
This means the nonlinear weights \eqref{OUR-weights} are scaling-invariant.
It follows that $\mathcal P_{K_0} ( \zeta\overline u ) = \zeta \mathcal P_{K_0} \overline u$. 
\end{proof}

\begin{remark}\label{remark:weight}
It should be explained why we prefer to use the nonlinear weights \eqref{OUR-weights} different from that in \cite{zhu2018new}. 
With our above notations, the nonlinear weights used in \cite{zhu2018new} 
can be rewritten as
\begin{equation}\label{ZQweights}
	\widetilde \varpi_\ell = \frac{ \widetilde \delta_\ell }{\sum_{i=0}^4  \widetilde \delta_i}, \qquad \widetilde \delta_\ell = \gamma_\ell \left(1+\frac{\tau^2}{\beta_\ell+\epsilon}\right),\qquad   \ell=0,\cdots,4.
\end{equation}
Notice that our notation $\tau^2$ corresponds to the notation $\tau$ in \cite{zhu2018new}. The nonlinear weights \eqref{ZQweights} work well for many benchmark problems in \cite{zhu2018new}. However, the weights \eqref{ZQweights} are not scaling-invariant, even if we set $\epsilon = |K_0|\times \max_{K \in \mathbb N_0}\{ |\overline u_K|^2 \}$. 
This is because $(\tau^{(\zeta)} )^2 = \zeta^4\tau^2$ and $\beta^{(\zeta)}_\ell+\epsilon^{(\zeta)}=\zeta^2(\beta_\ell+\epsilon)$, so  
that $\tau^2/(\beta_\ell + \epsilon) + 1$ is ``not dimensionless''. 
It seems important to accommodate the scaling invariance, as the quantity $u$ may have very different scales/values for 
different characteristic variables and different problems. 
Because the nonlinear weights are used for suppressing potential numerical oscillations, they should be dimensionless 
and independent of the solution scales.  
Therefore, the scaling-invariant property and the resulting homogeneity of $\mathcal P_{K_0}$ are desirable and may be helpful for resolving multi-scale flow structures and suppressing nonphysical oscillations. 
This observation will be further confirmed by
numerical results in Examples \ref{1Driemann3} and \ref{1DMSP} of Section \ref{sec:examples}, where 
the two sets of nonlinear weights \eqref{OUR-weights} and \eqref{ZQweights} will be compared. 
Although our modification is proposed on unstructured meshes, it also applies to structured meshes.

\end{remark}

It is shown in \cite{zhu2018new} that $\beta_i= {\mathcal O} (|K|), i=0,\cdots,4$. It follows that $\frac{\tau^2}{(\beta_i+\epsilon)^2}= {\mathcal O} (|K|), i=0,\cdots,4$. %on the condition that $\epsilon\ll\beta_i$.
 Therefore,  $\delta_i=\gamma_i(1+ {\mathcal O} (|K|))$ which gives $\varpi_i=\gamma_i(1+ {\mathcal O} (|K|))$. As a result, the accuracy of the above WENO reconstruction is third-order as expected. This verifies that our nonlinear weights \eqref{OUR-weights} also meet the accuracy requirement.

\subsubsection{Characteristic WENO reconstruction for RHD system}
One can apply the above WENO reconstruction to the RHD equations \eqref{eq:RHD}
either component-wisely or in local characteristic directions.
It has been widely realized that characteristic reconstruction
usually produces better nonoscillatory results for high-order schemes.
Therefore, we impose the WENO reconstruction on the local characteristic variables for the RHD system.

Assume that $\overline {\bf U}_K \in G_u$ for all $K \in {\mathcal T}_h$.
Then, by the algorithms that will be introduced in Section \ref{sec:recovery},
we can uniquely recover the corresponding primitive variables
\begin{equation}
	\overline {\bf W}_K =	{\bm {\mathcal W}} \left( \overline {\bf U}_K \right) \in G_w \quad \forall K \in {\mathcal T}_h.
\end{equation}
We would like to reconstruct, for every $K \in {\mathcal T}_h$, a polynomial vector function ${\bf P}_K({\bm x})$
satisfying
\begin{equation}\label{key134}
	\frac{1}{|K|}\iint_K {\bf P}_K({\bm x}) {\rm d} {\bm x} = \overline {\bf U}_K,
\end{equation}
so as to obtain
\begin{equation} \label{eq:RC113}
	{\bf U}_h ({\bm x}) = \mathcal{R}_h^k \overline {\bf U} := \sum_{ K \in {\mathcal T}_h }  {\bf P}_K({\bm x})  \chi_K({\bm x}), \qquad
	\chi_K({\bm x}) = \begin{cases}
		1,~~ & {\bm x} \in K,
		\\
		0,~~ & {\bm x} \notin K.
	\end{cases}
\end{equation}
The property \eqref{key134} implies that the reconstructed piecewise polynomial vector function ${\bf U}_h \in \overline {\mathbb{G}}_h^k$.

Based on local characteristic decomposition, we reconstruct
the polynomial vector ${\bf P}_{K_0}({\bm x})$ in \eqref{eq:RC113}
 for an arbitrary target cell $K_0 \in {\mathcal T}_h$ as follows:

\begin{description}
	\item[Step 1] For each normal direction ${\bf n}_j:=  {\bf n}^{(j)}_{K_0}$ of $K_0$, $j \in \{1,2,3\}$, do the following:
\end{description}
\begin{itemize}
	\item  Compute the local eigenvector matrix in the direction ${\bf n}_j$, i.e.,  ${\bf R}_{{\bf n}_j} (\overline{\bf U}_{K_0})$ and ${\bf R}^{-1}_{{\bf n}_j} (\overline{\bf U}_{K_0})$ according to the formulas in Proposition \ref{prop111}. 
	For the RHD system, ${\bf R}_{{\bf n}_j}$ and ${\bf R}^{-1}_{{\bf n}_j}$ cannot be explicitly expressed by the conservative variables, therefore, we have to first recover the primitive vector $\overline{\bf W}_{K_0}$ and then use the primitive variables to evaluate ${\bf R}_{{\bf n}_j}$ and ${\bf R}^{-1}_{{\bf n}_j}$.
	\item Project the cell averages $\{\overline {\bf U}_s: s \in \mathbb N_0 \}$
	into the local characteristic fields
	$$
	\overline {\bf Z}_s^{(j)} := {\bf R}^{-1}_{{\bf n}_j}  \overline {\bf U}_s \quad \forall s \in \mathbb N_0.
	$$
	\item Perform the scalar WENO reconstruction procedure $\mathcal P_{K_0}$, defined in
	 \eqref{WENO-scalar}, component-wisely to the cell averages $ \overline {\bf Z}^{(j)} := \{
	 \overline {\bf Z}_s^{(j)}: s \in \mathbb N_0
	 \}$ and obtain the polynomial approximation of the characteristic variables
	 $$
	 {\bf Z}_{K_0} ( {\bm x} ) = \mathcal P_{K_0} \overline {\bf Z}^{(j)}.
	 $$
	\item Project the polynomial vectors ${\bf Z}_{K_0}$ into the physical space of conservative variables
	$$
	{\bf P}_{K_0}^{(j)} ({\bm x}) = {\bf R}_{{\bf n}_j} {\bf Z}_{K_0} ( {\bm x} ).
	$$
\end{itemize}

\begin{description}
	\item[Step 2]
	The final reconstructed polynomial vector on the target cell $K_0$ is obtained by taking a weighted average of $\{ {\bf P}_{K_0}^{(j)}, 1\le j \le 3 \}$, i.e.
\begin{equation}\label{keyPPP}
		{\bf P}_{K_0} ({\bm x}) = \frac{ \sum_{j=1}^3 |K_j| {\bf P}_{K_0}^{(j)} ({\bm x}) }{ \sum_{j=1}^3 |K_j| }.
\end{equation}
\end{description}
One can verify that the reconstructed polynomial vector \eqref{keyPPP} satisfies \eqref{key134}.

Thanks to \eqref{eq:DML} and the homogeneity of the operator $\mathcal{P}_{K_0}$ proven in Lemma \ref{lem:Pk}, 
we immediately obtain that the characteristic WENO reconstruction operator $\mathcal{R}_h^k$ is also homogeneous.  
\begin{lemma} \label{lem:charactweno}
     For any constant $\zeta>0$, $ \mathcal{R}_h^k(\zeta  \bf \overline U) = \zeta \mathcal{R}_h^k(  \bf \overline U)$.
\end{lemma}

	\subsection{Physical-constraint-preserving limiting operator $\Pi_h$}\label{sec:constraint-preserving}
Now we detail the 
operator $\Pi_h: \overline{\mathbb{G}}_h^k \to \mathbb{G}_h^k$.  
Let ${\bf U}_h = \sum_{ K \in {\mathcal T}_h }  {\bf P}_K({\bm x})  \chi_K({\bm x}) \in  \overline{\mathbb{G}}_h^k$ denote the  
reconstructed WENO solution with ${\bf P}_K({\bm x})=:( D_K({\bm x}), {\bm m}_K( {\bm x} ), E_K( {\bm x} ) )^\top$. 
Define $\overline{\bf U}_K =: ( \overline D_K, \overline {\bm m}_K, \overline E_K)^\top$. 
We denote the PCP limited solution by 
\begin{equation}\label{constraint-preservinglimited}
	\widetilde {\bf U}_h = {\Pi}_h {\bf U}_h =: \sum_{ K \in {\mathcal T}_h } \widetilde {\bf P}_K({\bm x})  \chi_K({\bm x}),
\end{equation}
where the limited polynomial vector $\widetilde {\bf P}_K({\bm x})= ( \widetilde D_K({\bm x}), \widetilde {\bm m}_K( {\bm x} ), \widetilde E_K( {\bm x} ) )^\top$ is 
given as follows. 

\begin{description}
	\item[Step 1] First, modify the mass density: 
	\begin{equation}\label{constraint-preservinglimit1}
		\widehat { D }_K ( {\bm x} ) = \theta_D \left( D_K( {\bm x}  ) - \overline D_K \right) + \overline D_K,\qquad  
		\theta_D :=\text{min} \left\{ \left|\frac{\overline{D}_K-\varepsilon_D}{\overline{D}_K-D_{\text{min}}} \right|,1 \right \},
	\end{equation}
	where $D_{\rm min}:=\min \left \{ D_K^{(1)}, \min_{jq} D_K({\bm x}_K^{(jq)})  \right \} $ with 
	\begin{equation}\label{key23452}
		D_K^{(1)}=\frac{1}{1-2\widehat{\omega}_1}\left(\overline{D}_K-\frac23\widehat{\omega}_1\sum\limits_{j=1}^3\sum\limits_{q=1}^Q\omega_{q}D_K({\bm x}_K^{(jq)})\right),
	\end{equation}
	and $\varepsilon_D$ is a small positive number introduced 
	to avoid the influence of the round-off error on the PCP property and may be taken as $\varepsilon_D = \min\{ 10^{-13}, \overline D_K \}$. 
	\item[Step 2] Then, modify 
	the polynomial vector $\widehat {\bf P}_K ({\bm x}) := \left( \widehat {D}_K( {\bm x} ), {\bm m}_K( {\bm x} ), E_K ({\bm x}) \right)^\top$ into 
	\begin{equation}\label{constraint-preservinglimit2}
		\widetilde {\bf P}_K ( {\bm x} ) = \theta_g 
		( \widehat {\bf P}_K ({\bm x}) - \overline {\bf U}_K ) + \overline {\bf U}_K, \qquad \theta_g := \min \left\{  \left| \frac{ g (\overline {\bf U}_K ) -\varepsilon_g }{ g(\overline {\bf U}_K  ) - g_{\rm min} }  \right|, 1
		\right\},
	\end{equation}
	where $ g_{\rm min} := \min \left\{ g( \widehat {\bf U}^{(2)} ), \min_{jq} g\left( \widehat {\bf P}_K ( {\bm x}^{(jq)} ) \right) \right \} $ with 
	\begin{equation}\label{key6732}
		\widehat {\bf U}^{(2)} :=\frac{1}{1-2\widehat{\omega}_1}\left(\overline{\vec{U}}_K-\frac23\widehat{\omega}_1\sum\limits_{j=1}^3\sum\limits_{q=1}^Q\omega_{q}\widehat{\bf P}_K({\bm x}_K^{(jq)})\right),
	\end{equation}
	and $\varepsilon_g$ is a small positive number introduced 
	to avoid the influence of the round-off error on the PCP property and may be taken as $\varepsilon_g = \min\{ 10^{-13}, g (\overline {\bf U}_K) \}$.
\end{description}

\begin{remark}\label{2Dquadrature}
	On any triangular cell $K$, one can construct a 2D quadrature rule,    
	which is exact for all polynomials $P \in \mathbb P^k(K)$, has positive weights, and includes all the edge Gaussian points $\{ {\bm x}_K^{(jq)} \}$ as a subset of the 2D quadrature points.  
	Zhang, Xia, and Shu \cite{zhang2012} constructed such a quadrature by using a Dubinar transform from rectangles to triangles, which gives 
	$$
	\frac{1}{ |K| } \iint_K P ( {\bm x} ) {\rm d} {\bm x} 
	= \frac23\widehat{\omega}_1\sum\limits_{j=1}^3\sum\limits_{q=1}^Q\omega_{q} P ({\bm x}_K^{(jq)}) + \sum_{q=1}^{\widetilde Q} \widetilde \omega_{q} P ( \widetilde {\bm x}_K^{ (q) } ),
	$$
	where $Q=2$ for $k=2$, $\{ \widetilde {\bm x}_K^{ (q) }, 1\le q \le \widetilde Q \}$ are the other quadrature points in $K$ with $\widetilde \omega_{q}>0$ being the associated weights satisfying 
	$
	\frac23\widehat{\omega}_1\sum\limits_{j=1}^3\sum\limits_{q=1}^Q\omega_{q}  + \sum_{q=1}^{\widetilde Q} \widetilde \omega_{q} = 2 \widehat{\omega}_1 + \sum_{q=1}^{\widetilde Q} \widetilde \omega_{q} = 1.
	$
	It follows that  
	\begin{equation}\label{key673211}
		\frac{1}{1-2\widehat{\omega}_1}\left( \frac{1}{ |K| } \iint_K P ( {\bm x} ) {\rm d} {\bm x}-\frac23\widehat{\omega}_1\sum\limits_{j=1}^3\sum\limits_{q=1}^Q\omega_{q} P_K({\bm x}_K^{(jq)})\right) = \sum_{q=1}^{\widetilde Q} 
		\frac{ \widetilde \omega_{q} }{ 1 - 2 \widehat{\omega}_1 } P ( \widetilde {\bm x}_K^{ (q) } ),
	\end{equation}
	with $\sum_{q=1}^{\widetilde Q} 
	\frac{ \widetilde \omega_{q} }{ 1 - 2 \widehat{\omega}_1 } = 1$. 
	The quadrature points are illustrated in Fig.~\ref{fig:quadraturepoints},  where the (orange) solid points denote $\{{\bm x}_K^{(jq)}\}$ and the (black) hollow circles stand for $\{{\widetilde{\bm x}}_K^{(q)}\}$. It is worth noting that this 2D quadrature is not applied to evaluate any integrals, but merely employed in our PCP limiter and theoretical analysis. 
	%as it decomposes the cell average into a convex combination of the desired point values.
 \end{remark}

\begin{figure}[htbp]
	\centering
	\includegraphics[width=0.46\textwidth]{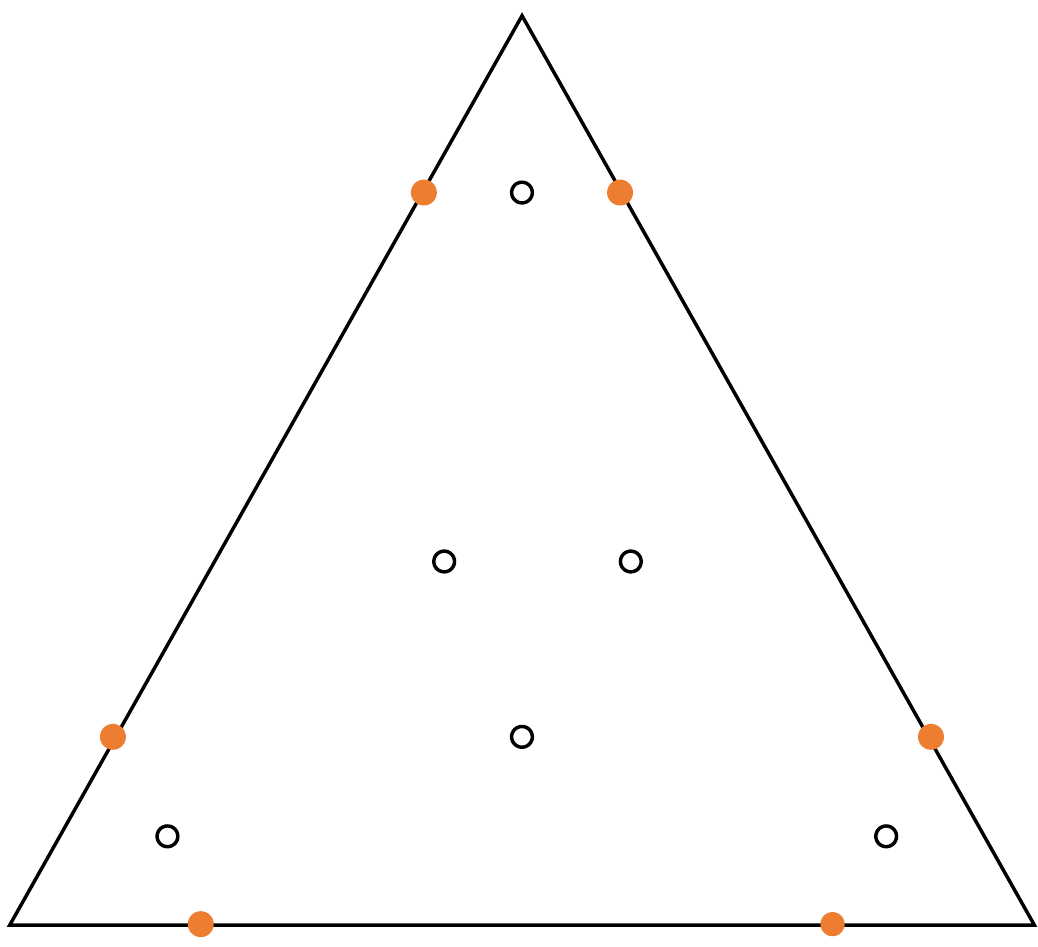}
	\caption{The quadrature points for $Q=2$ and $k=2$. }
	\label{fig:quadraturepoints}
\end{figure}

\begin{proposition}
	The limited solution $\widetilde {\bf U}_h$, defined in \eqref{constraint-preservinglimited} 
	with $\widetilde {\bf P}_K ( {\bm x} )$ given by \eqref{constraint-preservinglimit2}, satisfies 
	$$
	\widetilde {\bf U}_h \in \mathbb G_h^k, \qquad \quad \frac{1}{|K|} \iint_K \widetilde {\bf U}_h {\rm d} {\bm x} = \overline {\bf U}_K.
	$$
\end{proposition}

\begin{proof}
	According the local scaling nature of the PCP limiter, one has 
	\begin{align*}
		& \frac{1}{|K|} \iint_K \widehat D_K ({\bm x}) {\rm d} {\bm x}	
		= \theta_D \left( \frac{1}{|K|} \iint_K  D_K ({\bm x}) {\rm d} {\bm x} - \overline D_K  \right) + \overline D_K = \overline D_K, \\
		&
		\frac{1}{|K|} \iint_K \widetilde {\bf U}_h ({\bm x}) {\rm d} {\bm x}
		= \frac{1}{|K|} \iint_K \widetilde {\bf P}_K ({\bm x}) {\rm d} {\bm x}
		= \theta_g \left( \frac{1}{|K|} \iint_K \widehat {\bf P}_K ({\bm x}) {\rm d} {\bm x} - \overline {\bf U}_K \right) + \overline {\bf U}_K = \overline {\bf U}_K. 
	\end{align*}
	Note that 
	\begin{align*}
		\widetilde D_K ( {\bm x}^{(jq)} ) & = \theta_g \left( 
		\widehat D_K ( {\bm x}^{(jq)} ) - \overline D_K 
		\right)	+ \overline D_K 
		\\
		& =\theta_g \theta_D \left( 
		D_K ( {\bm x}^{(jq)} ) - \overline D_K 
		\right)	+ \overline D_K 
		\\
		& \ge \theta_g \theta_D \left( 
		D_{\rm min} - \overline D_K 
		\right)	+ \overline D_K 
		\\
		& \ge - \theta_D \left| D_{\rm min} - \overline D_K  \right| + ( \overline D_K - \varepsilon_D ) + \varepsilon_D \ge \varepsilon_D > 0,
	\end{align*}
	and applying Jensen's inequality to the concave function $g (\vec{U})=E-\sqrt{D^2+\|{\bm m}\|^2}$, we get 
	\begin{align*}
		g \left(	\widetilde {\bf P}_K ( {\bm x}^{(jq)} ) \right) & =
		g \left(  \theta_g	\widehat {\bf P}_K ( {\bm x}^{(jq)} ) + (1-\theta_g) 
		\overline {\bf U}_K \right)
		\\
		& \ge    \theta_g g \left(	\widehat {\bf P}_K ( {\bm x}^{(jq)} )  \right) + (1-\theta_g)   g \left(
		\overline {\bf U}_K \right)
		\\
		& \ge    \theta_g g_{\rm min} + (1-\theta_g)   g \left(
		\overline {\bf U}_K \right) 
		\\
		& = \theta_g  \left( 
		g_{\rm min} - g \left(
		\overline {\bf U}_K \right) 
		\right)	+ g \left(
		\overline {\bf U}_K \right) 
		\\
		& \ge - \theta_g \left|  g_{\rm min} - g \left(
		\overline {\bf U}_K \right)  \right| + \left(  g \left(
		\overline {\bf U}_K \right) - \varepsilon_g \right) + \varepsilon_g \ge \varepsilon_g > 0.
	\end{align*}
	Thus we have $\widetilde {\bf P}_K ( {\bm x}^{(jq)} ) \in G_u^{(1)} = G_u$ for all $j$, $q$,  and $K$.

	Thanks to the identity \eqref{key673211}, we obtain 
	\begin{align*}
		& \frac{1}{1-2\widehat{\omega}_1}\left(\overline{D}_K-\frac23\widehat{\omega}_1\sum\limits_{j=1}^3\sum\limits_{q=1}^Q\omega_{q} \widetilde D_K({\bm x}_K^{(jq)})\right)
		=  \sum_{q=1}^{\widetilde Q} 
		\frac{ \widetilde \omega_{q} }{ 1 - 2 \widehat{\omega}_1 } \widetilde D_K ( \widetilde {\bm x}_K^{ (q) } ) 
		\\
		&\qquad = \theta_D {\theta_g} 
		\left(
		\sum_{q=1}^{\widetilde Q} 
		\frac{ \widetilde \omega_{q} }{ 1 - 2 \widehat{\omega}_1 }  D_K ( \widetilde {\bm x}_K^{ (q) } ) - \overline D_K 
		\right) +  \overline D_K 
		\\
		& \qquad  = \theta_D \theta_g  \left( D_K^{(1)} - \overline D_K \right) +   \overline D_K
		\\
		& \qquad  \ge \theta_D \theta_g  \left( D_{\rm min} - \overline D_K \right) +   \overline D_K > 0,
	\end{align*}
	and 
	\begin{align*}
		& g \left(	\frac{1}{1-2\widehat{\omega}_1}\Big(\overline{\bf U}_K-\frac23\widehat{\omega}_1\sum\limits_{j=1}^3\sum\limits_{q=1}^Q\omega_{q} \widetilde {\bf P}_K({\bm x}_K^{(jq)})\Big) \right) 
		= g \left( \sum_{q=1}^{\widetilde Q} 
		\frac{ \widetilde \omega_{q} }{ 1 - 2 \widehat{\omega}_1 } \widetilde {\bf P}_K ( \widetilde {\bm x}_K^{ (q) } )  \right) 
		\\
		& \qquad  = g \left( \theta_g \Big( \sum_{q=1}^{\widetilde Q} 
		\frac{ \widetilde \omega_{q} }{ 1 - 2 \widehat{\omega}_1 } \widehat {\bf P}_K ( \widetilde {\bm x}_K^{ (q) } ) \Big) + (1-\theta_g) \overline {\bf U}_K  \right) 
		\\
		& \qquad  \ge \theta_g g \left( \sum_{q=1}^{\widetilde Q} 
		\frac{ \widetilde \omega_{q} }{ 1 - 2 \widehat{\omega}_1 } \widehat {\bf P}_K ( \widetilde {\bm x}_K^{ (q) } ) \right) +  (1-\theta_g) g \left( \overline {\bf U}_K  \right) 
		\\
		&\qquad  = \theta_g \left( g( \widehat{\bf U}^{(2)} ) - g \left( \overline {\bf U}_K  \right)   \right) + g \left( \overline {\bf U}_K  \right) 
		\\
		& \qquad  \ge \theta_g \left( g_{\rm min} - g \left( \overline {\bf U}_K  \right)   \right) + g \left( \overline {\bf U}_K  \right)  > 0
	\end{align*}
by applying Jensen's inequality to the concave function $g (\vec{U})$.
	Hence, we have 
	$$\frac{1}{1-2\widehat{\omega}_1}\Big(\overline{\bf U}_K-\frac23\widehat{\omega}_1\sum\limits_{j=1}^3\sum\limits_{q=1}^Q\omega_{q} \widetilde {\bf P}_K({\bm x}_K^{(jq)})\Big) \in G_u^{(1)} = G_u \qquad \forall K \in {\mathcal T}_h,
	$$
	which along with $\widetilde {\bf P}_K ( {\bm x}^{(jq)} ) \in G_u$ implies $\widetilde {\bf U}_h \in \mathbb G_h^k$. The proof is completed. 
\end{proof}

If we ignore the effects of $\varepsilon_D$ and $\varepsilon_g$ (they are $\le 10^{-13}$ and their effects are close to round-off errors), then we have 
\begin{equation}\label{eq:PiHOM}
	{\Pi_h} ( \zeta {\bf U}_h ) = \zeta {\Pi_h} (  {\bf U}_h )
\end{equation}
for any constant $\zeta >0$.

\begin{remark}
	It is worth noting that the above limiter is valid only when the reconstructed WENO solution $ {\bf U}_h \in  \overline{\mathbb{G}}_h^k$, which is ensured by \eqref{eq:rec1} and the PCP property $\overline {\bf U}_K \in G_u$ obtained in the prior Runge--Kutta stage or time-step. Similar to the bound-preserving limiters in \cite{zhang2010,zhang2010b,zhang2012,QinShu2016,ZHANG2017301,WuTang2017ApJS}, our PCP limiter also does not destroy the high-order accuracy of the reconstructed WENO solution; this will be further confirmed by numerical results in Example \ref{ex:accurary} of Section \ref{sec:examples}.  
	A similar PCP limiter was proposed in \cite{QinShu2016} for the DG methods on structured meshes. A challenge of extending such a limiter from structured meshes to unstructured meshes is to construct the 2D quadrature mentioned in Remark \ref{2Dquadrature}, which was addressed by Zhang, Xia, and Shu in \cite{zhang2012}. In addition, the present PCP limiter is different from the one in \cite{QinShu2016} in two aspects: (i) our PCP limiter uses the concavity \cite{WuTang2015}  of $g({\bf U})$  and thus avoids solving a quadratic equation; (ii) our limiter is 
		motivated by the simplified limiter in  \cite{ZhangShuReview}, which only involves the quadrature points on the cell edges and avoids the use of the interior quadrature points in \eqref{key673211}. Note that the present PCP limiter modifies the reconstructed solution polynomials with locally scaling, and thus 
		significantly differs from the PCP flux limiter in \cite{WuTang2015} which modifies the high-order numerical fluxes based on a PCP first-order numerical flux. 
  
\end{remark}

\subsection{Convergence-guaranteed algorithms for primitive variables recovery}\label{sec:recovery}
For any given conservative vector ${\bf U}\in G_u$, we need to recover the corresponding primitive vector ${\bf W}=\mathcal{W}({\bf U})\in G_w$. 
This procedure requires to solve a nonlinear algebraic equation \eqref{eq:Eqforp} by some root-finding algorithms, since 
the function $\mathcal{W}$ cannot be explicitly formulated due to the highly nonlinear relationship between ${\bf U}$ and ${\bf W}$. 
In this subsection, we will present three iterative algorithms, which are  provably convergent, for the recovery of primitive quantities from admissible conservative variables. 	

Before discussing our algorithms, we first look into the unique solvability of the positive solution to the nonlinear equation \eqref{eq:Eqforp}. 
Assume that ${\bf U} = (D, {\bm m}, E)^\top \in G_u^{(1)}$, we have 
$$
D>0, \qquad E> \sqrt{ D^2 + \| {\bm m} \|^2 }. 
$$
Note that  $\Phi_{ \bf U}(p) \in C^1[0,+\infty)$. We obtain  
$$
\Phi_{ \bf U}^{\prime}(p) = \frac{1}{\Gamma-1}-\frac{\|{\bm m}\|^2}{(E+p)^2}\left(1-\frac{D}{\sqrt{(E+p)^2-\|{\bm m}\|^2}}\right)\geq 1- \frac{\|{\bm m}\|^2}{(E+p)^2}>0 \quad \forall p \in [0,+\infty),
$$
where $1< \Gamma \le 2$ and $D>0$ have been used. Thus, the function $\Phi_{ \bf U}(p)$ of $p$ is strictly monotonically increasing in the interval $[0,+\infty)$. Besides, we observe that  
$$
\Phi_{ \bf U}(0) =  \left(D-\sqrt{E^2-\|\bm m\|^2}\right)\sqrt{1-\frac{\|\bm m\|^2}{E^2}}<0, 
$$
and $\lim\limits_{p\rightarrow +\infty} \Phi_{ \bf U}(p) = +\infty $ because $\lim\limits_{p\rightarrow +\infty}  \frac{\Phi_{ \bf U}(p)}p = \frac{1}{\Gamma-1}>0$. According to the Intermediate Value Theorem, there exists a unique positive pressure $p({\bf U})$ such that $\Phi_{ \bf U}( p({\bf U}) ) = 0$. By equation \eqref{eq:vUdU}, we obtain the velocity ${\bm v}({\bf U})$ and density $\rho({\bf U})$ satisfying $\|{\bm v}({\bf U})\|<1$ and $\rho({\bf U}) > 0$, respectively. 
Therefore, for any ${\bf U}\in G_u = G_u^{(1)}$, we have a unique 
${\bf W}=\mathcal{W}({\bf U})\in G_w$. 
On the other hand, for any ${\bf W} \in G_w$, the corresponding conservative vector ${\bf U} 
\in G_u = G_u^{(1)}$ is uniquely defined by \eqref{eq:UU}. We conclude: 

\begin{lemma}
	The operator $\mathcal{W}:G_u \rightarrow G_w$ is bijective.
\end{lemma}

As shown by the above analysis, the positive solution $p({\bf U})$ to the nonlinear equation \eqref{eq:Eqforp} is uniquely solvable, provided that ${\bf U} \in G_u^{(1)}=G_u$. 
However, in general it is very difficult (if not impossible) to  analytically obtain the root $p({\bf U})$. Some  root-finding algorithms have to be used to numerically compute $p({\bf U})$. 
The existence and uniqueness of the positive root are not sufficient to ensure the 
convergence of root-finding algorithms and the positivity of the numerical root. 
Therefore, a convergent root-finding algorithm, which guarantees 
the uniquely positive numerical root for the nonlinear equation \eqref{eq:Eqforp}, is highly desirable for obtaining a provably PCP scheme.

In the following, we discuss three iterative algorithms for solving the nonlinear equations \eqref{eq:Eqforp}. 
These three algorithms are provably convergent to recover a 
positive pressure as long as ${\bf U}\in G_u$.

\begin{algo}[Bisection algorithm]
	The monotonicity of $\Phi_{ \bf U}(p)$ on the interval $[0,+\infty)$ motivates us to consider the bisection method. 
	First, we need to seek an explicit bounded interval for $p({\bf U})$. Note that 
	\begin{equation*}
		\Phi_{ \bf U}  ( p ( {\bf U} ) ) = \frac{p  ( {\bf U} )  }{\Gamma - 1} - E + 
		\frac{ \| {\bm m} \|^2 }{E+p ( {\bf U} )  } 
		+ D \sqrt{  1 - \frac{ \| {\bm m} \|^2 }{ (E+p ( {\bf U} ) )^2 }  } =0,
	\end{equation*}	
	which implies 
	\begin{align*}
		p({\bf U}) &= (\Gamma -1) \left( E- \frac{ \| {\bm m} \|^2 }{E+p ( {\bf U} )  } 
		- D \sqrt{  1 - \frac{ \| {\bm m} \|^2 }{ (E+p ( {\bf U} ) )^2 }  } \right) 
		\\
		& \le (\Gamma -1) \left( E 
		- D \sqrt{ 1- \frac{ \| {\bm m} \|^2 }{ E^2 } } \right) = : p_R^{(0)}.
	\end{align*}
	Let $p_L^{(0)}=0$, then $p({\bf U}) \in ( p_L^{(0)}, p_R^{(0)}  )$. The bisection method proceeds as follows:  
	\begin{equation*}
		\left( p_L^{(n)},~p_R^{(n)}  \right) = 
		\begin{cases}
			( p_L^{(n-1)},~p^{(n-1)}  ), & \mbox{if~} \Phi_{ \bf U}  ( p^{(n-1)} ) > 0, 
			\\
			( p^{(n-1)},~p_R^{(n-1)}  ), & \mbox{if~} \Phi_{ \bf U}  ( p^{(n-1)} ) \le 0,
		\end{cases}
		\quad p^{(n-1)} := \frac{ p_L^{(n-1)} + p_R^{(n-1)} }2,\quad n=1,2,\dots
	\end{equation*}
	It is easy to show that $p({\bf U}) \in ( p_L^{(n)}, p_R^{(n)}  )$ and
	$$
	\big| p^{(n)} - p({\bf U})  \big| \le 
	\frac{ p_R^{(0)} }{2^{n+1}}, 
	$$	
	which indicates the convergence $\lim \limits_{n \to +\infty} p^{(n)} = p({\bf U})$.
\end{algo}

\begin{algo}[Fixed-point iteration algorithm]
	Motivated by \cite{marquina2019capturing}, we consider the following iterative method for solving the nonlinear equations  \eqref{eq:Eqforp}:
	\begin{equation} \label{key5633}
		\begin{aligned}
			p^{(0)} &= \frac12  (\Gamma -1) \left( E 
			- D \sqrt{ 1- \frac{ \| {\bm m} \|^2 }{ E^2 } } \right) > 0,
			\\
			p^{(n)} &= -(\Gamma -1) \Phi_{\bf U}(  p^{(n-1)}  ) +  p^{(n-1)}, \qquad n=1,2,\dots
		\end{aligned}
	\end{equation}
	For any $p\ge 0$ it holds 
	\begin{align} \nonumber
		0\le -(\Gamma -1) \Phi_{\bf U}'(p) + 1  &= (\Gamma -1) \frac{\|{\bm m}\|^2}{(E+p)^2}\left(1-\frac{D}{\sqrt{(E+p)^2-\|{\bm m}\|^2}}\right) 
		\\  \nonumber
		& \le (\Gamma -1) \frac{\|{\bm m}\|^2}{E^2} := \delta < 1,
	\end{align}
	which implies that 
	\begin{equation}\label{key65443}
		\max_{ 0 \le p < +\infty }	\left |-(\Gamma -1) \Phi_{\bf U}'(p) + 1 \right | \le \delta < 1,
	\end{equation}
	and that 
	$-(\Gamma -1) \Phi_{\bf U}(p) + p $ is a monotonically increasing function of $p$ in the interval $[0,+\infty)$. Thus, if $p^{(n-1)}>0$, we have 
	$$
	p^{(n)} \ge -(\Gamma -1) \Phi_{\bf U}(  0  ) +  0 =  (\Gamma-1) \left( 
	E - \frac{ \| {\bm m} \|^2 }{E} - D \sqrt{ 1 - \frac{ \| {\bm m} \|^2 }{E^2 }  }
	\right) > 0,
	$$
	where $0<D<\sqrt{E^2-\| {\bm m} \|^2}$ has been used. By induction, we obtain $p^{(n)}>0$ for all $n\ge 0$. The error for the iteration \eqref{key5633} can be estimated as follows:
	\begin{align*}
		\left| p^{(n)} - p({\bf U}) \right| &= \left| 
		\left( -(\Gamma -1) \Phi_{\bf U}(  p^{(n-1)}  ) +  p^{(n-1)} \right) 
		- \left( -(\Gamma -1) \Phi_{\bf U}(  p ({\bf U}) ) +  p ({\bf U}) \right) 
		\right|
		\\
		& =  \left |-(\Gamma -1) \Phi_{\bf U}'( \xi ) + 1 \right | \left| p^{(n-1)} - p({\bf U}) \right|, 
		\\
		& \le \delta \left| p^{(n-1)} - p({\bf U}) \right|,
	\end{align*}
	where $0< \min\{ p^{(n-1)},  p ({\bf U}) \} \le \xi \le \max\{ p^{(n-1)},  p ({\bf U}) \}$;  
	we have sequentially used $\Phi_{\bf U}(  p ({\bf U}) )=0$, the Mean Value Theorem, $p^{(n-1)}>0$, and \eqref{key65443}. Recursively using the above estimate gives 
	$$
	\left| p^{(n)} - p({\bf U}) \right| \le \delta^2 \left| p^{(n-2)} - p({\bf U}) \right| \le \dots \le \delta^n \left| p^{(0)} - p({\bf U}) \right| \le \delta^n \frac{ p_R^{(0)} }2,
	$$ 
	which indicates the convergence $\lim \limits_{n \to +\infty} p^{(n)} = p({\bf U})$ because $0\le \delta < 1$.
\end{algo}

\begin{algo}[Hybrid iteration algorithm]
	As we can see from the above analysis, 
	the bisection algorithm and the fixed-point iteration algorithm 
	have  different contraction rates, specifically, the  rate is  
	$\frac12$ for the bisection algorithm, and the estimated rate for the fixed-point iteration algorithm is $\delta = (\Gamma -1) \frac{\|{\bm m}\|^2}{E^2}$. 
	In order to further accelerate the convergence, 
	we devise a new hybrid algorithm, which enjoys the smaller contraction rate by switching the above two algorithms. Specifically, when ${\bf U}$ satisfies $\delta\ge \frac12$ 
	we use the bisection algorithm; otherwise, the fixed-point iteration algorithm is employed instead. 
	Clearly, such a hybrid algorithm is also convergence-guaranteed. 
	Our numerical experiments discussed in Remark \ref{Rem:Compare} will show that this hybrid iteration algorithm 
	is very efficient and faster than the other two algorithms.  
\end{algo}

%\subsection{The algpseudocode of fixed-point iteration algorithm}
\begin{algorithm}[htbp]
	\caption{Efficient implementation of hybrid iteration algorithm for recovering primitive variables}\label{alg:fix}
	\begin{algorithmic}
		\Require ${\bf U}=(D,{\bm m},E) \in G_u$ satisfying $D>0$ and $E> \sqrt{ D^2 + \| {\bm m} \|^2 } $.
		\Ensure $(\rho,{\bm v},p) = {\bm {\mathcal W}} ({\bf U}) \in G_w$ satisfying $\rho>0$, $p>0$, and $\| {\bm v}\|<1$. 
		\State Define $\varepsilon_{rf}$ as the round-off error. Set the allowable error tolerance $\varepsilon_{tol} = 10^{-15}$
		\State Set $r = \min \Big\{ \frac12, (\Gamma -1) \frac{\|{\bm m}\|^2}{E^2} \Big\}$
		\State $p_L \gets 0$.  
		\State $p_R \gets (\Gamma -1) \Big( E 	- D \sqrt{ 1- \frac{ \| {\bm m} \|^2 }{ E^2 } } \Big) $
		\State $p \gets \frac{p_R}{2}$
		\State Set $n=0$. Set $N  =  \log(  \varepsilon_{rf} / p )/\log(r)$
		\State $\Phi_0 \gets \varepsilon_{tol} + 1$
		\While{$\Phi_0 > \varepsilon_{tol}$ and $n < N$}
		\State $\Phi_{ \bf U} \gets  \frac{p }{\Gamma - 1} - E + 
		\frac{ \| {\bm m} \|^2 }{E+p} 
		+ D \sqrt{  1 - \frac{ \| {\bm m} \|^2 }{ (E+p )^2 }  }$
		\If { $r<\frac12$ }
				\State $p\gets-(\Gamma-1)  \Phi_{ \bf U} + p$
		\Else 
		\If{$\Phi_{ \bf U} <0$ }
		\State $p_L \gets p$
		\Else
		\State $p_R \gets p$
		\EndIf 
		\State $p \gets \frac{p_L + p_R}{2}$
		\EndIf
		\State $\Phi_0 \gets |\Phi_{ \bf U} |$
		\State $n \gets n+1$
		\EndWhile
		\State ${\bm v} \gets \frac{\bm m}{E+p}$
		\State $\rho \gets D \sqrt{1- \| {\bm v} \|^2}$
	\end{algorithmic}
\end{algorithm}

\begin{remark}
	Both the bisection algorithm and the fixed-point iteration algorithm converge linearly,  and so does our hybrid algorithm. 
	Another popular root-finding algorithm is Newton's algorithm, which often converges quadratically. 
	However, the convergence of Newton's algorithm requires the initial guess $p^{(0)}$ to be 
	sufficiently close to the true root $p({\bf U})$, which is difficult to guarantee in practice. 
	Moreover, in the present problem, we observe that the approximate pressure may become negative during the Newton's iteration (even if ${\bf U} \in G_u$ is admissible), causing the failure of the iteration. 
	When such failure or divergence occurs, we have to restart the Newton's iteration by trying a different initial guess, until it successfully converges to a positive pressure at the desired accuracy. 
	Therefore, Newton's algorithm is not convergence-guaranteed. 
	We will compare our three algorithms with the Newton's algorithm by numerical experiments; see Remark \ref{Rem:Compare}.
\end{remark}

	\subsection{Rigorous proof of physical-constraint-preserving property}\label{sec:CPP}
Now we are in the position to provide a rigorous proof of the PCP property \eqref{eq:constraint-preservingproperty} of our numerical method. Several lemmas are first derived, which pave the way to our proof.

A novel equivalent form of the set $G_u$ is {first} given in \eqref{key333}. 
Compared with the original form in \eqref{Gu} and the equivalent form in \eqref{Gu1}, the following equivalent form $G_u^{(2)}$ has a distinctive feature---all the constraints in $G_u^{(2)}$ are not only explicit but also linear with respect to $\bf U$. Benefit from this feature, the proof of the PCP property becomes more convenient; see Theorem \ref{THCPP}.

\begin{lemma}\label{lem:Gu2}
	The admissible state set $G_u$ is exactly equivalent to the following set
\begin{equation}\label{key333}
G_u^{(2)} := 
\left\{ 
{\bf U}=(D, {\bm m}, E)^\top \in \mathbb R^4:~D>0,~E-{\bm m} \cdot {\bm v}_* - D \sqrt{ 1 - \| {\bm v}_* \|^2 }>0,~\forall {\bm v}_* \in \mathbb B_1( {\bf 0} )
\right\},
\end{equation}	
where $\mathbb B_1( {\bf 0} ):=\{ {\bm x}\in \mathbb R^2: \|{\bm x}\|<1 \}$ denotes the open unit ball centered at $\bf 0$ in $\mathbb R^2$. 
\end{lemma}
\begin{proof}
	To prove the equivalence of these two sets $G_u^{(2)}$ and $G_u^{(1)}$, it is enough to prove that $G_u^{(2)} \subset G_u^{(1)}$ and $G_u^{(1)}\subset G_u^{(2)}$ establised simultaneously.
	
	First, prove $\vec{U}\in G_u^{(2)}  \Rightarrow \vec{U}\in G_u^{(1)}$. Let ${\bf U}=(D, {\bm m}, E)^\top \in G_u^{(2)}$, then we have $D>0$ and $E-{\bm m} \cdot {\bm v}_* - D \sqrt{ 1 - \| {\bm v}_* \|^2 }>0,~\forall {\bm v}_* \in \mathbb B_1( {\bf 0} )$. If we take special $ {\bm v}_* = \frac{\bm m}{\sqrt{D^2+\| \bm m\|^2}}$ satisfying ${\bm v}_* \in \mathbb B_1( {\bf 0} )$, then  it is easy to obtain 
	\[
	  0	< E-{\bm m} \cdot {\bm v}_* - D \sqrt{ 1 - \| {\bm v}_* \|^2 } = E-\sqrt{ D^2+\| \bm m\|^2}.
	\]
	which implies the second constraint of $G_u^{(1)}$, along with $D>0$, yields ${\bf U} \in G_u^{(1)}$.
	
	Then, prove $\vec{U}\in G_u^{(1)}  \Rightarrow \vec{U}\in G_u^{(2)}$. Let ${\bf U}=(D, {\bm m}, E)^\top \in G_u^{(1)}$, then by definition we have $D>0$ and $E-\sqrt{ D^2+\| \bm m\|^2}>0$. By using Cauchy--Schwarz inequality, we deduce that 
	$$
	   \sqrt{ D^2+\| \bm m\|^2} - {\bm m} \cdot {\bm v}_* - D \sqrt{ 1 - \| {\bm v}_* \|^2 }\geq \sqrt{ D^2+\| \bm m\|^2} - \sqrt{ D^2+\| \bm m\|^2} \cdot \sqrt{ \| {\bm v}_* \|^2 + 1 - \| {\bm v}_* \|^2} = 0,
	$$
   which gives
	\begin{align*}
		  E-{\bm m} \cdot {\bm v}_* - D \sqrt{ 1 - \| {\bm v}_* \|^2 } &= E- \sqrt{ D^2+\| \bm m\|^2} + \sqrt{ D^2+\| \bm m\|^2} - {\bm m} \cdot {\bm v}_* - D \sqrt{ 1 - \| {\bm v}_* \|^2 }\\
		  &\geq E- \sqrt{ D^2+\| \bm m\|^2} >0.
	\end{align*}
    This along with $D>0$ implies $\vec{U}\in G_u^{(2)}$. The proof is completed.
\end{proof}

\begin{lemma}\label{lem:IEQ3}
For any ${\bf U}\in G_u$ and any unit vector ${\bf n}\in \mathbb R^2$, 
the following inequalities hold
\begin{align}
&	0<c_s < \sqrt{ \Gamma -1 }, \\
&   v_{\bf n} < \lambda_{\bf n}^{(4)} < 1,
\\ \label{eq:IEQ3}
& \frac{ v_{\bf n} }{ \lambda_{\bf n}^{(4)} - v_{\bf n} } < 
-1+ \gamma^2 \left( 1 + \frac{1}{c_s} \right).
\end{align}
\end{lemma}

\begin{proof}
Direct calculation gives 
$$
0< c_s = \sqrt{ \frac{ \Gamma p }{ \rho + \frac{\Gamma}{\Gamma-1}p } } \le \sqrt{ \frac{ \Gamma p }{  \frac{\Gamma}{\Gamma-1}p } } = \sqrt{ \Gamma -1 },
$$
and 
$$
\lambda_{\bf n}^{(4)} - v_{\bf n} =  \frac{ -v_{\bf n} c_s^2 \gamma^{-2}+c_s\gamma^{-1}\sqrt{1-v^2_{\bf n}-(\|{\bm v}\|^2-v^2_{\bf n} )c^2_s}}{1-\|{\bm v}\|^2c^2_s} \ge \frac{c_s\gamma^{-2} (1-v_{\bf n} c_s) }{1-\|{\bm v}\|^2c^2_s} > 0.
$$
It follows that 
\begin{align*}
	\frac{ v_{\bf n} }{ \lambda_{\bf n}^{(4)} - v_{\bf n} } & = v_{\bf n} 
	\left(
	\frac{ v_{\bf n} }{ 1 - v_{\bf n}^2 } + \frac{ \sqrt{1-v^2_{\bf n}-(\|{\bm v}\|^2-v^2_{\bf n} )c^2_s} } { c_s \gamma^{-1} ( 1 - v_{\bf n}^2 ) }
	\right)
	\\
	& \le \frac{ v_{\bf n}^2 }{ 1 - v_{\bf n}^2 } + \frac{  |v_{\bf n}| } { c_s \gamma^{-1} \sqrt{1-v^2_{\bf n}} }
	\\
	& \le \frac{ \| {\bm v} \|^2 }{ 1 - \| {\bm v} \|^2 } + \frac{  \| {\bm v} \| } { c_s \gamma^{-1} \sqrt{1- \| {\bm v} \|^2  } }
	\\
	& = -1 + \gamma^2 \left( 1+ \frac { \| {\bm v} \| }{c_s} \right)  < -1 + \gamma^2 \left( 1+ \frac { 1 }{c_s} \right). 
\end{align*}
The proof is completed. 
\end{proof}

\begin{lemma}\label{leam:keyIEQ}
For any ${\bf U}\in G_u$, any unit vector ${\bf n}\in \mathbb R^2$, any ${\bm v}_* \in \mathbb B_1({\bf 0})$, and any $\lambda \ge \lambda_{\bf n}^{(4)} $, we have 
\begin{equation}\label{keyIEQ}
	- \left( {\bf n} \cdot {\bf F} ( {\bf U} ) \right) \cdot {\bm \xi}_* > - \lambda {\bf U}  \cdot {\bm \xi}_*,
\end{equation}
where ${\bm \xi}_* := \left( -\sqrt{ 1-\| {\bm v}_* \|^2 }, -{\bm v}_*, 1 \right)^\top$.
\end{lemma}

\begin{proof}
First we consider the case $\lambda = \lambda_{\bf n}^{(4)} $, which satisfies 
\begin{equation}\label{key6681}
	\gamma^2 ( \lambda - v_{\bf n} )^2 = \frac{ (1 - \lambda^2) c_s^2 }{ 1 - c_s^2 }, 
\end{equation}
and we deduce that 
\begin{align*}
	 \left( \lambda {\bf U} - {\bf n} \cdot {\bf F} ( {\bf U} ) \right)  \cdot {\bm \xi}_* 
	 & = -( \lambda - v_{\bf n} ) \rho \gamma \sqrt{ 1 - \| {\bm v}_* \|^2 } 
	 - \left(
	 (\lambda - v_{\bf n}) \rho H \gamma^2 v_{\bf n} - p
	 \right)  {\bm v}_* \cdot {\bf n}
	 \\
	 & \quad 
	 - (\lambda - v_{\bf n}  ) \rho H \gamma^2 v_\tau {\bm v}_* \cdot ( -n_2, n_1 )
	 + 
	 (\lambda - v_{\bf n})(\rho H \gamma^2 - p) - p v_{\bf n}
	 \\
	 & \ge - J_1 \times \sqrt{ 1 - \| {\bm v}_* \|^2 + ({\bm v}_* \cdot {\bf n})^2 + ( {\bm v}_* \cdot ( -n_2, n_1 ) )^2  } + J_2
	 = J_2 - J_1,
\end{align*}
where the Cauchy–Schwarz inequality has been used, with 
\begin{align*}
J_1 &=	\sqrt{ ( \lambda - v_{\bf n} )^2 \rho^2 \gamma^2
	+ 	\left(
	(\lambda - v_{\bf n}) \rho H \gamma^2 v_{\bf n} - p
	\right)^2 +   (\lambda - v_{\bf n}  )^2 \rho^2 H^2 \gamma^4 v_\tau^2 
},
\\
J_2 & = (\lambda - v_{\bf n})(\rho H \gamma^2 - p) - p v_{\bf n} = p (\lambda - v_{\bf n}) 
\left( \frac{ \Gamma \gamma^2 } {c_s^2} - 1 - \frac{ v_{\bf n} }{ \lambda - v_{\bf n} } \right) 
\\
& >  p (\lambda - v_{\bf n})  \gamma^2 c_s^{-2} \left( 
\Gamma - c_s (c_s +1) 
\right) > 0,
\end{align*}
in which we have used the three inequalities from Lemma \ref{lem:IEQ3}. Using \eqref{key6681}, we further derive that 
\begin{align*}
\left( \lambda {\bf U} - {\bf n} \cdot {\bf F} ( {\bf U} ) \right)  \cdot {\bm \xi}_* 
		& \ge J_2 - J_1 = \frac{1}{J_1 + J_2} \left[
		\left(
		(\rho H)^2 - 2 \rho H p - \rho^2
		\right) \gamma^2 ( \lambda - v_{\bf n} )^2 - p^2 ( 1 - \lambda^2 ) 
		\right] 
		\\
		& = \frac{1}{J_1 + J_2} \left[
		\left(
		(\rho H)^2 - 2 \rho H p - \rho^2
		\right) \frac{ (1 - \lambda^2) c_s^2 }{ 1 - c_s^2 } - p^2 ( 1 - \lambda^2 ) 
		\right] 
		\\
		& = \frac{p^2( 1 - \lambda^2 ) }{ (1-c_s^2)(J_1+J_2) } \left\{
				\left[
		\left( \frac{\Gamma}{c_s^2}  \right)^2 - \frac{2 \Gamma} { c_s^2 }
		- \left( \frac{\Gamma}{\Gamma - 1} - \Gamma c_s^{-2}
		\right)^2
		\right] c_s^2 - 1 + c_s^2 
		\right\}
		\\
		& = \frac{p^2( 1 - \lambda^2 ) }{ (1-c_s^2)(J_1+J_2) } 
		\left(
		\frac{\Gamma + 1}{\Gamma -1} + \frac{ 1 - 2\Gamma }{ (\Gamma -1)^2 } c_s^2
		\right)
		\\
		&> \frac{p^2( 1 - \lambda^2 ) }{ (1-c_s^2)(J_1+J_2) } 
		\left(
		\frac{2 -\Gamma }{\Gamma -1} 
		\right) \ge 0,
\end{align*}
which completes the proof of \eqref{keyIEQ} for $\lambda = \lambda_{\bf n}^{(4)} $. 
Because ${\bf U}\in G_u = G_u^{(2)}$, we have ${\bf U}  \cdot {\bm \xi}_*>0$. 
Therefore, when $\lambda \ge \lambda_{\bf n}^{(4)} $, it holds 
$$
- \left( {\bf n} \cdot {\bf F} ( {\bf U} ) \right) \cdot {\bm \xi}_* > - \lambda_{\bf n}^{(4)} {\bf U}  \cdot {\bm \xi}_* \ge  - \lambda {\bf U}  \cdot {\bm \xi}_*. 
$$
The proof is completed. 
\end{proof} 

\begin{lemma}\label{lem:HLL}
For any ${\bf U}^- \in G_u$, any ${\bf U}^+ \in G_u$, any unit vector ${\bf n}\in \mathbb R^2$, and any ${\bm v}_* \in \mathbb B_1({\bf 0})$, we have
\begin{align}
&	- \widehat{\bf F}^{hll} ( {\bf U}^-, {\bf U}^+; {\bf n} ) \cdot {\bm \xi}_* 
	\ge - \max_{ {\bf U} \in \{ {\bf U}^-,{\bf U}^+ \}  } \widehat \sigma ( {\bf U}; {\bf n} ) {\bf U}^- \cdot {\bm \xi}_*,
	\\
&  - \widehat{D}^{hll} ( {\bf U}^-, {\bf U}^+; {\bf n} ) \ge - \max_{ {\bf U} \in \{ {\bf U}^-,{\bf U}^+ \}  } \widehat \sigma ( {\bf U}; {\bf n} ) D^-
%	= \frac{ \sigma_r {\bf n} \cdot {\bf F}({\bf U}^-) 
%	- \sigma_l  {\bf n}  \cdot	{\bf F}({\bf U}^+) + \sigma_l \sigma_r ( {\bf U}^+- {\bf U}^- )
% } { \sigma_r - \sigma_l }
\end{align}
where $\widehat{D}^{hll}$ denotes the first component of $\widehat{\bf F}^{hll} ( {\bf U}^-, {\bf U}^+; {\bf n} )$, and 
$\widehat \sigma ( {\bf U}; {\bf n} )$ is the spectral radius of the Jacobian matrix ${\bf A}_{\bf n} ( {\bf U} )$ and is defined by 
$$
\widehat \sigma ( {\bf U}; {\bf n} ) : = \frac{|v_{\bf n}|(1-c_s^2)+c_s\gamma^{-1}\sqrt{1-v^2_{\bf n}-(\|{\bm v}\|^2-v^2_{\bf n} )c^2_s}}{1-\|{\bm v}\|^2c^2_s}.
$$
\end{lemma}

\begin{proof}
Thanks to Lemmas \ref{lem:Gu2} and \ref{leam:keyIEQ}, we have 
$$
-{\bf U}^\pm  \cdot {\bm \xi}_*>0, \qquad - \left( {\bf n} \cdot {\bf F} ( {\bf U}^\pm ) \right) \cdot {\bm \xi}_* > - \lambda_{\bf n}^{(4)} ({\bf U}^\pm) \left( {\bf U}^\pm  \cdot {\bm \xi}_* \right).
$$
Note that $ \sigma_r \ge 0$, $\sigma_l \le 0$, and $ \max_{ {\bf U} \in \{ {\bf U}^-,{\bf U}^+ \}  } \widehat \sigma ( {\bf U}; {\bf n} ) \ge \sigma_r \ge \lambda_{\bf n}^{(4)} ({\bf U}^\pm)$. 
Therefore 
\begin{align*}
	- \widehat{\bf F}^{hll} ( {\bf U}^-, {\bf U}^+; {\bf n} ) \cdot {\bm \xi}_*  & =  
	\frac{ \sigma_r \left( - {\bf n} \cdot {\bf F}({\bf U}^-) \right) \cdot {\bm \xi}_* 
		+(- \sigma_l) \left( - {\bf n}  \cdot	{\bf F}({\bf U}^+)
		\right) \cdot {\bm \xi}_*
		 - \sigma_l \sigma_r ( {\bf U}^+ - {\bf U}^- )\cdot {\bm \xi}_*
	} { \sigma_r - \sigma_l }
\\
& \ge 	\frac{- \sigma_r   \lambda_{\bf n}^{(4)} ({\bf U}^-) \left( {\bf U}^- \cdot {\bm \xi}_* \right)
	+ \sigma_l \lambda_{\bf n}^{(4)} ({\bf U}^+)
\left( {\bf U}^+ \cdot {\bm \xi}_* \right)
	- \sigma_l \sigma_r ( {\bf U}^+ - {\bf U}^- )\cdot {\bm \xi}_*
} { \sigma_r - \sigma_l }
\\
& = \frac{ \sigma_l \sigma_r - \sigma_r \lambda_{\bf n}^{(4)} ({\bf U}^-) }{\sigma_r - \sigma_l}  \left( {\bf U}^- \cdot {\bm \xi}_* \right) 
+ \frac{ (-\sigma_l) \left(  \sigma_r - \lambda_{\bf n}^{(4)} ({\bf U}^+) \right) }{\sigma_r - \sigma_l } \left( {\bf U}^+ \cdot {\bm \xi}_* \right)
\\
& \ge  \frac{ \sigma_l \sigma_r - \sigma_r \sigma_r }{\sigma_r - \sigma_l}  \left( {\bf U}^- \cdot {\bm \xi}_* \right) + 0 \left( {\bf U}^+ \cdot {\bm \xi}_* \right)
\\
& = -\sigma_r \left( {\bf U}^- \cdot {\bm \xi}_* \right) 
\ge - \max_{ {\bf U} \in \{ {\bf U}^-,{\bf U}^+ \}  } \widehat \sigma ( {\bf U}; {\bf n} ) \left( {\bf U}^- \cdot {\bm \xi}_* \right),
\\
	- \widehat{D}^{hll} ( {\bf U}^-, {\bf U}^+; {\bf n} )   & =  
\frac{ \sigma_r \left( - D^- v^-_{\bf n} \right) 
	+(- \sigma_l) \left( - D^+ v^+_{\bf n}
	\right) 
	- \sigma_l \sigma_r ( D^+ - D^- )
} { \sigma_r - \sigma_l }
\\
&=  \frac{ \sigma_l \sigma_r - \sigma_r v^-_{\bf n} }{\sigma_r - \sigma_l}  D^- 
+ \frac{ (-\sigma_l) \left(  \sigma_r - v^+_{\bf n} \right) }{\sigma_r - \sigma_l } D^+ 
\\
& \ge  \frac{ \sigma_l \sigma_r - \sigma_r \sigma_r }{\sigma_r - \sigma_l}  D^- 
\ge - \max_{ {\bf U} \in \{ {\bf U}^-,{\bf U}^+ \}  } \widehat \sigma ( {\bf U}; {\bf n} ) D^-.
\end{align*}
The proof is completed. 
\end{proof}

Based on the above lemmas, we are now ready to give the rigorous proof of the PCP property \eqref{eq:constraint-preservingproperty} for our high-order finite volume method.

\begin{theorem}\label{THCPP}
The proposed finite volume method satisfies the PCP property \eqref{eq:constraint-preservingproperty}, if $\widetilde {\bf U}_h \in \mathbb{G}_h^k$, then 
	\begin{equation}\label{eq:constraint-preservingproperty00}
	\overline{\bf U}_K + \Delta t \vec{L}_K( \widetilde {\bf U}_h)  \in G_u \quad \forall K\in\mathcal{T}_h,  
\end{equation}
under the CFL condition 
\begin{equation}\label{eq:CFL}
\Delta t \widehat \sigma_K^{(j)} \frac{ \left| {\mathcal E}_K^j \right| }{ |K| } 
\le \frac{2}{3} \widehat \omega_1,
\end{equation}
where $\widehat{\omega}_1 = \frac{1}{ L ( L-1 ) }$ is the first weight of
the $L$-point Gauss-Lobatto quadrature with $L=\left\lceil \frac{k+3}2 \right\rceil $, and
$$
\widehat \sigma_K^{(j)} := \max_{ {\bf U} \in \left\{ \widetilde {\bf U}_{jq}^{ {\rm int}(K) }, \widetilde {\bf U}_{jq}^{ {\rm ext}(K) },\forall q \right \} } \widehat \sigma ( {\bf U}; {\bf n}_K^{(j)} ).
$$ 
\end{theorem}

\begin{proof}
Because $\widetilde{\vec{U}}_h\in {\mathbb{G}}_h^k$, we have 
\begin{equation*}
	\widetilde{\vec{U}}_{jq}^{{\rm int}(K)} :=	\widetilde{\vec{U}}_h^{\text{int}(K)}({\bm x}_K^{(jq)}) \in G_u, \quad \widetilde{\vec{U}}_{jq}^{{\rm ext}(K)} := \widetilde{\vec{U}}_h^{\text{ext}(K)}({\bm x}_K^{(jq)}) \in G_u,\quad \forall j,q,K,
\end{equation*}
and 
$$
\frac{ \overline{\vec{U}}_K-\frac23\widehat{\omega}_1\sum\limits_{j=1}^3\sum\limits_{q=1}^Q\omega_{q} \widetilde{\vec{U}}_{jq}^{{\rm int}(K)}  }{1-2\widehat{\omega}_1}  \in G_u,
$$
which implies 
\begin{align}\label{eq:DDDD}
	 \overline{D}_K & > \frac23\widehat{\omega}_1\sum\limits_{j=1}^3\sum\limits_{q=1}^Q\omega_{q} \widetilde{D}_{jq}^{{\rm int}(K)},
	\\ \label{eq:UUUU} 
	\overline{\vec{U}}_K \cdot {\bm \xi}_* 	& > \frac23\widehat{\omega}_1\sum\limits_{j=1}^3\sum\limits_{q=1}^Q\omega_{q} \widetilde{\vec{U}}_{jq}^{{\rm int}(K)} \cdot {\bm \xi}_*.
\end{align}

Define $\overline{\bf U}_K + \Delta t \vec{L}_K( \widetilde {\bf U}_h)=: {\bf U}_{\Delta t} = ( D_{\Delta t}, {\bm m}_\Delta, E_{\Delta t} )^\top$. Let us first prove the positivity of $D_{\Delta t}$. 
Thanks to Lemma \ref{lem:HLL} and equation \eqref{eq:DDDD}, we obtain 
\begin{align*}
D_{\Delta t} & = \overline{D}_K 	- \frac{ \Delta t }{ |K| } \sum_{j=1}^{3}|\mathcal{E}_K^j|\sum_{q=1}^{\text{Q}} \omega_q \widehat{D}^{hll} \left( \widetilde{\vec{U}}_{jq}^{{\rm int}(K)}, \widetilde{\vec{U}}_{jq}^{{\rm ext}(K)}; {\bf n}_K^{(j)} \right)
\\
& \ge \frac23\widehat{\omega}_1\sum\limits_{j=1}^3\sum\limits_{q=1}^Q\omega_{q} \widetilde{D}_{jq}^{{\rm int}(K)} 	- \frac{ \Delta t }{ |K| } \sum_{j=1}^{3}|\mathcal{E}_K^j|\sum_{q=1}^{\text{Q}} \omega_q 
\widehat \sigma_K^{(j)} \widetilde{D}_{jq}^{{\rm int}(K)}
\\
& = \sum\limits_{j=1}^3\sum\limits_{q=1}^Q\omega_{q} \left(
\frac{2}{3} \widehat \omega_1 - 
\Delta t \widehat \sigma_K^{(j)} \frac{ \left| {\mathcal E}_K^j \right| }{ |K| } 
\right) \widetilde{D}_{jq}^{{\rm int}(K)} > 0,
\end{align*} 
where the CFL condition \eqref{eq:CFL} has been used in the last inequality. 
Similarly, using Lemma \ref{lem:HLL} and equation \eqref{eq:UUUU}, we obtain for any ${\bm v}_* \in \mathbb B_1({\bf 0})$ that 
\begin{align*}
	{\bf U}_{\Delta t} \cdot {\bm \xi}_* & = \overline{\bf U}_K \cdot {\bm \xi}_*	- \frac{ \Delta t }{ |K| } \sum_{j=1}^{3}|\mathcal{E}_K^j|\sum_{q=1}^{\text{Q}} \omega_q \widehat{\bf F}^{hll} \left( \widetilde{\vec{U}}_{jq}^{{\rm int}(K)}, \widetilde{\vec{U}}_{jq}^{{\rm ext}(K)}; {\bf n}_K^{(j)} \right) \cdot {\bm \xi}_*
	\\
	&  \ge \frac23\widehat{\omega}_1\sum\limits_{j=1}^3\sum\limits_{q=1}^Q\omega_{q} \widetilde{\vec{U}}_{jq}^{{\rm int}(K)} \cdot {\bm \xi}_*	
	- \frac{ \Delta t }{ |K| } \sum_{j=1}^{3}|\mathcal{E}_K^j|\sum_{q=1}^{\text{Q}} \omega_q 
	\widehat \sigma_K^{(j)}  \widetilde{\vec{U}}_{jq}^{{\rm int}(K)} \cdot {\bm \xi}_*
	\\
	& = \sum\limits_{j=1}^3\sum\limits_{q=1}^Q\omega_{q} \left(
	\frac{2}{3} \widehat \omega_1 - 
	\Delta t \widehat \sigma_K^{(j)} \frac{ \left| {\mathcal E}_K^j \right| }{ |K| } 
	\right) \widetilde{\vec{U}}_{jq}^{{\rm int}(K)} \cdot {\bm \xi}_* > 0.
\end{align*} 
Therefore, we have $ {\bf U}_{\Delta t} = \overline{\bf U}_K + \Delta t \vec{L}_K( \widetilde {\bf U}_h) \in G_u^{(2)} = G_u$. The proof is completed. 
\end{proof}

\subsection{Homogeneousity of our numerical method}

We now show that our numerical method inherits the homogeneity of the exact evolution operator in Proposition \ref{prop:invar}. 
From \eqref{eq:DML} one can deduce that 
$$
\sigma_l ( \zeta {\bf U}^-, \zeta {\bf U}^+; {\bf n}) = \sigma_l ({\bf U}^-, {\bf U}^+; {\bf n}), \qquad 
\sigma_r ( \zeta {\bf U}^-, \zeta {\bf U}^+; {\bf n}) = \sigma_r ({\bf U}^-, {\bf U}^+; {\bf n}).
$$
This along with ${\bf F}_i( \zeta {\bf U} ) = \zeta {\bf F}_i(  {\bf U} ) $ implies that  $\widehat{\bf F}^{hll} ( {\zeta \bf U}^-, {\zeta \bf U}^+; {\bf n} ) =\zeta \widehat{\bf F}^{hll} ( {\bf U}^-, {\bf U}^+; {\bf n} )$. Note that $\omega_q$, $|\mathcal{E}_K^j|$, and $|K|$ in \eqref{eq:FV} are independent of $\bf U$. 
We thus obtain  $\vec{L}_K(\zeta \vec{U}_h)=\zeta \vec{L}_K(\vec{U}_h)$.
Thanks to Lemma \ref{lem:charactweno} and the identity \eqref{eq:PiHOM}, we obtain 

\begin{theorem}\label{thrm:invarnumeric}
	Denote ${\mathcal S}_{h}(\overline{\bf U}) := \overline{\vec{U}}_K + \Delta t \vec{L}_K \left( {\Pi}_h \mathcal{R}_h^k \overline {\bf U} \right)$ be the single-step numerical evolution operator of our numerical scheme,
	then for any constant $\zeta>0$, we have 
	$$
	\frac{1}{\zeta} {\mathcal S}_h \big( \zeta \overline{\bf U} \big) = {\mathcal S}_h(\overline{\bf U}). 
	$$   
\end{theorem}

	\subsection{Extension to axisymmetric RHD equations in cylindrical coordinates}\label{sec:RHDAXIS}
	In order to simulate the axisymmetric jet problem (see, e.g., Example \ref{jet} in Section \ref{sec:examples}), we discuss the application of the PCP finite volume scheme to    
	the axisymmetric  RHD equations in cylindrical coordinates $(r,z)$, which can be written as  
	\begin{equation}\label{eq:RHDcylind3D}
		\frac{\partial \vec{U}}{\partial t} + \frac{\partial\vec F_{1}(\vec{U})}{\partial r} + \frac{\partial\vec F_{2}(\vec{U})}{\partial z} = {\bf S(U},r),
	\end{equation}
	where the flux ${\bf F}_i$ is the same as in \eqref{eq:UF1}--\eqref{eq:UF2}, $i=1,2,r\geq 0$, and the source term 
	\begin{equation*}
		{\bf S(U},r) = \frac{1}{r}(Dv_1, m_1v_1,m_2v_1,m_1)^{\top}.
	\end{equation*}
	All the fluid variables have the same meanings as in section \ref{sec:GovenEqn} except that the subscripts 1 and 2 denote radial and axial directions in cylindrical coordinates $(r,z)$. Similar to \eqref{eq:FV}, the semi-discrete finite volume scheme for the axisymmetric RHD equations \eqref{eq:RHDcylind3D} reads  	
	\begin{equation}\label{eq:FVSource}
		\frac{ {\rm d} \overline{{\bf U}}_K }{ {\rm d} t} =  \vec{L}_K(\vec{U}_h) + \overline{{\bf S}}_K,
	\end{equation}
	where $\overline{{\bf S}}_K$ is an approximation to the average of ${\bf S(U},r)$ over the cell $K$ which can be computed by, for example, the 2D quadrature rule in Remark \ref{2Dquadrature}. To achieve  high-order accuracy in time, the third-order SSP Runge--Kutta method is used. To ensure the PCP property, it suffices to guarantee that  
	\begin{equation}\label{eq:PCPSRHD}
		\overline{\bf U}_K + \Delta t \vec{L}_K( \widetilde {\bf U}_h) + \Delta t \overline{{\bf S}}_K \in G_u~~\forall K\in\mathcal{T}_h,  \quad
		\mbox{provided that } \widetilde {\bf U}_h \in \widetilde{\mathbb{G}}_h^k,
	\end{equation}
	where 
	\begin{align}\label{admissablesetsource}
		\widetilde{\mathbb{G}}_h^k:= \left \{\vec{u}\in \overline {\mathbb{G}}_h^k:~ {\bf u}_K^{ (jq) } \in G_u, 1\le j \le 3, 1\le q \le Q;~{\bf u}({\widetilde{\bm x}}_K^{ (q) }) \in G_u,1 \le q \le \widetilde{Q},\forall K \in\mathcal{T}_h \right \}. 
	\end{align}  
	Following Theorem \ref{THCPP} and \cite[Section 3.2]{WuTang2015}, one can deduce that if 
	the PCP limiter is used to enforce 
	$\widetilde {\bf U}_h \in \widetilde{\mathbb{G}}_h^k$, then the property \eqref{eq:PCPSRHD} is satisfied
	under the CFL type condition 
	$\Delta t \le \min\left\{{(1-\beta)} A_l, \beta A_s  \right\} $ with 
	$$A_l := \frac{2 \widehat \omega_1 |K|}{3 \max_{1\le j\le 3} \{ | {\mathcal E}_K^j | \widehat \sigma_K^{(j)} \} }, \quad A_s := \min\limits_{(r,z) \in \mathbb S_K, v_1( \widetilde {\bf U}_h(r,z) )>0} \left\{\frac{r g( \widetilde {\bf U}_h(r,z) )}{( p( \widetilde {\bf U}_h(r,z) ) + g( \widetilde {\bf U}_h(r,z) ) ) |v_1( \widetilde {\bf U}_h(r,z) )|}\right\},$$ 
	where ${\mathbb S}_K := \{{\bf x}_K^{ (jq)}, 1\le j \le 3, 1\le q \le Q;~{\widetilde{\bm x}}_K^{ (q) }, 1 \le q \le \widetilde{Q}\}$, and the parameter $\beta \in (0,1)$ can be taken as $\beta=\frac{A_l}{A_s + A_l}$.

	\section{Numerical tests}\label{sec:examples}
In this section, we will conduct several benchmark tests 
to validate the robustness, accuracy, and effectiveness of our PCP finite volume method on unstructured triangular meshes. 
All our triangular meshes are
generated by EASYMESH \cite{easymesh}, with all the grid points on the boundary uniformly distributed, and the length of the cell edges on the domain boundary will be denoted by $h$. 
Unless otherwise stated, the CFL number is taken as $0.5$, and  the ideal equation of state \eqref{EOS} with the ratio of specific heats $\Gamma = 5/3$ will be used in our computations.

\begin{expl}[Accuracy test]\label{ex:accurary}\rm
	To examine the accuracy of our method, we test two smooth relativistic isentropic vortexes propagating periodically  with a constant velocity  magnitude $w$ along the $(-1, -1)$ direction. 
	The computational domain $[-5,5]\times[-5,5]$ is divided into unstructured triangular cells with cell number  $N\in\{932,3728,14912,59648,238592\}$. 
	The setup is similar to those in \cite{BalsaraKim2016,LingDuanTang2019}. 
	 The initial rest-mass density and pressure are
	\begin{equation*}
		\rho(x,y) = (1-\alpha e^{1-r^2})^{\frac{1}{\Gamma-1}},\quad p = \rho^{\Gamma}
	\end{equation*}
	with 
	\begin{align*}
		&\alpha = \frac{(\Gamma-1)/\Gamma}{8\pi^2}\epsilon^2, \quad r = \sqrt{x^2_0 + y^2_0},\\
		&x_0=x+\frac{\gamma-1}{2}(x+y), \quad y_0 = y + \frac{\gamma-1}{2}(x+y),\quad \gamma=\frac{1}{\sqrt{1-w^2}},
	\end{align*}
	and the initial velocities are
	\begin{align*}
		&v_1 = \frac{1}{1-w(v^0_1+v^0_2)/\sqrt{2}}\left[\frac{v^0_1}{\gamma}-\frac{w}{\sqrt{2}}+\frac{\gamma w^2}{2(\gamma+1)}(v^0_1+v^0_2)\right]\\
		&v_2 = \frac{1}{1-w(v^0_1+v^0_2)/\sqrt{2}}\left[\frac{v^0_2}{\gamma}-\frac{w}{\sqrt{2}}+\frac{\gamma w^2}{2(\gamma+1)}(v^0_1+v^0_2)\right]
	\end{align*}
	with
	\begin{equation*}
		(v_1^0,v_2^0)=(-y_0,x_0)f,\quad f=\sqrt{\frac{\beta}{1+\beta r^2}},\quad \beta=\frac{2\Gamma \alpha e^{1-r^2}}{2\Gamma -1 -\Gamma \alpha e^{1-r^2}}.
	\end{equation*}
	
	For the first vortex, we take the speed $w=0.5\sqrt{2}$ and the vortex strength $\epsilon=5$. In this mild case, the PCP limiter is not needed.  
	Table \ref{tableaccuracy} lists the numerical errors of the rest-mass density $\rho$ in $l^1,l^2$-norms and the corresponding convergence rates at $t=1$ for different grid resolutions. The results show that the expected third-order convergence is obtained.
	
	% \scriptsize
	\begin{table}[htbp]
		\setlength{\abovecaptionskip}{0.cm}
		\setlength{\belowcaptionskip}{-0.cm}
		\caption{Example \ref{ex:accurary}: Numerical errors and orders for  $\rho$ at $t = 0.15$ with vortex strength $\epsilon=5$.}\label{tableaccuracy}
		\begin{center}
			\begin{tabular}{*{7}{l}}
				\toprule
				N & $l^1$ error  & order  & $l^2$ error  & order\\
				\midrule
				932         &  6.33e-02            &-               & 2.77e-02  &-      \\
				3728       &  9.48e-03            &2.7376      & 3.98e-03  & 2.7980    \\
				14912      &  1.30e-03            &2.8662      & 5.21e-04  & 2.9323  \\
				59648     &  1.67e-04             &2.9590      & 6.59e-05  & 2.9829   \\
				238592   &  2.10e-05            &2.9971       & 8.23e-06   & 3.0015 \\							
				\bottomrule
			\end{tabular}
		\end{center}
	\end{table}
	
    In order to verify the PCP property of the proposed method, 
    we consider a much stronger vortex with $\epsilon=10.0828$.  
    In this case, the lowest pressure and density are $1.78\times10^{-20}$ and $7.8\times10^{-15}$ respectively. 
    The PCP limiter is required in this test to maintain the positivity of the pressure and density, otherwise the code will break down. 
     The numerical errors and orders of $\rho$ in $l^1,l^2$-norms
    are shown in Table \ref{tableaccuracy2}, which also displays 
     the ratio $\Theta$ of the number of the PCP limited cells to the total number of cells. 
	We observe that the PCP limiter is employed on only a few cells and does not destroy the accuracy of the scheme.
	\begin{table}[h]
		\setlength{\abovecaptionskip}{0.cm}
		\setlength{\belowcaptionskip}{-0.cm}
		\caption{Numerical errors and orders for $\rho$ at $t = 0.15$ with vortex strength $\epsilon=10.0828$.}\label{tableaccuracy2}
		\begin{center}
			\begin{tabular}{*{8}{l}}
				\toprule
				N & $l^1$ error  & order  & $l^2$ error  & order  &$\Theta$\\
				\midrule
				932      &1.44e-01   &-     & 5.68e-02  &-      &1.502\%\\
				3728     &2.54e-02   &2.50  & 1.16e-02  &2.30   &0.939\%\\
				14912    &3.38e-03   &2.91  & 1.44e-03  &3.01   &0.141\%\\
				59648    &4.741e-04  &2.83  & 2.25e-04  &2.68   &0.008\%\\
				238592   &5.730e-05  &3.05  & 2.68e-05  &3.07   &0.000\%\\
				\bottomrule
			\end{tabular}
		\end{center}
	\end{table}
\end{expl}

\begin{expl}[Quasi-1D Riemann problem \uppercase\expandafter{\romannumeral1}]\label{1Driemann1}\rm
	The initial data are taken as
	\begin{equation*}
		(\rho_0, {\bm v}_0,p_0) =
		\begin{cases}
			(1.0,-0.6,0,10),&x<0.5,\\
			(10,0.5,0,20),&x>0.5.
		\end{cases}
	\end{equation*}
This example investigates the capability of our scheme in resolving rarefaction waves and contact discontinuity. 
 We divide the computational domain $[0,1]\times [-1/100,1/100]$ into a triangular mesh with $h=1/500$. 
The outflow boundary conditions are applied to all boundaries.
 \figref{fig:1Driemann1} shows both the numerical (symbols ``$\circ$'') and the 
 exact (solid lines) solutions along the line $y=0$ at $t=0.4$ obtained by our third-order finite volume scheme. It is shown that the right and left moving rarefaction waves as well as the contact discontinuity are well captured.  
\end{expl}

\begin{figure}[htbp]
	\centering
	\subfigure[$\rho$]{\includegraphics[width=0.44\textwidth]{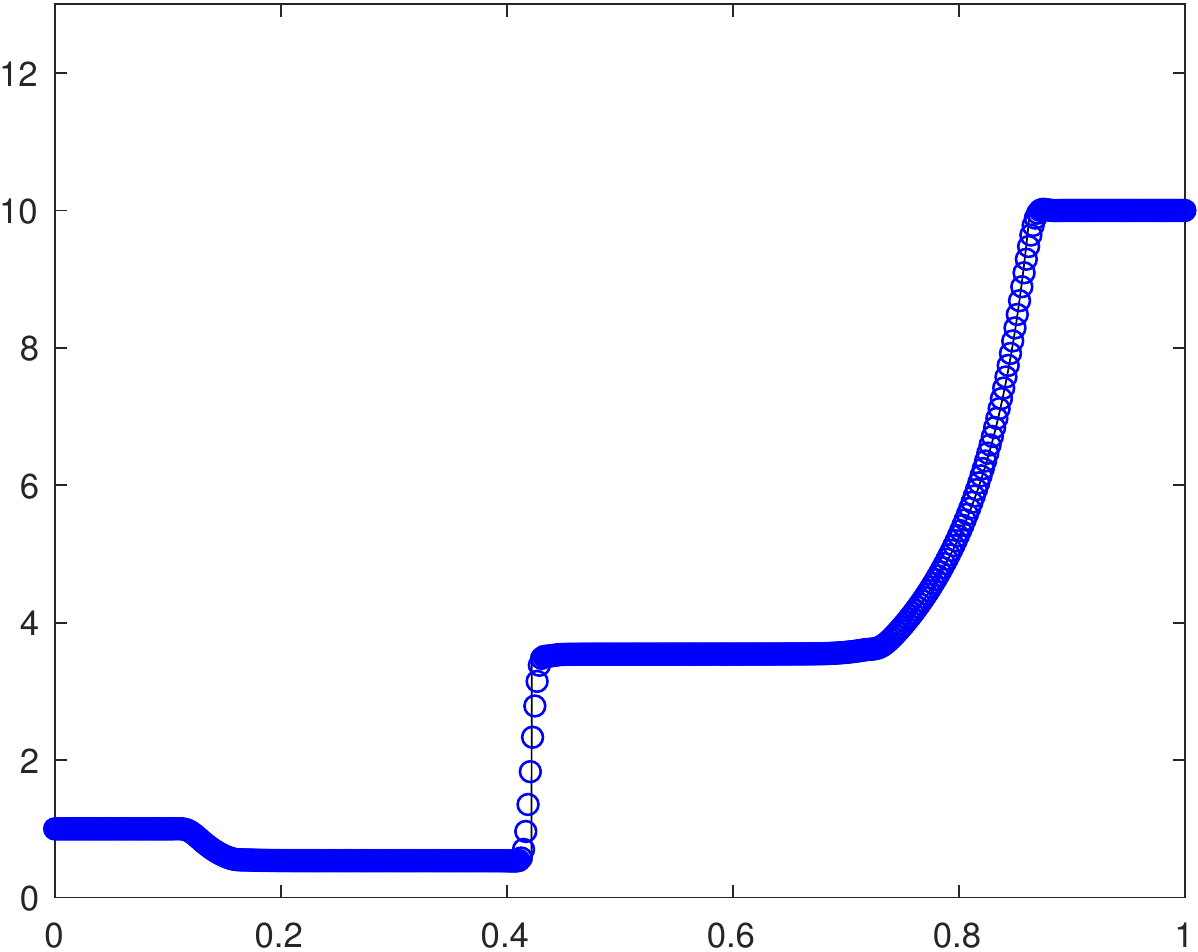}}
	\subfigure[$v_1$]{\includegraphics[width=0.45\textwidth]{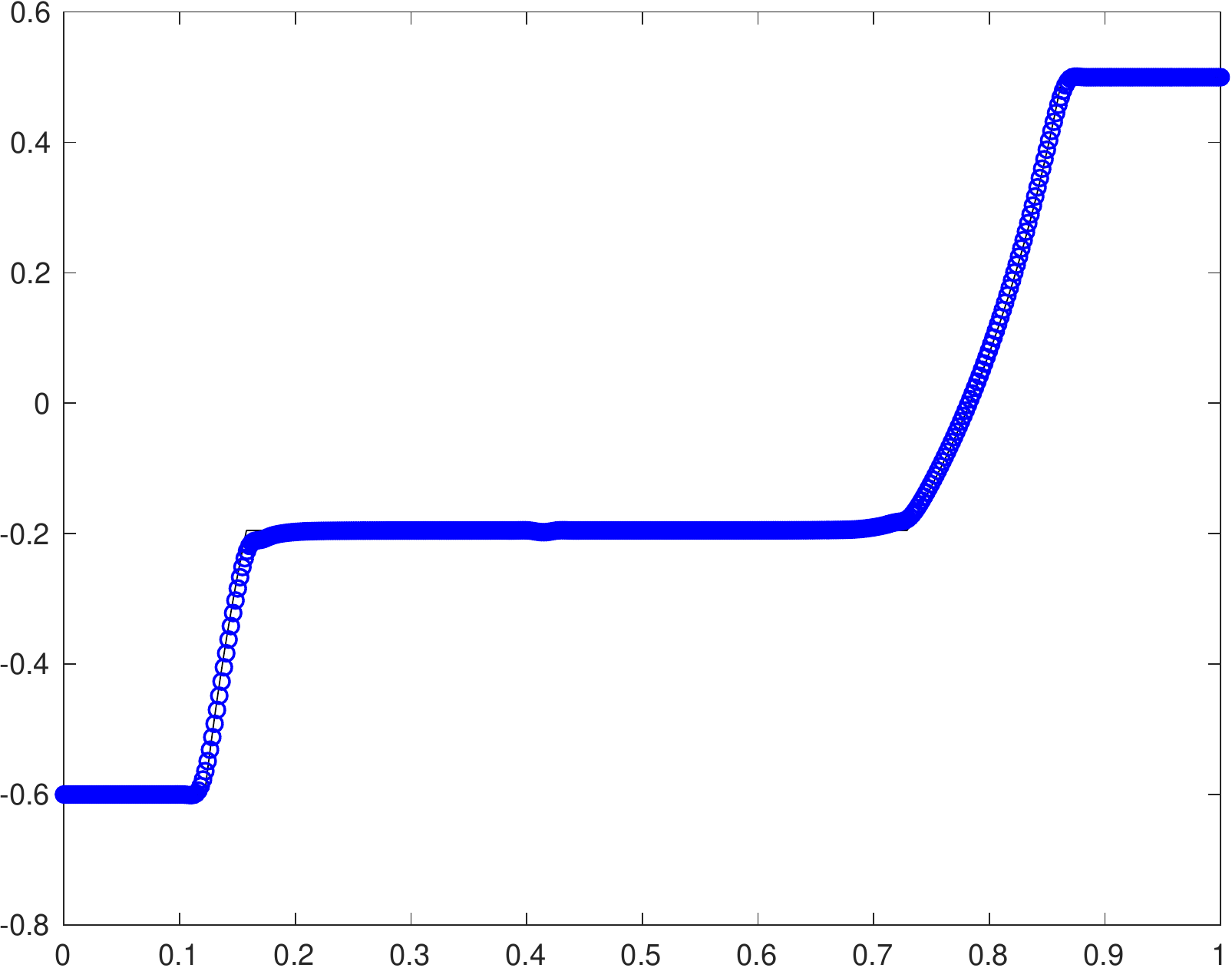}}
	\subfigure[$p$]{\includegraphics[width=0.45\textwidth]{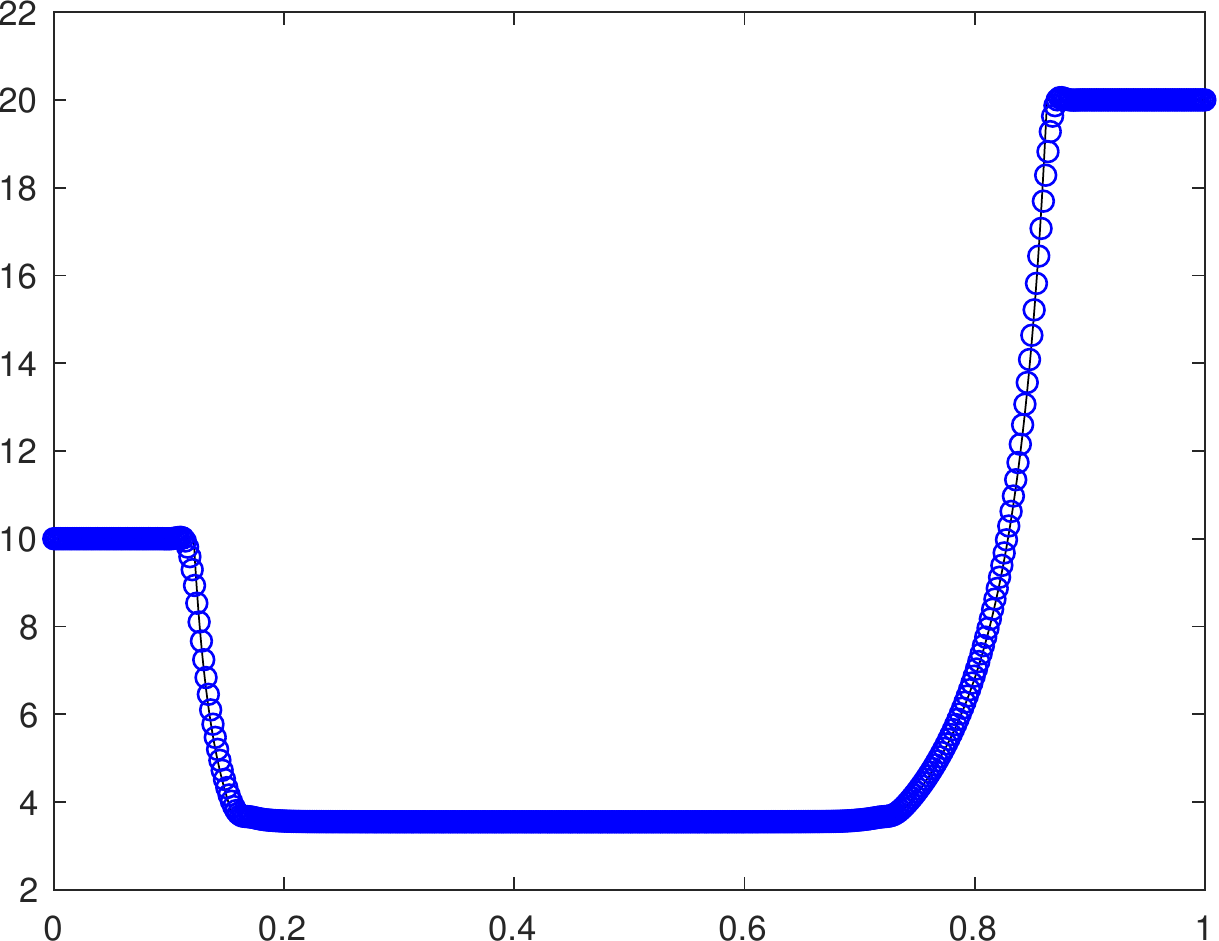}}
	\subfigure[3D density surface]{\includegraphics[width=0.45\textwidth]{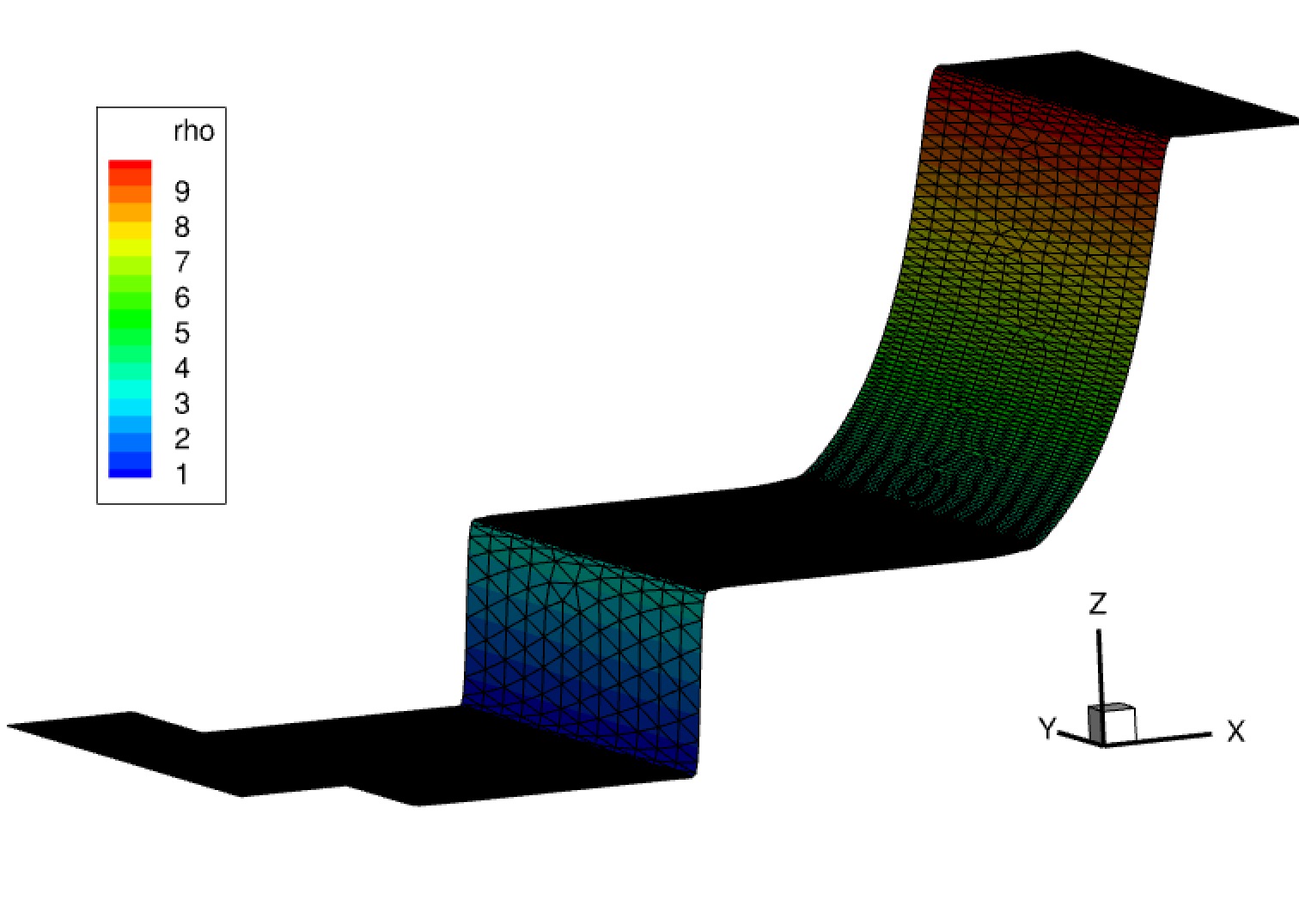}}
	\caption{Example \ref{1Driemann1}: The numerical (symbols ``$\circ$'') and the exact (solid lines) solutions of the density $\rho$, velocity $v_1$, pressure $p$ along the line $y=0$ as well as the 3D density surface at $t=0.4$. A triangular mesh with $h=1/500$ is used.} 
	\label{fig:1Driemann1}
\end{figure}

\begin{expl}[Quasi-1D Riemann problem \uppercase\expandafter{\romannumeral2}]\label{1Driemann2}\rm
	The second quasi-1D Riemann problem \cite{WuTang2015} describes the evolution of a right-moving shock wave and contact discontinuity as well as a left-moving rarefaction wave. The initial conditions are 
	\begin{equation}
		(\rho_0, {\bm v}_0,p_0)=
		\begin{cases}
			(1,0,0,10^4), &x<0.5,\\
			(1,0,0,10^{-8}), &x>0.5.
		\end{cases}
	\end{equation}
 The computational domain $[0,1]\times[-1/320,1/320]$ is divided into triangular cells with $h=1/1600$ and the outflow boundary conditions are applied to all boundaries. \figref{fig:1Driemann2} shows the numerical solutions (symbols ``$\circ$'') and the exact solutions (solid lines) of density $\rho$ and its close-up, velocity $v_1$, and pressure $p$ along the line $y=0$ at $t=0.45$. It is challenging to sharply resolve the shock and contact discontinuity since the region between them is extremely narrow; see \cite{WuTang2015}. 
 The results 
  demonstrate the good resolution of our scheme, 
  in comparison with the results of the ninth-order PCP finite difference WENO scheme \cite{WuTang2015} on uniform 1D grids. It should be noticed that the PCP limiter is essential to enforce the numerical solutions in $\mathbb{G}_h^k$; without the limiter the simulation would break down within a few time steps.

	\begin{figure}[htbp]
		\centering
		\subfigure[$ \rho $]{
			\includegraphics[width=0.46\textwidth]{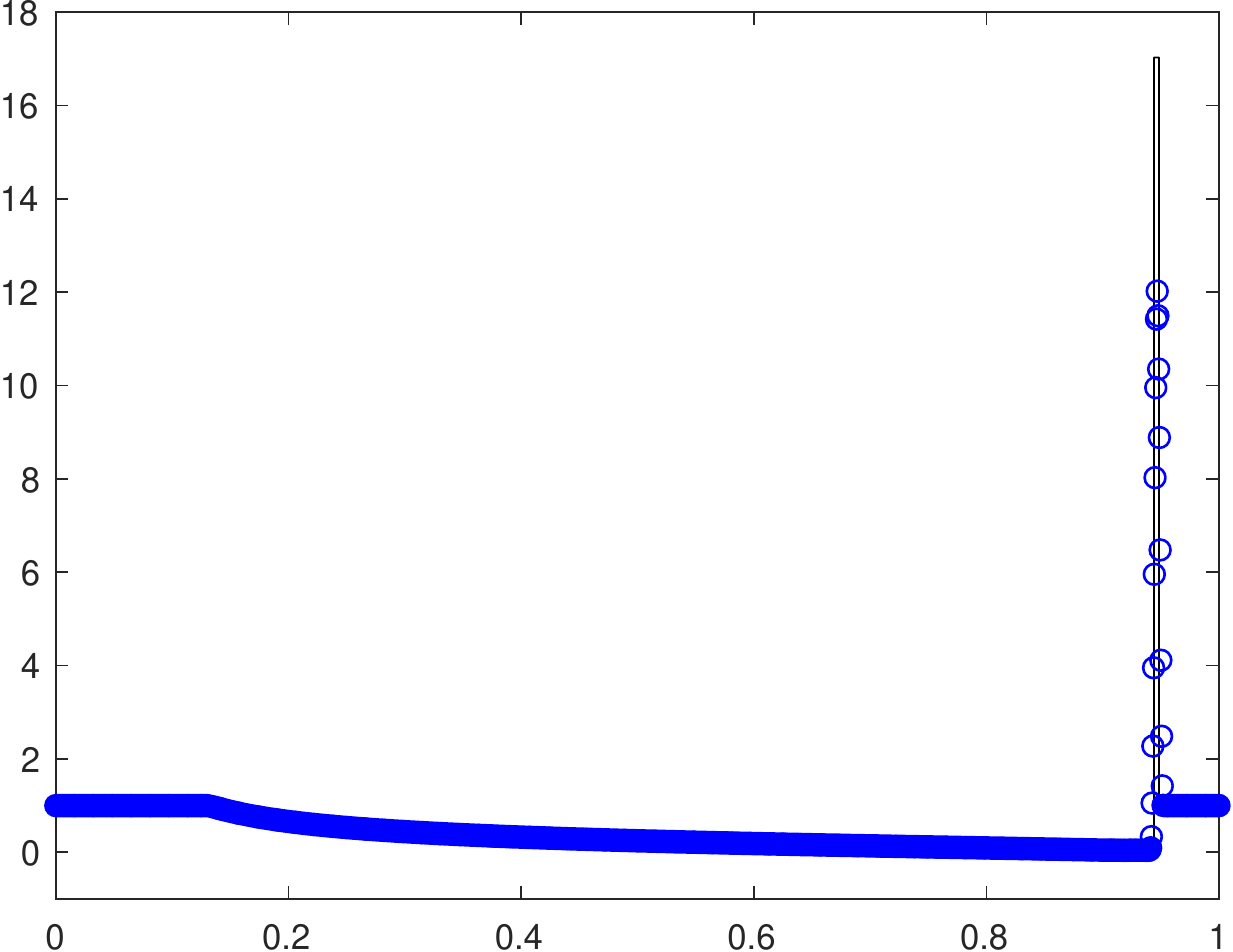}
		}
		\subfigure[Close-up of $\rho$]{
			\includegraphics[width=0.46\textwidth]{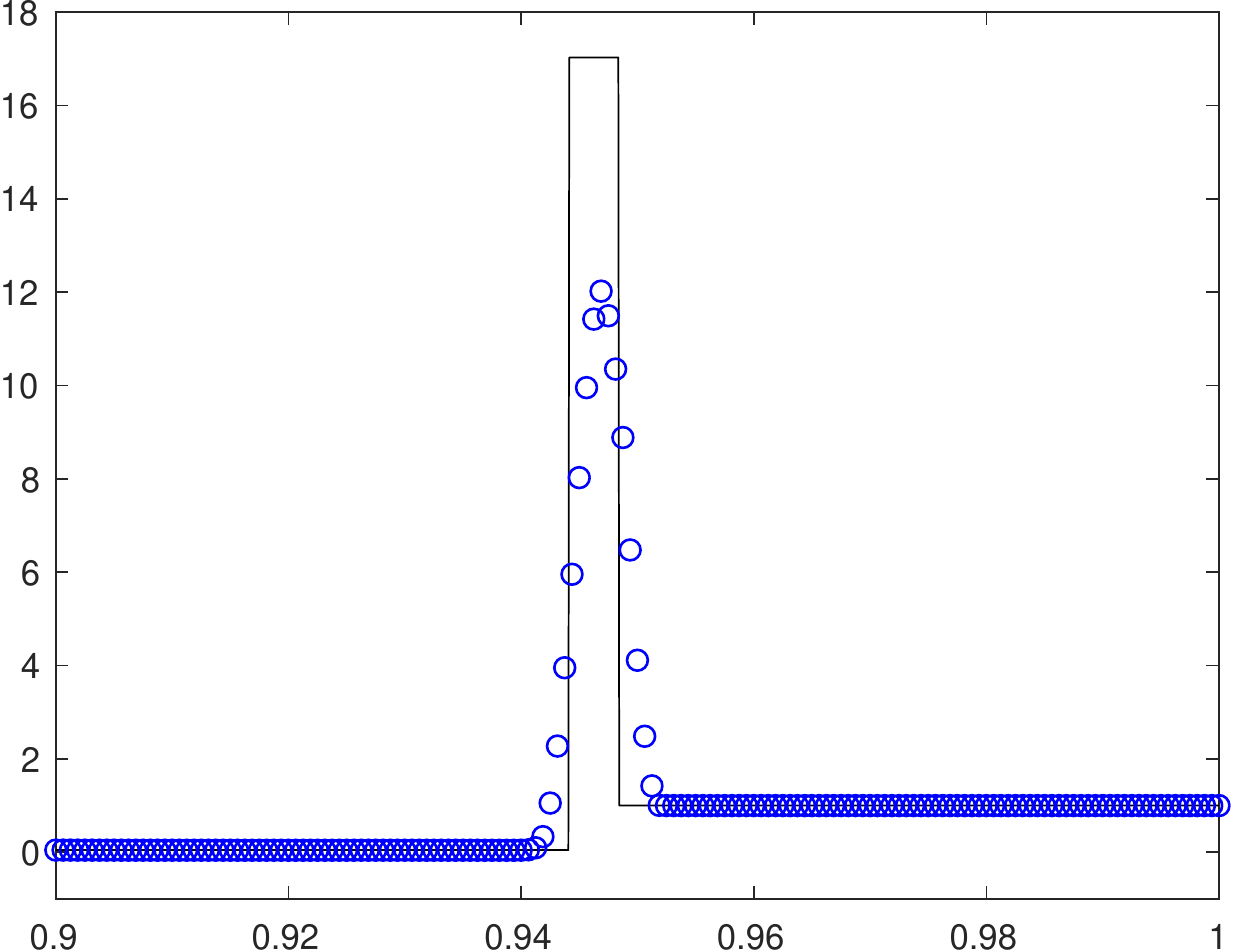}
		}
		\subfigure[$v_1$]{
			\includegraphics[width=0.45\textwidth]{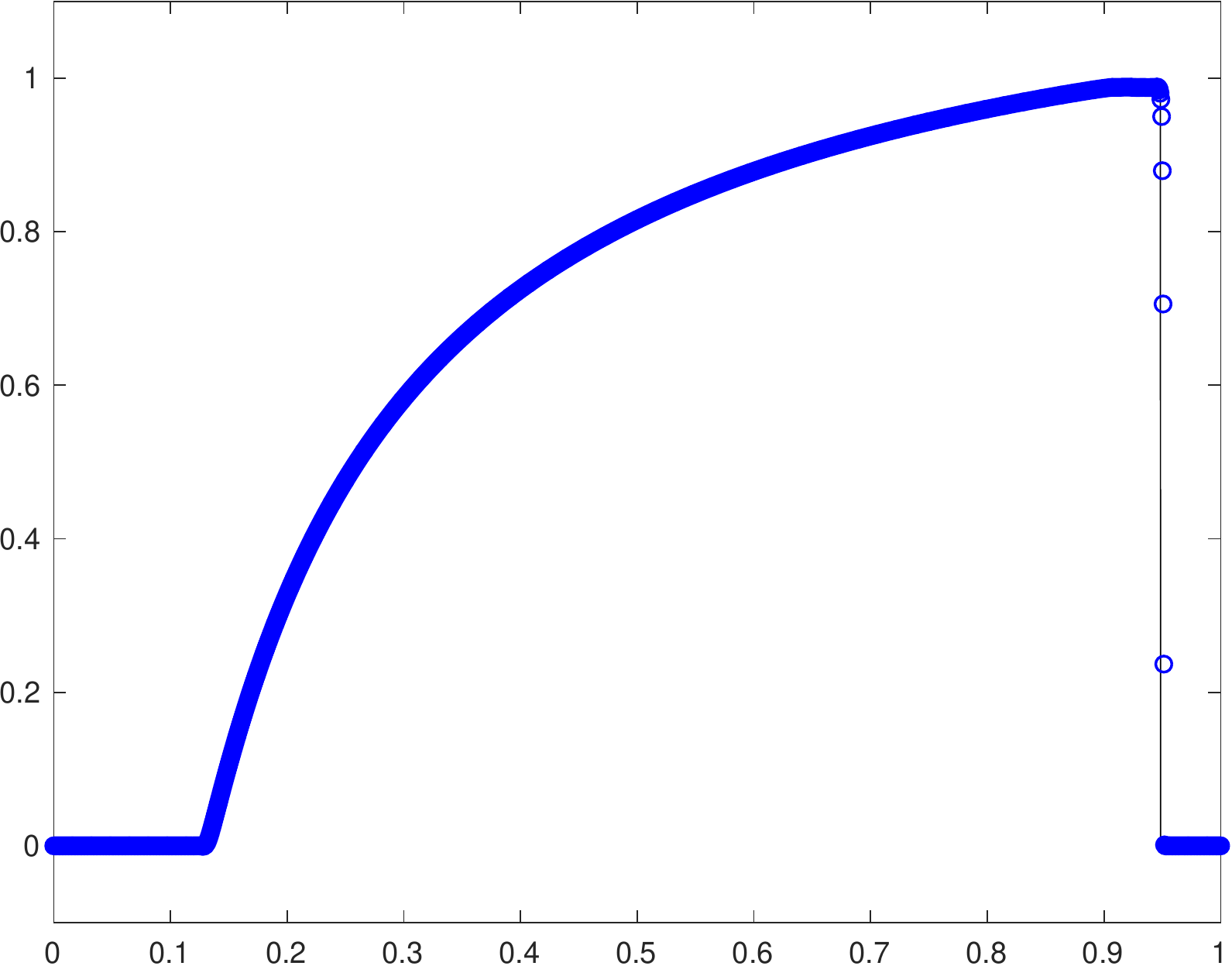}
		}
		\subfigure[$p$]{
			\includegraphics[width=0.48\textwidth]{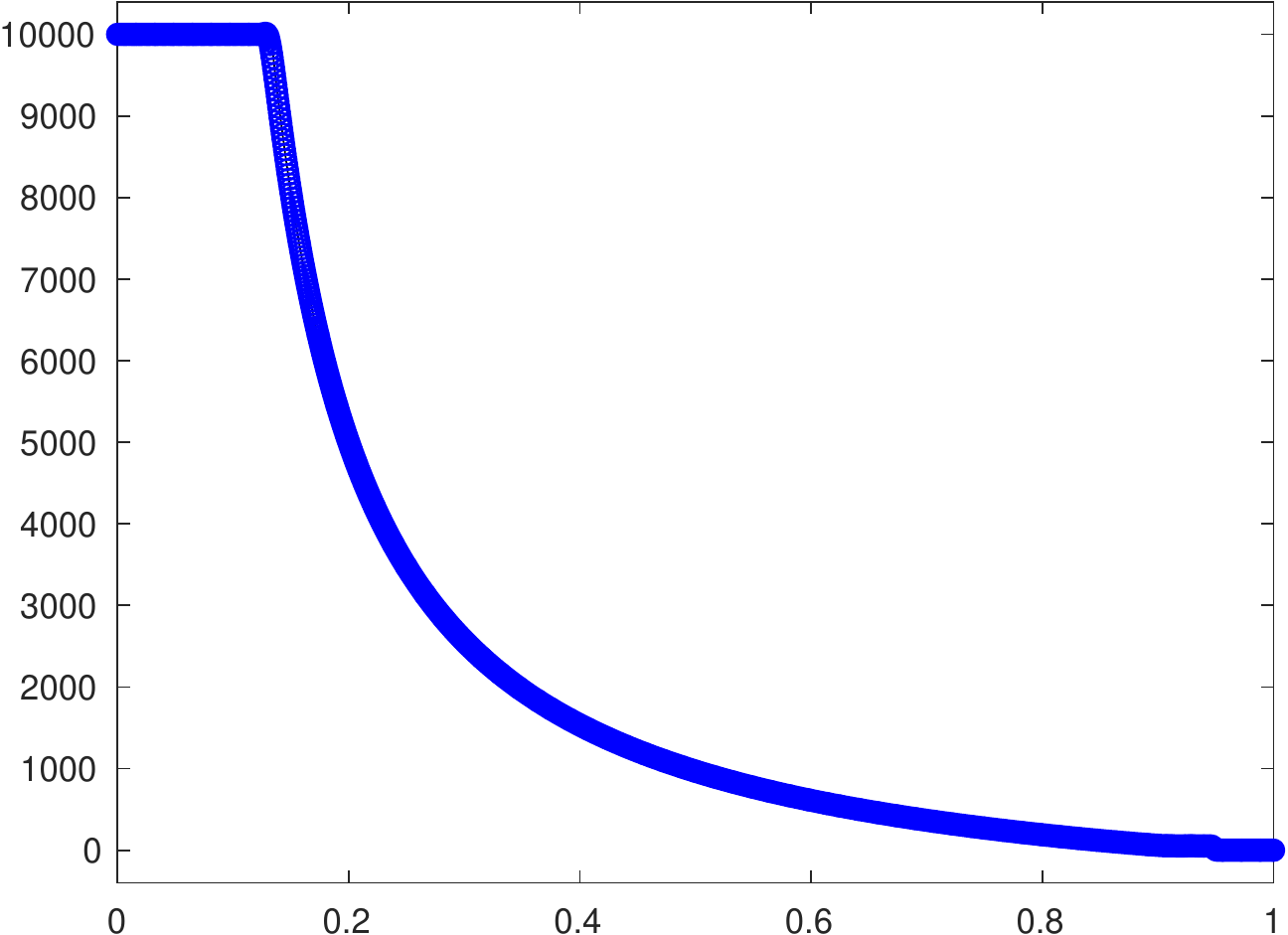}
		}
		\caption{Example \ref{1Driemann2}: The numerical (symbols ``$\circ$'') and the exact (solid lines) solutions of density $\rho$ and its close-up, velocity $v_1$, and pressure $p$ along the line $y=0$ at $t=0.45$.  A triangular mesh with $h=1/1600$ is used.}
		\label{fig:1Driemann2}
	\end{figure}
\end{expl}

\begin{expl}[Quasi-1D Riemann problem \uppercase\expandafter{\romannumeral3}]	       \label{1Driemann3}\rm
	We take this and the next examples to 
illustrate the importance of using scaling-invariant nonlinear weights in 
WENO reconstruction (as discussed in Remark \ref{remark:weight}), and to 
confirm that our numerical method does inherit the 
homogeneity of the evolution operator (as discussed in Theorem \ref{thrm:invarnumeric} and Proposition \ref{prop:invar}).

	The initial data of this Riemann problem are taken as
	\begin{equation}\label{eq:1DR3}
		(\rho_0, {\bm v}_0,p_0)=
		\begin{cases}
			(10^2,0,0,10^4), &x<0.5,\\
			(10^2,0,0,10^2), &x>0.5,
		\end{cases}
	\end{equation}
	which has similar wave structures as those of Example \ref{1Driemann2}.   The computational domain $[0,1]\times[-1/320,1/320]$ is divided into triangular cells with $h=1/1600$ and the outflow boundary conditions.
	%But here the right-moving contact discontinuity and shock wave move more slowly, and the space between these two waves is wider.
	\begin{figure}[htbp]
		\centering
		\subfigure[Use our scaling-invariant nonlinear weights \eqref{OUR-weights} ]{
			\includegraphics[width=0.44\textwidth]{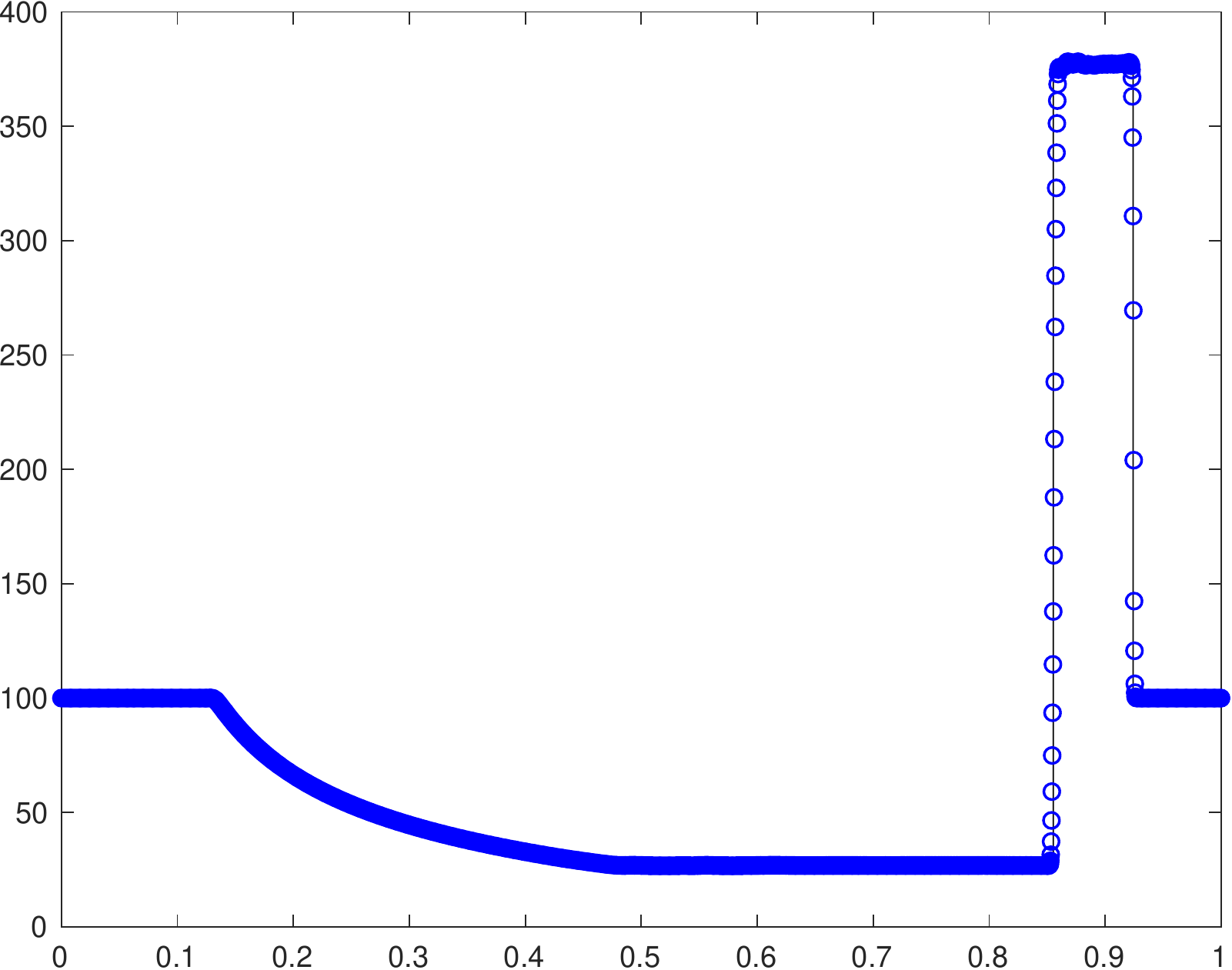}
		}
%		\subfigure[$v_1$ with dimensionless weight $\varpi_\ell$ ]{
%			\includegraphics[width=0.44\textwidth]{results/V4_1D_riemann_scaling_newweight/SN100/riemann4_u.pdf}
%		}
		\subfigure[Use the nonlinear weights \eqref{ZQweights} from \cite{zhu2018new}]{
			\includegraphics[width=0.44\textwidth]{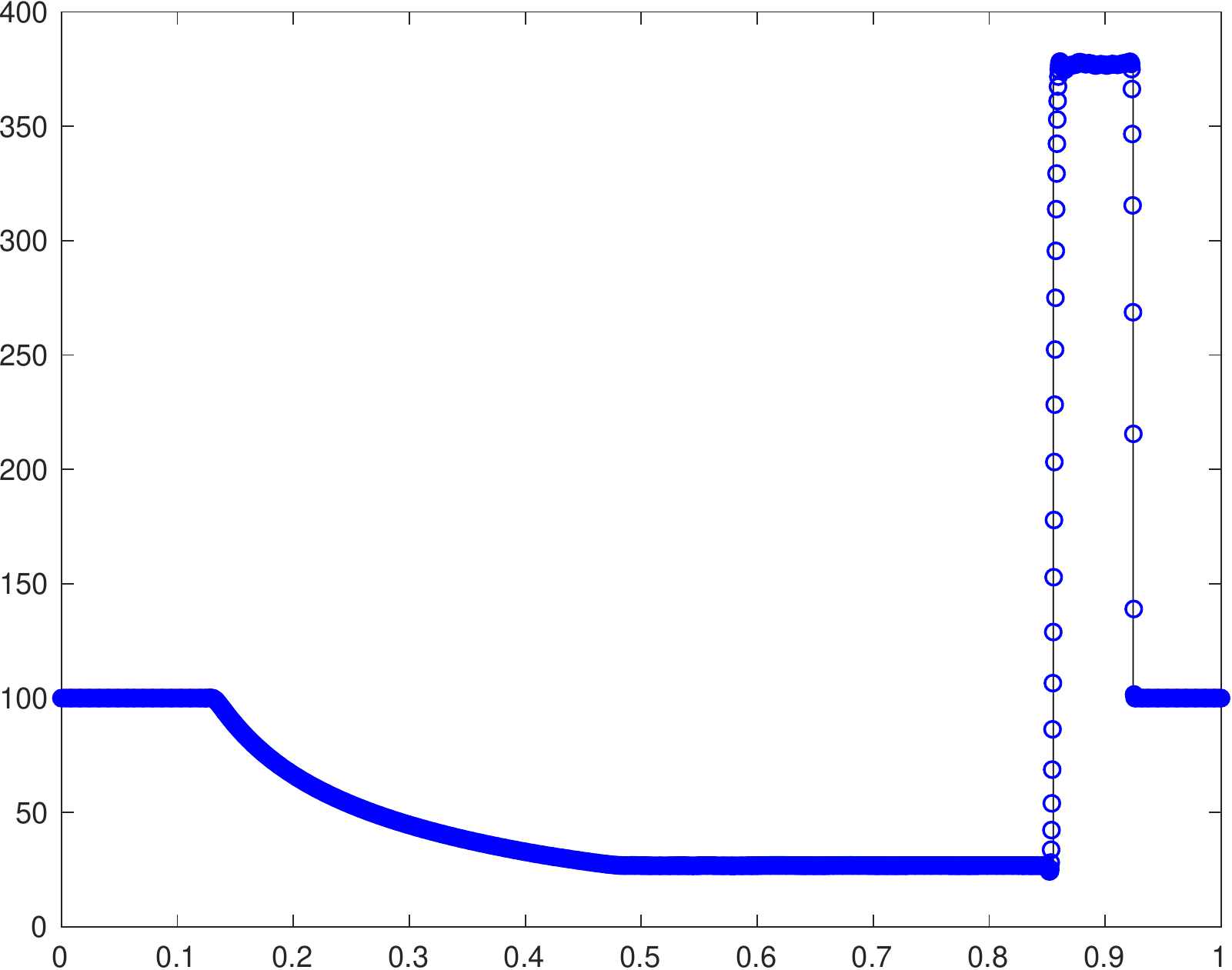}
		}
%		\subfigure[$v_1$  with dimensioned weight $\tilde{\varpi}_\ell$]{
%			\includegraphics[width=0.44\textwidth]{results/V4_1D_riemann_scaling_oldweight/SN100/riemann4_u.pdf}
%		}
		\caption{Example \ref{1Driemann3}: The numerical solution (symbols ``$\circ$'') and exact solution (solid lines) of density along the line $y=0$ at $t=0.45$ obtained by using different nonlinear weights. A triangular mesh with $h=1/1600$ is used. 
		}
		\label{1Driemann3:fig1}
	\end{figure}
	Fig.~\ref{1Driemann3:fig1} shows the numerical solutions along the line $y=0$ obtained by our method using respectively 
	the scaling-invariant nonlinear weights \eqref{OUR-weights} and the non-scaling-invariant nonlinear weights \eqref{ZQweights} from \cite{zhu2018new}. We see that both weights deliver satisfactory results which match the exact solution well. 
	
	However, their performances are quite different if the initial data are scaled. 
	To confirm this, we 
	 scale the initial data \eqref{eq:1DR3} to be $(\zeta \rho_0, {\bm v}_0,\zeta p_0)$ with the constant $\zeta >0$,  
	then 
	 Proposition \ref{prop:invar} tells us that the exact density at time $t$ is equal to  $\zeta \rho(x,t)$. 
	Let $\rho_h(x,t)$ and $\rho^{(\zeta)}_h(x,t)$ denote the numerical solutions for 
	the unscaled and  scaled initial data, respectively. 
	We hope the numerical solutions also satisfy the homogeneity, namely, $\rho^{(\zeta)}_h(x,t) = \zeta \rho_h(x,t)$ up to round-off error.  
	We run the code with the scaling number $\zeta=10^{-2}$, 
	by using respectively 
	the scaling-invariant weights \eqref{OUR-weights} and the non-scaling-invariant weights \eqref{ZQweights}. The results are presented and compared in Fig.~\ref{1Driemann3:fig2}. 
	One can see that the numerical result obtained by using 
	the scaling-invariant weights \eqref{OUR-weights} is non-oscillatory and 
	  preserves the homogeneity up to $10^{-12}$. 
	  However, the numerical solution with the non-scaling-invariant weights \eqref{ZQweights}
	has obvious overshoots/undershoots, and the homogeneity is also not satisfied. 
	These observations are further validated by
	the simulation results for $\zeta=10^{-4}$, as shown in Fig.~\ref{1Driemann3:fig3}.

	\begin{figure}[htbp]
		\centering
		\subfigure[$\rho^{(\zeta)}_h$ with scaling-invariant weights \eqref{OUR-weights}]{
			\includegraphics[width=0.44\textwidth]{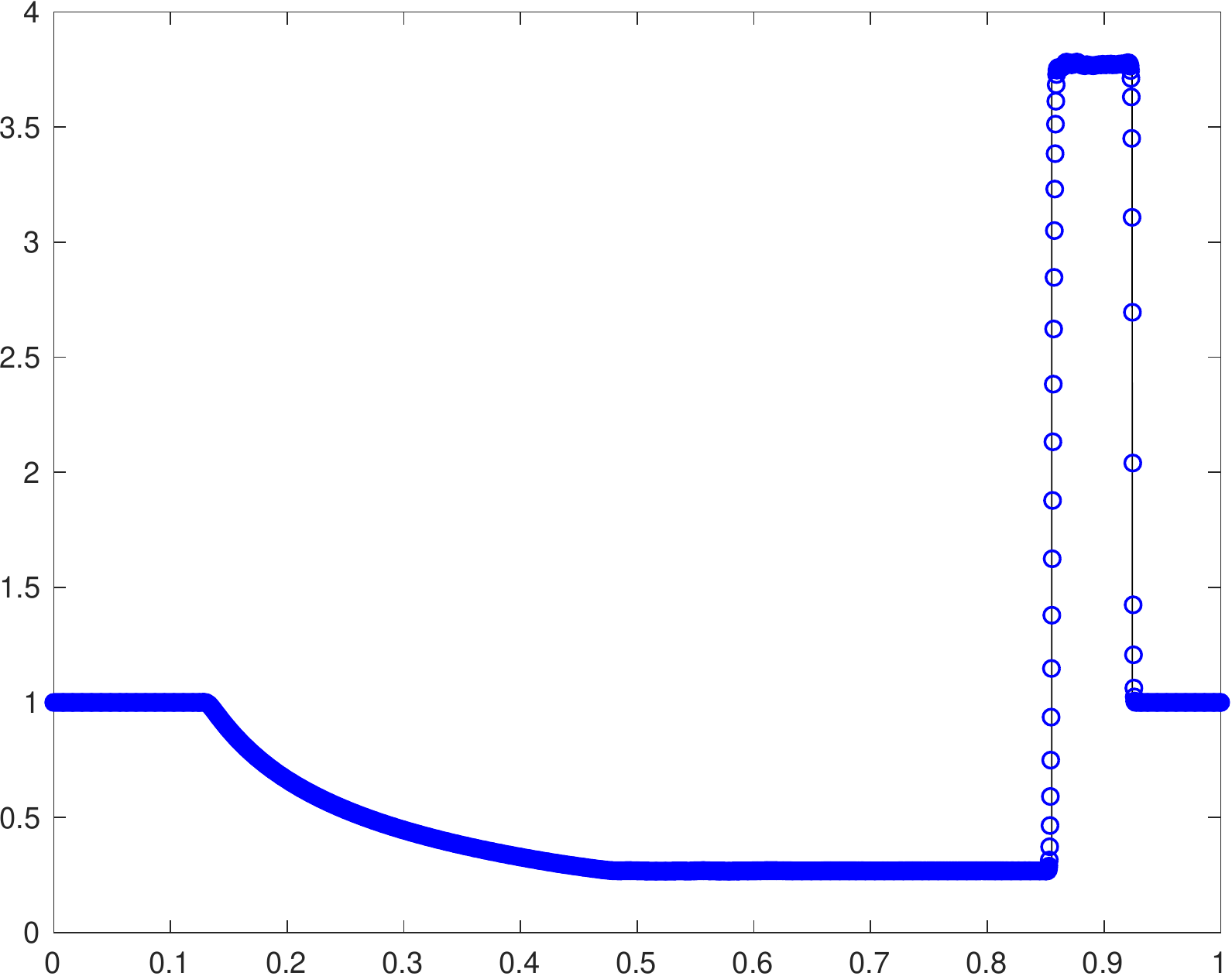}
		}
		\subfigure[$\rho^{(\zeta)}_h - \zeta \rho_h$ with scaling-invariant weights \eqref{OUR-weights}]{
			\includegraphics[width=0.44\textwidth]{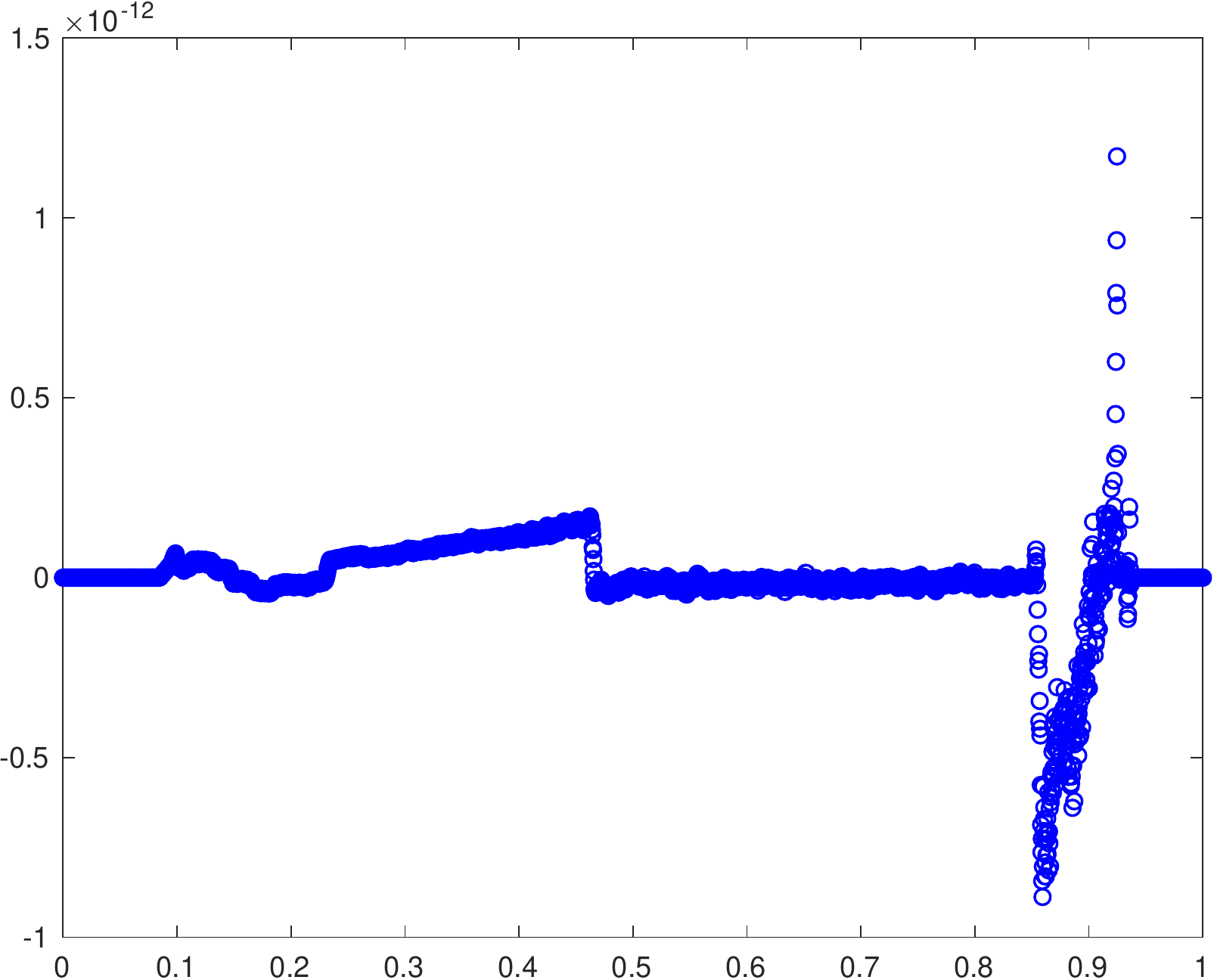}
		}
		\subfigure[$\rho^{(\zeta)}_h$ with weights \eqref{ZQweights} from \cite{zhu2018new}]{
			\includegraphics[width=0.44\textwidth]{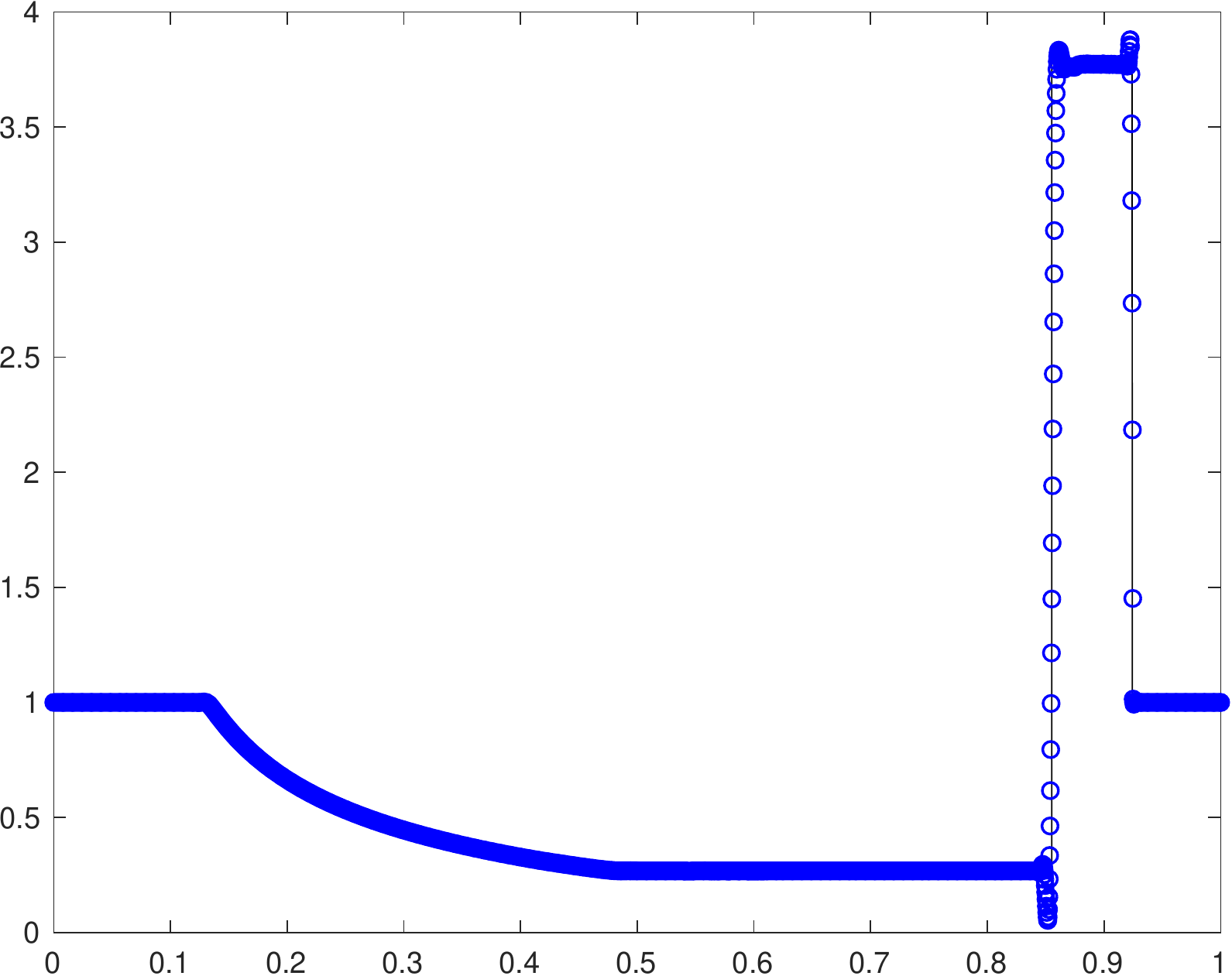}
			\label{scalingRHO}
		}
		\subfigure[$\rho^{(\zeta)}_h - \zeta \rho_h$ with weights \eqref{ZQweights} from \cite{zhu2018new}]{
			\includegraphics[width=0.44\textwidth]{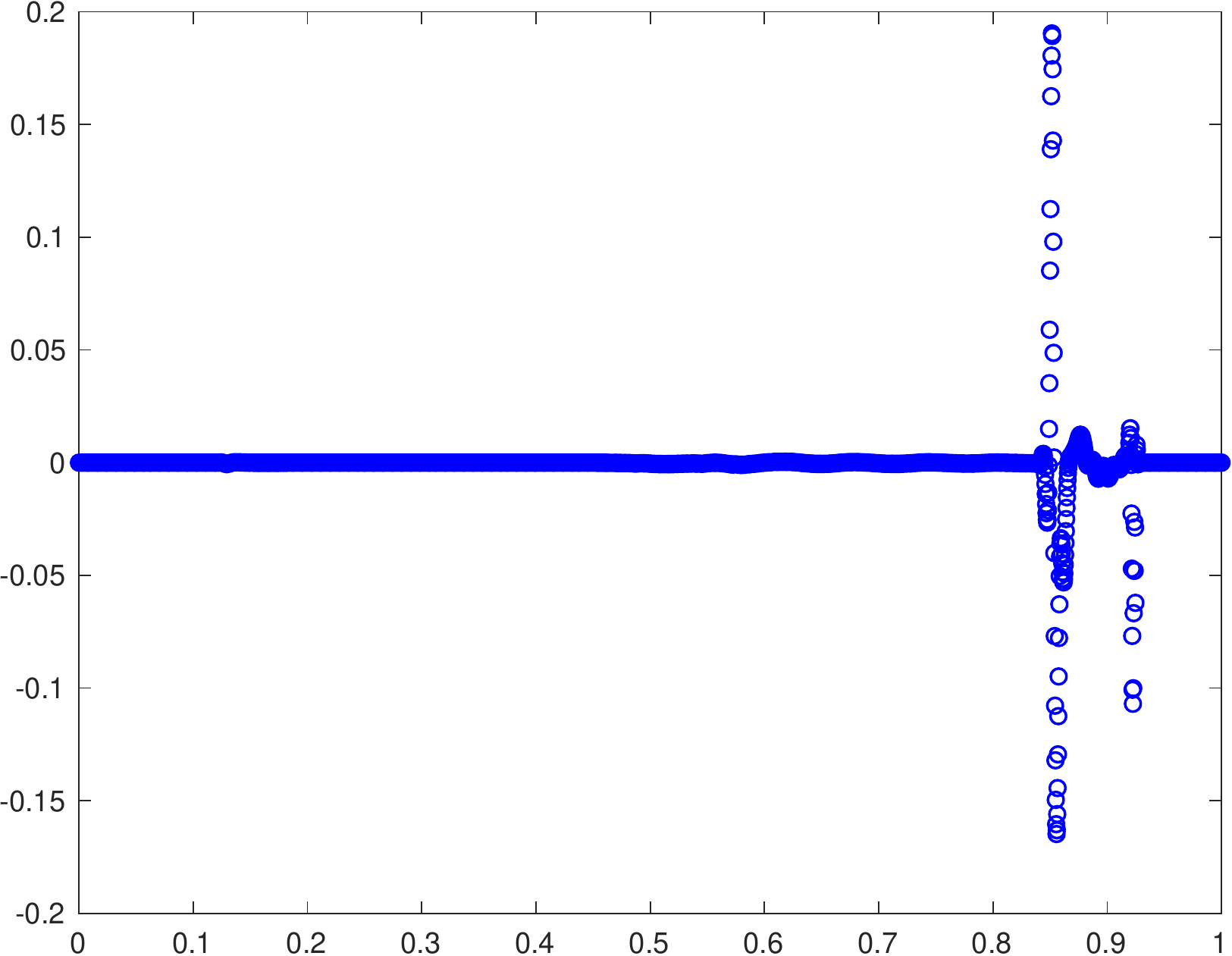}
		}
		\caption{Example \ref{1Driemann3}: The numerical solutions along along the line $y=0$ at $t=0.45$ for the scaled initial data with $\zeta=10^{-2}$. The solid lines on the left figures denote the exact solution.}
		\label{1Driemann3:fig2}
	\end{figure}
	
	\begin{figure}[htbp]
		\centering
		\subfigure[$\rho^{(\zeta)}_h$ with scaling-invariant weights \eqref{OUR-weights} ]{
			\includegraphics[width=0.45\textwidth]{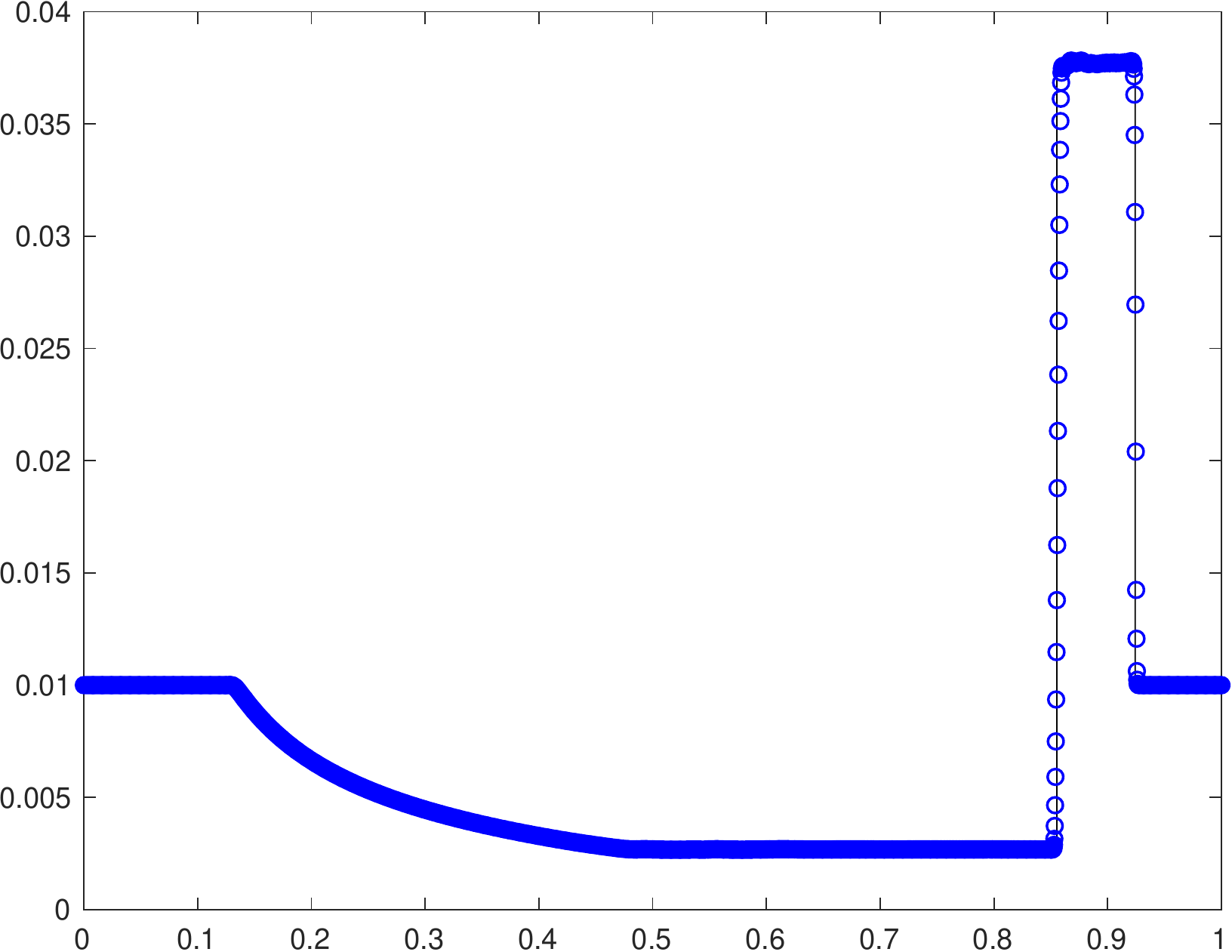}
		}
		\subfigure[$\rho^{(\zeta)}_h - \zeta \rho_h$ with scaling-invariant weights \eqref{OUR-weights}]{
			\includegraphics[width=0.45\textwidth]{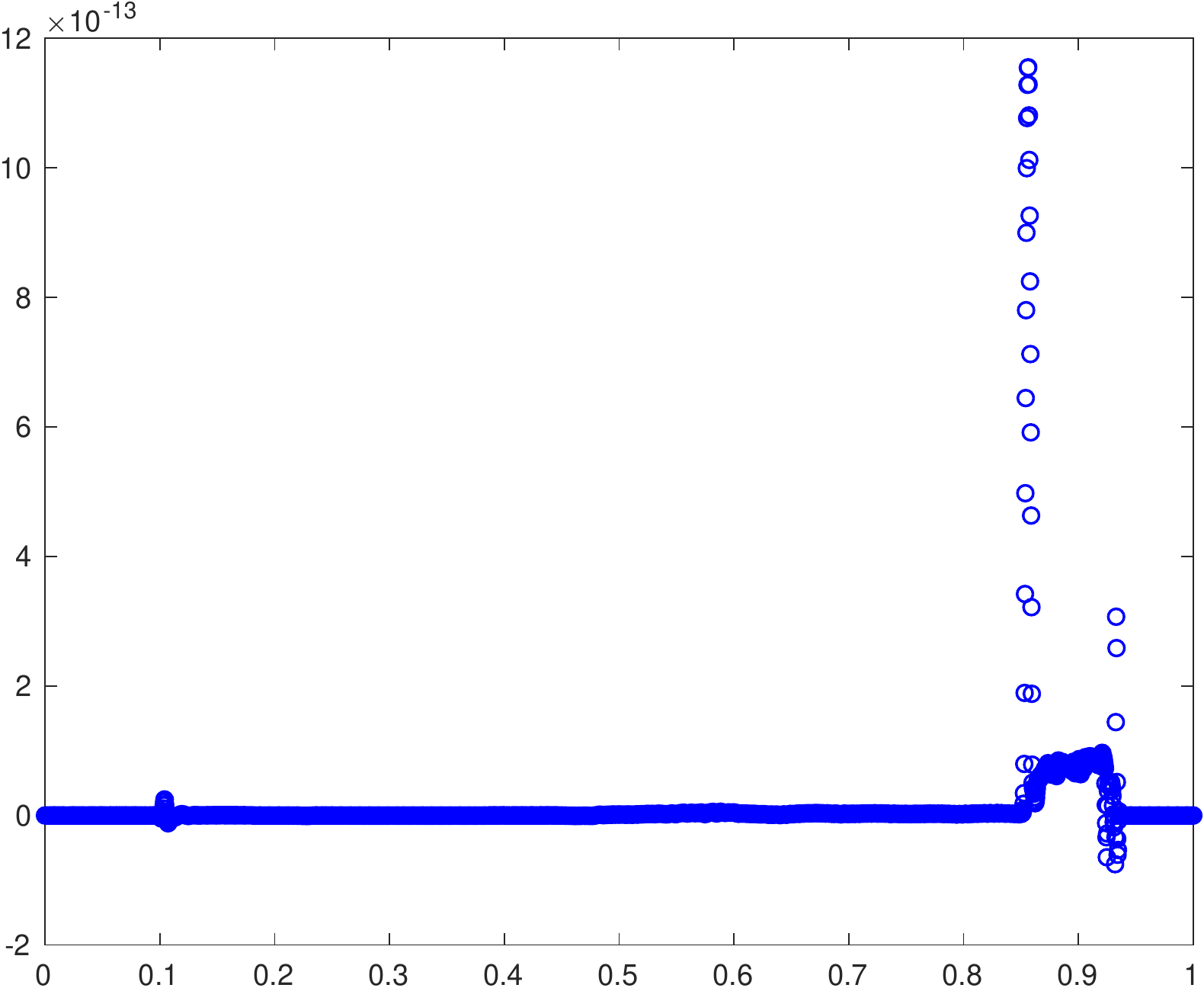}
		}
		\subfigure[$\rho^{(\zeta)}_h$ with weights \eqref{ZQweights}]{
			\includegraphics[width=0.45\textwidth]{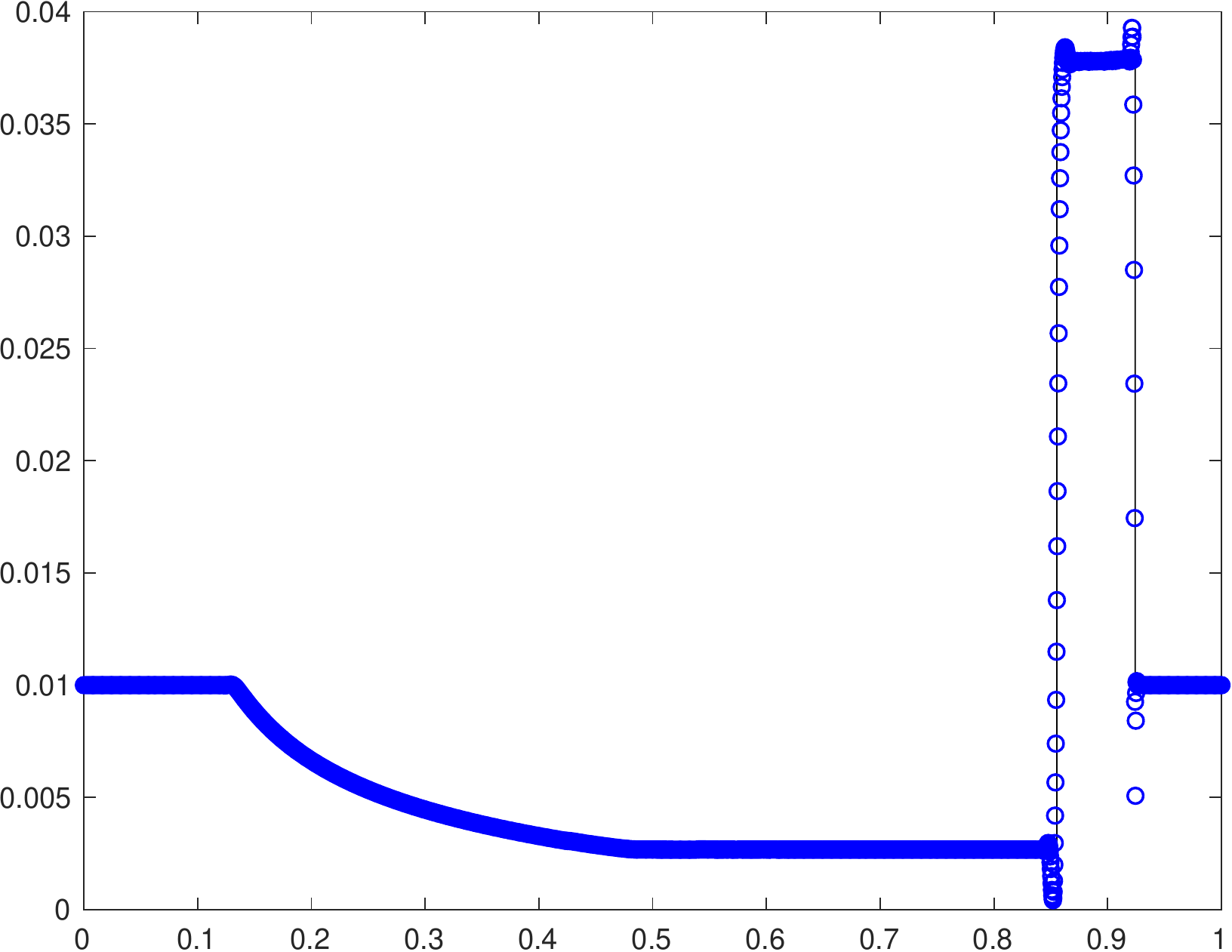}
			\label{scalingRHO2}
		}
		\subfigure[$\rho^{(\zeta)}_h - \zeta \rho_h$ with weights \eqref{ZQweights}]{
			\includegraphics[width=0.45\textwidth]{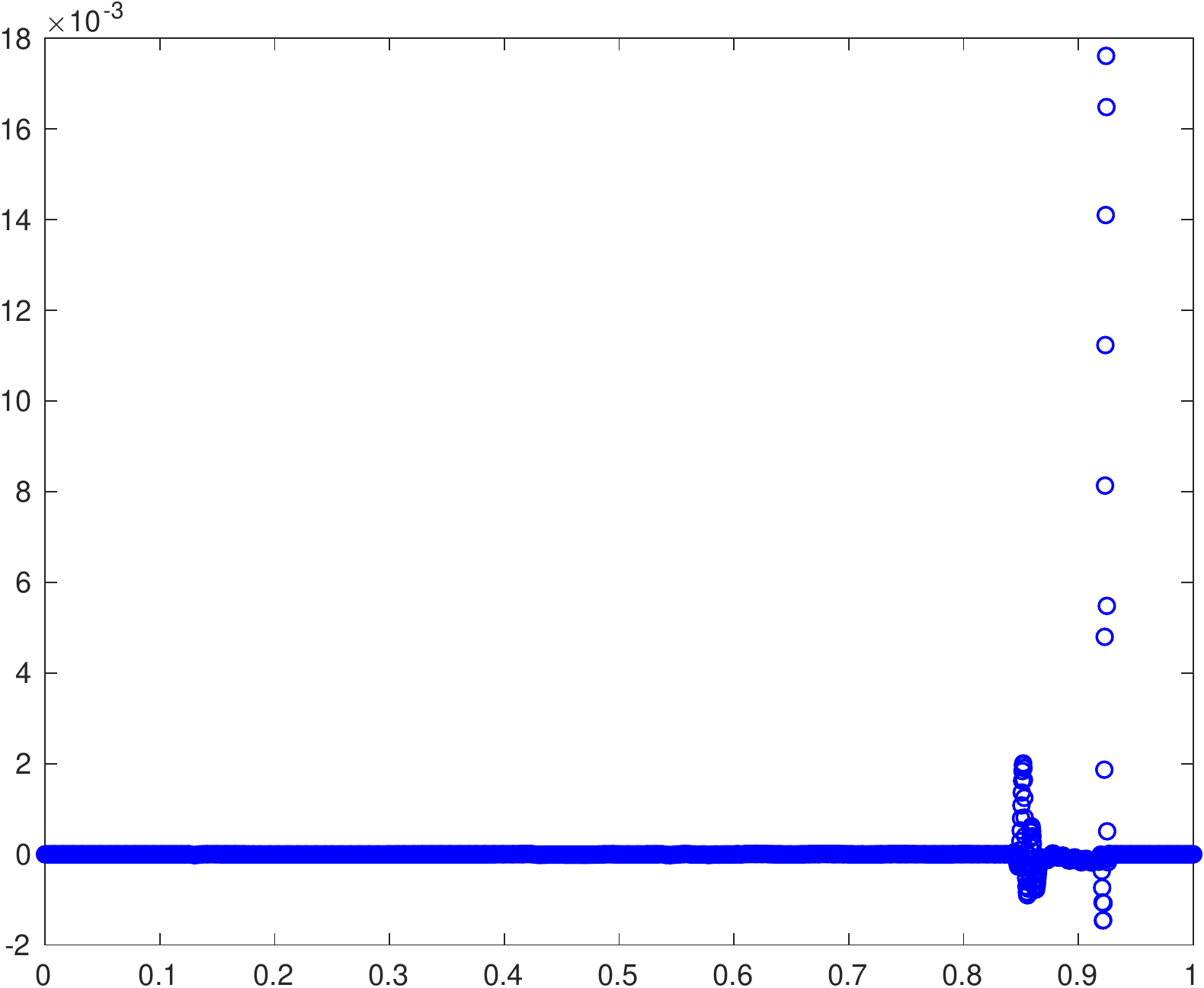}
			\label{scalingV2}
		}
		\caption{Example \ref{1Driemann3}: Numerical results along the line $y=0$ at $t=0.45$ for the scaled initial data with $\zeta=10^{-4}$. The solid lines on the left figures denote the exact solution.}
		\label{1Driemann3:fig3}
	\end{figure}

\end{expl}

\begin{expl}[Quasi-1D multi-scale problem]	       \label{1DMSP}\rm 
		To further demonstrate the importance and advantages of using our scaling-invariant weights \eqref{OUR-weights}, we simulate a problem involves the interaction of multi-scale waves. 
		The initial data are taken as
		\begin{equation}\label{eq:1DR3-2}
			(\rho_0, {\bm v}_0,p_0)=
			\begin{cases}
				(100,0,0,10^4), &0<x<0.5,\\
				(100,0,0,100), &0.5<x<1,\\
				(1,0,0,100), &1<x<1.5,\\
				(1,0,0,1), &1.5<x<2,
			\end{cases}
		\end{equation}
		which  combine the initial solution in \eqref{eq:1DR3} and its scaled case with $\zeta=10^{-2}$ in the computational domain $[0,2]\times[-1/320,1/320]$. We use a triangular mesh with $h=1/1600$ and the outflow boundary conditions. Fig.~\ref{twowave} shows the density in the logarithmic scale along the line $y=0$  
		 at $t=0.45$ obtained by using respectively 
		 the scaling-invariant nonlinear weights \eqref{OUR-weights} and the non-scaling-invariant nonlinear weights \eqref{ZQweights}.  One can see that the numerical solution with the non-scaling-invariant weights  \eqref{ZQweights} has obvious undershoots in resolving the stationary contact discontinuity at $x=1$ and right-moving contact discontinuity near $x \approx 1.85$ for this multi-scale problem . 
		 Whereas, the numerical solution with the proposed scaling-invariant weights \eqref{OUR-weights} does not suffer from this issue. 
		 This further indicates that the scaling-invariant weights \eqref{OUR-weights} 
		 may be more robust and advantageous in capturing wave structures with different scales.

	\begin{figure}[htbp]
	\centering
	\subfigure[Computed with our scaling-invariant weights \eqref{OUR-weights} ]{
		\includegraphics[width=0.9\textwidth]{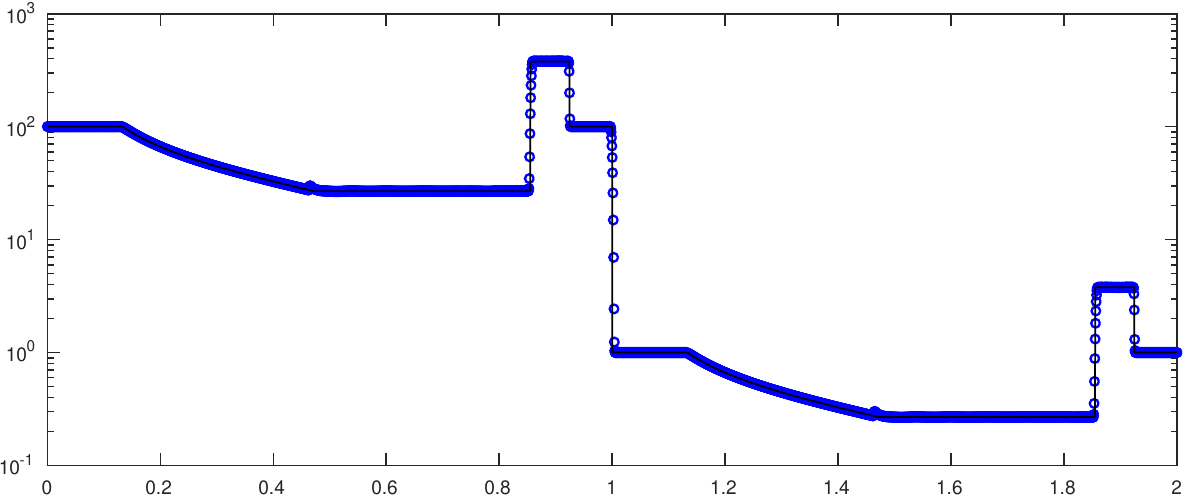} 
		\label{twowave_new}
	}
	\subfigure[Computed with the nonlinear weights \eqref{ZQweights}]{
		\includegraphics[width=0.9\textwidth]{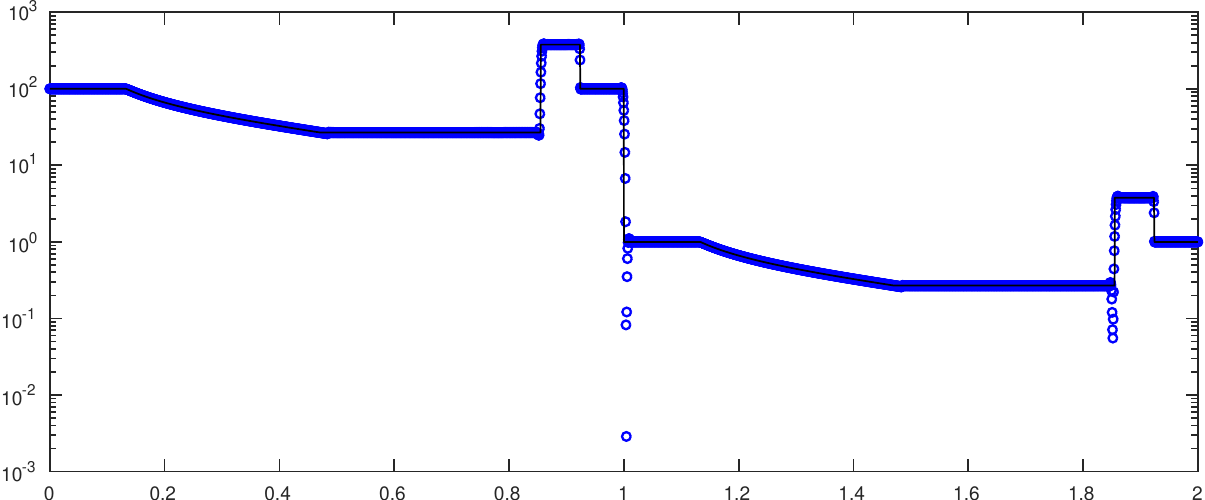} 
		\label{twowave_old}
	}
	\caption{Example \ref{1Driemann3}: The numerical solution (symbols ``$\circ$'') and exact solution (solid lines) of density plotted in logarithmic scale along the line $y=0$ at $t=0.45$ obtained by using two different nonlinear weights in the WENO method. A triangular mesh with $h=1/1600$ is used. }
	\label{twowave}
\end{figure}

\end{expl}

\begin{expl}[2D Riemann problem \uppercase\expandafter{\romannumeral1}]\label{2Driemann1}\rm
 Both this and the next examples   
simulate  2D Riemann problems of the ideal relativistic fluid within the domain $[0,1]^2$, all on the same unstructured triangular mesh with $h=1/400$. 
\figref{fig:2Driemann1}\subref{fig:2Driemann1b}  shows a sample mesh (with $h=1/20$) which is coarser than the computational mesh. 

The initial condition of this example are taken as
	\begin{equation*}
		(\rho_0, {\bm v}_0,p_0)=
		\begin{cases}
			(0.1,0,0,0.01), &x>0.5,y>0.5,\\
			(0.1,0.99,0,1), &x<0.5,y>0.5,\\
			(0.5,0,0,1),      &x<0.5,y<0.5,\\
			(0.1,0,0.99,1), &x>0.5,y<0.5.
		\end{cases}
	\end{equation*}
	  \figref{fig:2Driemann1}\subref{fig:2Driemann1a} gives the contours of the density logarithm $\ln \rho$ at $t=0.4$. Due to the interaction of the initial four discontinuities, two reflected curved shock waves and a complex mushroom structure are formed. The result is in agreement with that in \cite{WuTang2015}. We observe that, at the beginning of the simulation, the number of PCP limited cells is about 100 which is only  $\sim 0.3\%$ of the total cell numbers. As the time increases, the number finally drops to $3$. However, if the PCP limiter is not used, our code for this example would blow up  due to the nonphysically numerical solutions. 
	\begin{figure}[htbp]
		\centering
		\subfigure[A sample mesh with $h=1/20$]
		{
			\includegraphics[width=0.45\textwidth]{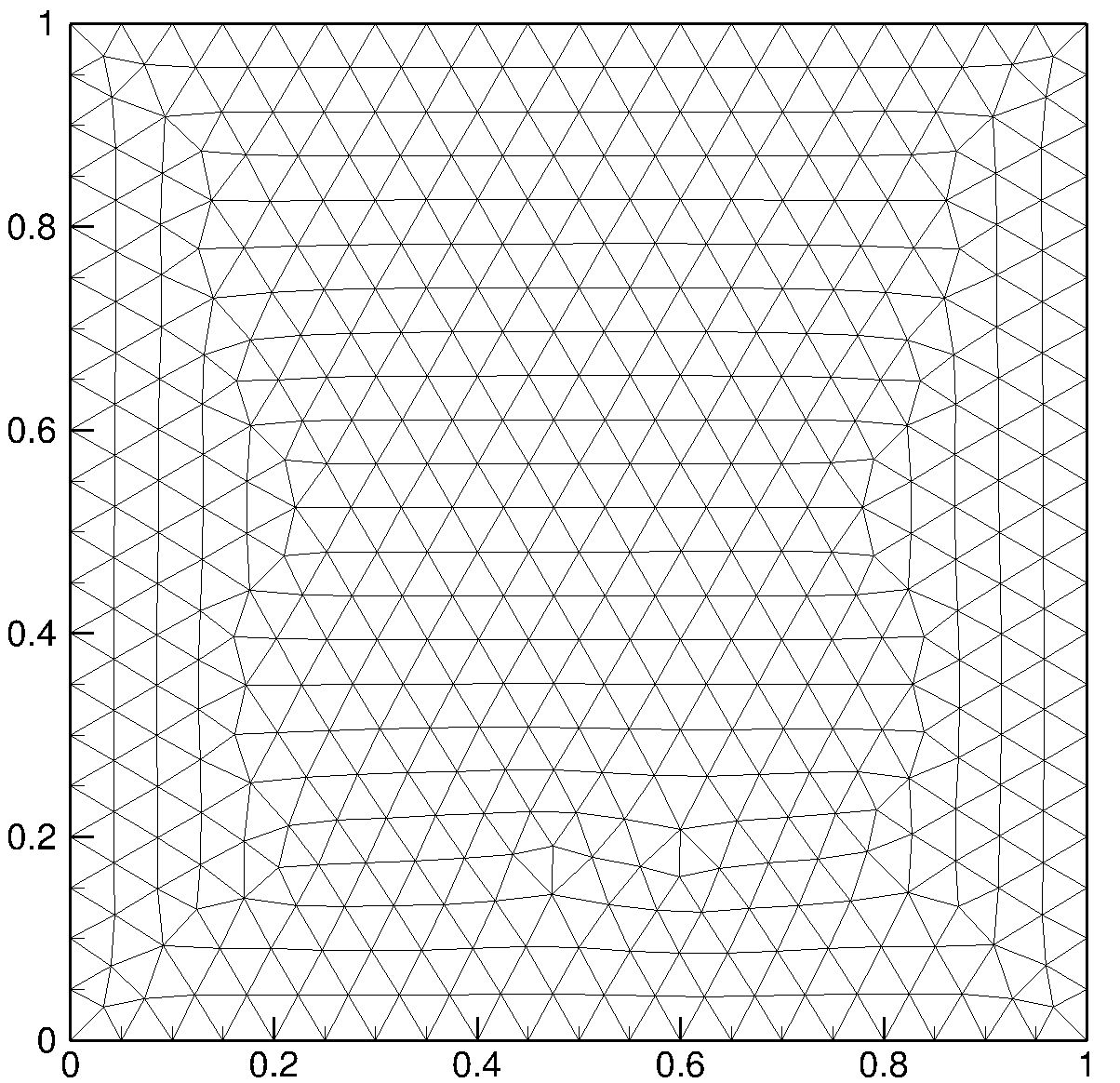}
			\label{fig:2Driemann1b}
		}
		\subfigure[$\ln \rho $]{
			\includegraphics[width=0.45\textwidth]{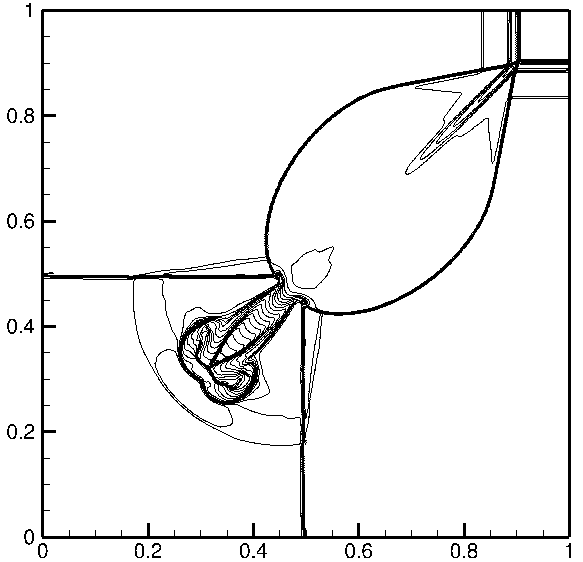}
			\label{fig:2Driemann1a}
			}
		\caption{Example \ref{2Driemann1}: The density logarithm $\ln\rho$ with 25 equally spaced contour lines from $-6$ to $1.87$ at $t=0.4$.  A triangular mesh with $h=1/400$ is used. }
		\label{fig:2Driemann1}
	\end{figure}
\end{expl}

\begin{expl}[2D Riemann problem \uppercase\expandafter{\romannumeral2}]\label{2Driemann1b}\rm
This is a more ultra-relativistic 2D Riemann problem 
first proposed in \cite{WuTang2015}, with the initial data 
	\begin{equation}
		(\rho_0, {\bm v}_0,p_0)=
		\begin{cases}
			(0.1,0,0,20), &x>0.5,y>0.5,\\
			(0.00414329639576, 0.9946418833556542, 0, 0.05), &x<0.5,y>0.5,\\
			(0.01,0,0,0.05),      &x<0.5,y<0.5,\\
			(0.00414329639576, 0, 0.9946418833556542, 0.05), &x>0.5,y<0.5.
		\end{cases}
	\end{equation}
  In comparison with the 2D Riemann problem \uppercase\expandafter{\romannumeral1}, the fluid velocity here is closer to the speed of light.  We use the  same computational domain and mesh as in the 2D Riemann problem \uppercase\expandafter{\romannumeral1}. We find that for this challenging example, it is also necessary to 
  employ the PCP limiter to enforce the numerical solutions in the set $\mathbb{G}_h^k$, otherwise 
  the physical constraints \eqref{Gw} would be violated and the simulation immediately breaks down in the first time step. 
  From the 25 equally spaced contour lines of density logarithm $\ln \rho$ at $t=0.4$ shown in  \figref{fig:2Driemann12}, it can be seen that the initial four discontinuities interact with each other near the central point $(0.5,0.5)$ and form a complex mushroom structure moving to the left-bottom region. Besides, two shock waves (right and top) move at a speed of $0.66525606186639$, and two contact discontinuities (bottom and left) are stationary.
 
	\begin{figure}[htbp]
		\centering
		\includegraphics[width=0.5\textwidth]{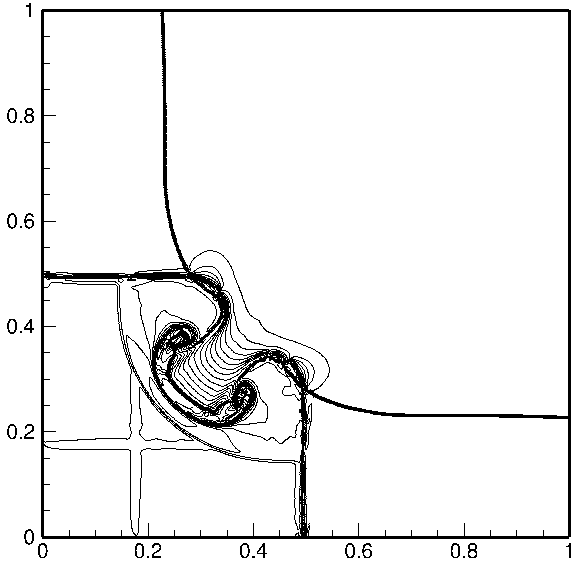}
		% \subfigure[sample mesh]{
		%   \includegraphics[width=0.45\textwidth,trim=50 50 50 50,clip]{results/V3_2D_riemann4/riemann_samplemesh.png}
		% }
		\caption{Example \ref{2Driemann1b}: The density logarithm $\ln\rho$ with 25 equally spaced contour lines from $-9.97$ to 2.56 at $t=0.4$. A triangular mesh with $h=1/400$ is used.}
		\label{fig:2Driemann12}
	\end{figure}
\end{expl}
	
	\begin{remark}\label{Rem:Compare} 
		We take several 1D and 2D examples to compare the efficiency of our proposed three convergence-guaranteed iterative algorithms 
		and Newton's algorithm for recovering the primitive variables from admissible conservative variables. For fairness, the error tolerance $\varepsilon_{tol}$ is set as $10^{-15}$ for all these algorithms, and 
		all our experiments are performed with one core on the same Linux environment of server with Intel(R) Core(TM) i7-8700K CPU @ 3.70GHz. 
		For Newton's algorithm to recover the pressure, we take the pressure at the last time step on each cell as the initial guess in the present time step. When negative pressure is produced in 
		Newton's iteration, we restart the iteration with zero as the initial guess, which works for all the tested examples. 
		Table \ref{tablespeed} shows the computational time spent on running these algorithms 
		 (numerator), the total CPU time for the whole simulation (denominator), and the corresponding percentage (quotient) for Examples 4.2, 4.3, 4.6, and 4.7. 
		 It can be seen that our hybrid iteration algorithm is the most efficient one among these four algorithms. 
		  We also observe that our three convergence-guaranteed iterative algorithms never fail and always safely recover the primitive variables in the physical region $G_w$, which is consistent with our theoretical analysis.
	
	\begin{table}[htbp]
		\centering
		\setlength{\abovecaptionskip}{0.cm}
		\setlength{\belowcaptionskip}{-0.cm}
		\caption{The percentage (quotient) of the CPU time (only one core is used) spent on
			recovering primitive variables (numerator) compared to the whole simulation time (denominator).}\label{tablespeed}
		\begin{center}
			\resizebox{1.02\columnwidth}{!}{
				\begin{tabular}{c|c|c|c|c}					
					\toprule
					Examples \& mesh sizes &  Bisection& Fixed-point& Hybrid  &  Newton \\
					\midrule
					Example 4.2 ($h=1/{500}$)& 17.1\% = $\frac{2m55s }{17m2s }$&   14.0\%=$\frac{2m19s}{16m32s}$ & {\bf 13.9\%} =$\frac{2m17s}{16m24s}$ &18.7\%=$\frac{3m15s}{17m23s}$\\
					\midrule
					Example 4.3 ($h={1}/{1600}$)&15.4\% = $\frac{34m23s}{3h42m46s}$ &15.2\%=$\frac{33m43s}{3h41m59s}$ &{\bf 11.1\%}=$\frac{23m11s}{3h29m53s}$&
					23.0\%=$\frac{56m30s}{4h5m26s}$ \\
					\midrule
					Example 4.6 ($h=1/400$)&14.6\% = $\frac{1h19m15s}{9h3m34s}$&27.4\%=$\frac{2h59m6s}{10h53m42s}$& {\bf 10.7\%}=$\frac{55m22s}{8h35m27s}$& 15.4\% =$\frac{1h24m36s}{9h9m55s}$\\
					\midrule
					Example 4.7 ($h=1/400$)&15.4\%=$\frac{1h24m19s}{9h6m51s}$ &19.2\%=$\frac{1h51m12s}{9h39m20s}$ & {\bf 8.7\%}=$\frac{41m46s}{7h59m45s}$ &
					18.2\%=$\frac{1h43m24s}{9h29m46s}$\\
					\bottomrule
				\end{tabular}
			}
		\end{center}
	\end{table}
    \end{remark}

\begin{expl}[Double Mach reflection]\label{doublemach}\rm
	This test problem was firstly proposed by Woodward and Colella \cite{woodward1984numerical} in the non-relativistic case, and later extended to the special RHD in \cite{2006Zhang}. It was originally used to study the reflections of planar shocks in the air from wedges which is experimentally set up by driving a shock down a tube that contains a wedge \cite{woodward1984numerical}.

	To facilitate the setting of boundary conditions on the structured mesh of numerical simulations, an equivalent rotation problem is usually solved, which will make the wall horizontal and the shock wave enters it at an angle of $60^{\circ}$. Since unstructured meshes can easily handle complex domains, here we directly solve the original problem without rotation.  \figref{fig:doublemach}\subref{fig:doublemachmesh} illustrates the computational domain with a sample mesh ($h=1/10$). 
	Our computational mesh with $h=1/240$ is much finer than the sample mesh.
	The ratio of specific heats in the equation of state will be taken as $\Gamma=1.4$ in this example. 
	
	Initially at $x=0$, there is a shock wave moving horizontally to the right with an initial speed of $0.4984$, and the primitive variables $\vec{W}=(\rho_0, {\bm v}_0,p_0)^\top$ on the left and right side of the shock are given by
	\begin{align*}
		&\vec{W}_L =( 8.564,0.4247,-0.4247,0.3808)^\top,\\
		&\vec {W}_R= (1.4,0,0,0.0025)^\top.
	\end{align*}
   The states at 
   both of the bottom boundary $y=0$ and the left boundary $x=-0.1$ are set as the left shock state $\vec{W}_L$, while for the right boundary $x=2.7$, the right shock state $\vec{W}_R$ is specified.  The states on the top boundary are given according to the location of the moving shock. At the wall, the reflective boundary condition is specified. The numerical results at $t = 4$, obtained by our scheme, are shown in \figref{fig:doublemach}\subref{fig:doublemachrho}.
 It can be seen that our scheme can clearly capture the correct flow patterns, including two Mach stems, two contact waves, reflected shock, and jet formed near the wall around the double Mach region. 
 In the first three time steps, there are about one or two cells where the PCP limiter must be applied to preserve the numerical solutions in the set $\mathbb{G}^k_h$.

	\begin{figure}[htbp]
		\centering
		\subfigure[A sample mesh with $h=1/10$]{
			\label{fig:doublemachmesh}
			\includegraphics[width=0.48\textwidth]{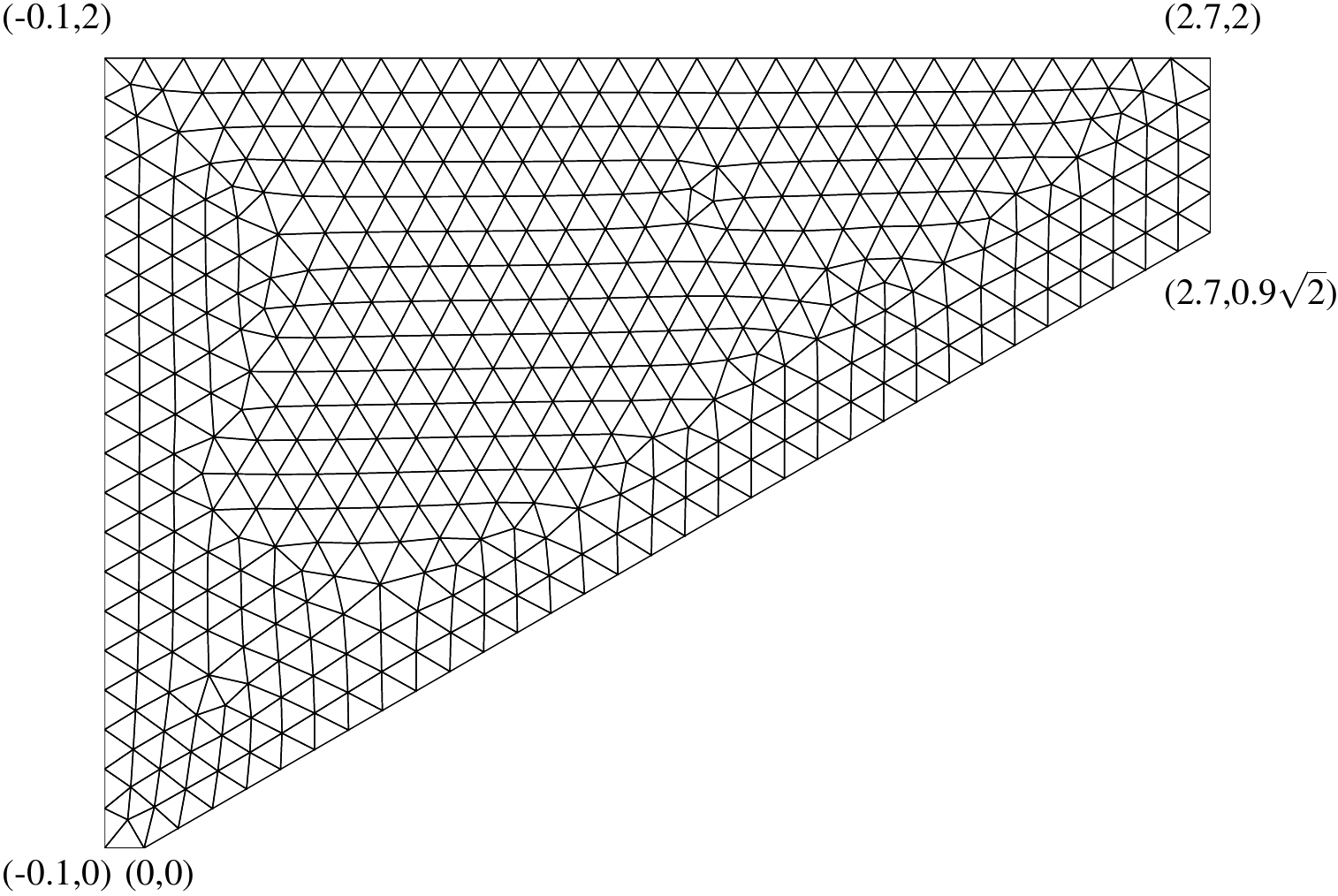}
		} %trim 上 左 下 右
		\subfigure[$ \rho $]{
			\label{fig:doublemachrho}
			\includegraphics[width=0.48\textwidth]{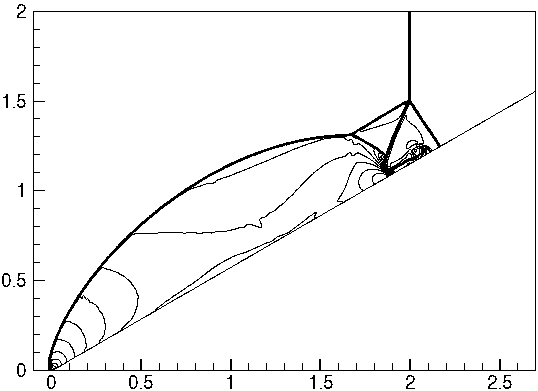}
		}
		\caption{Example \ref{doublemach}: The density contour map with 30 equidistant contour lines. 
		A triangular mesh with $h=1/240$ is used.}
		\label{fig:doublemach}
	\end{figure}
\end{expl}

\begin{expl}[Relativistic forward-facing step problems]\label{forwardstep}\rm
	This example simulates the flow in a wind tunnel over a forward-facing step. The problem has been studied in both classic fluid dynamics \cite{2008chen} and relativistic fluid dynamics \cite{2006Zhang,Lucas2004}.
	The computational domain $[0,3]\times[0,1]$ is shown in  \figref{fig:forwardstep}\subref{fig:forwardstepmesh} with a sample mesh ($h=1/10$). The domain has a step of height $0.2$ at a distance of $0.6$ from the left boundary. 
	In our computations, we divide the domain 
	 into triangular cells, with $h=1/160$ and 149436 total cells.
	Initially the domain is filled with a right-moving fluid with $\Gamma=1.4$, which everywhere has a rest-mass density of $\rho = 1.4$ and a Newtonian Mach number of $3.0$. Along the walls of tunnel and step, the reflective boundary conditions are applied. 
	The boundary condition at the right is outflow, and at the left is inflow. 
		
	We simulate this problem with three different configurations with 
		different initial velocities $v_1 = 0.9$, $0.99$, and $0.999$, respectively. 
		The end times of simulations for these three configurations are 
		 $t=6,$ $4.45$, and $4$, respectively. 
		 Figs.~\ref{fig:forwardstep}\subref{fig:v9}, \ref{fig:forwardstep}\subref{fig:v99}, and \ref{fig:forwardstep}\subref{fig:v999} show the snapshots for each configuration at the final time obtained by our PCP finite volume scheme. 
		 The flow structures, including the shock reflection patterns, are very similar for the three setups. 
		 We observe that, when hitting the step, the fluid is reflected by the step to form a bow-shaped shock wave. Afterwards, the bow shock wave collides with the top boundary. 
		 %When the reflected shock wave hits the step, a second curved reflected shock wave is generated. 
	The results show that the bow shock moves faster as the inflow velocity $v_1$ is set larger. Besides, our method 
	captures all 
	 the wave structures correctly and robustly,  without any special artificial entropy fix near the step corner. 
 	It is noticed that the PCP limiter is necessary for all three configurations to keep the numerical solution in the set $\mathbb{G}_h^k$, otherwise the code would blow up quickly. 
 
	\begin{figure}[htbp]
		\centering
		\subfigure[A sample mesh with $h=1/10$]{  
			\label{fig:forwardstepmesh}
			\includegraphics[width=0.49\textwidth]{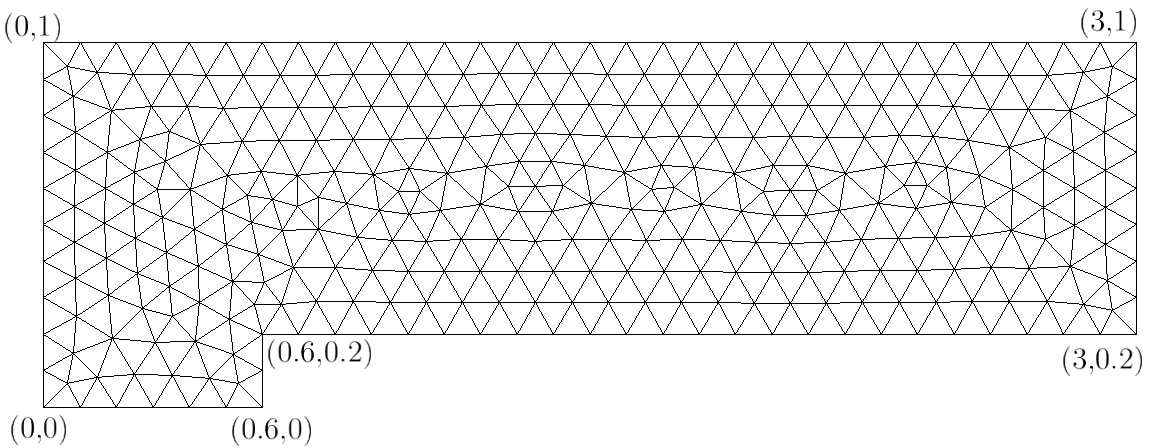}
		}%
		\subfigure[The first configuration at $t=6 $]{
			\label{fig:v9}
			\includegraphics[width=0.49\textwidth]{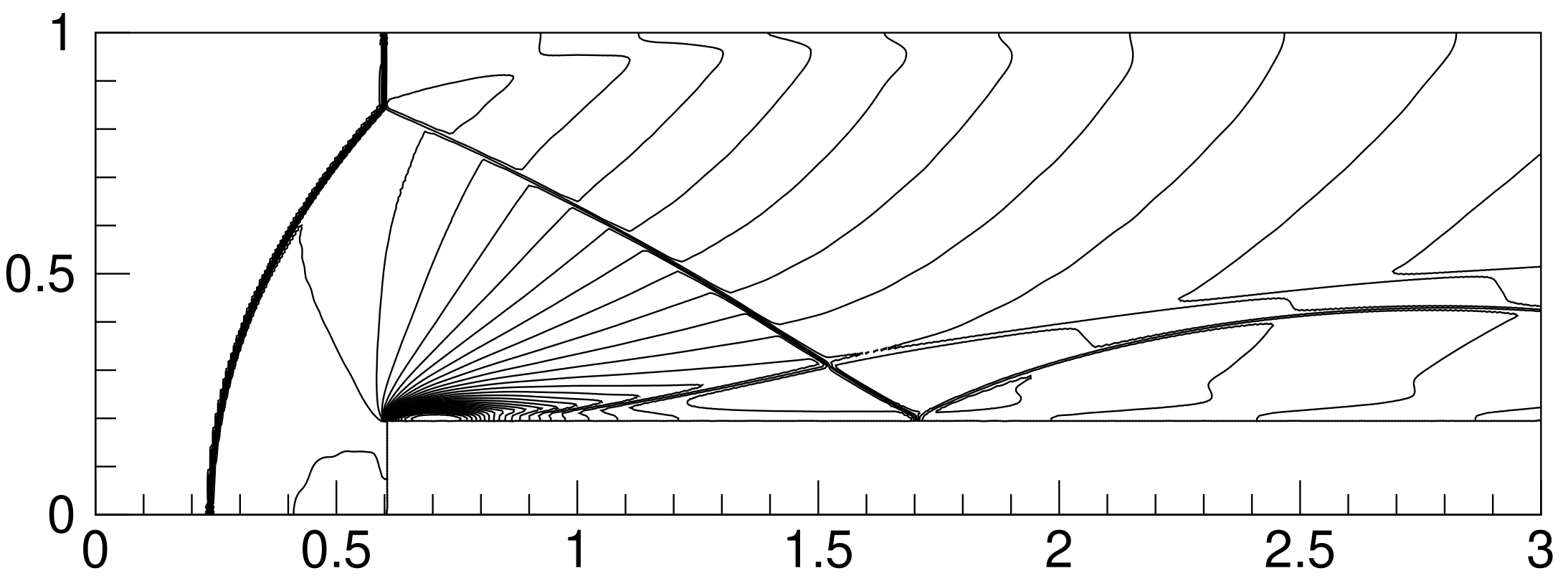}
		}
		
		\subfigure[The second configuration at $t=4.45 $]{
			\label{fig:v99}
			\includegraphics[width=0.49\textwidth]{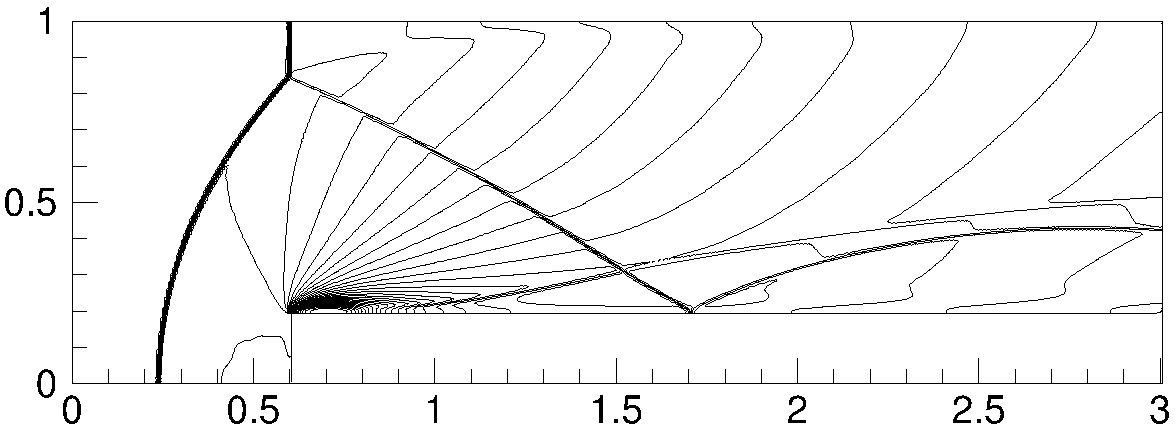}
		}%trim 上 左 下 右
		\subfigure[The third configuration at $t=4 $]{
			\label{fig:v999}
			\includegraphics[width=0.49\textwidth]{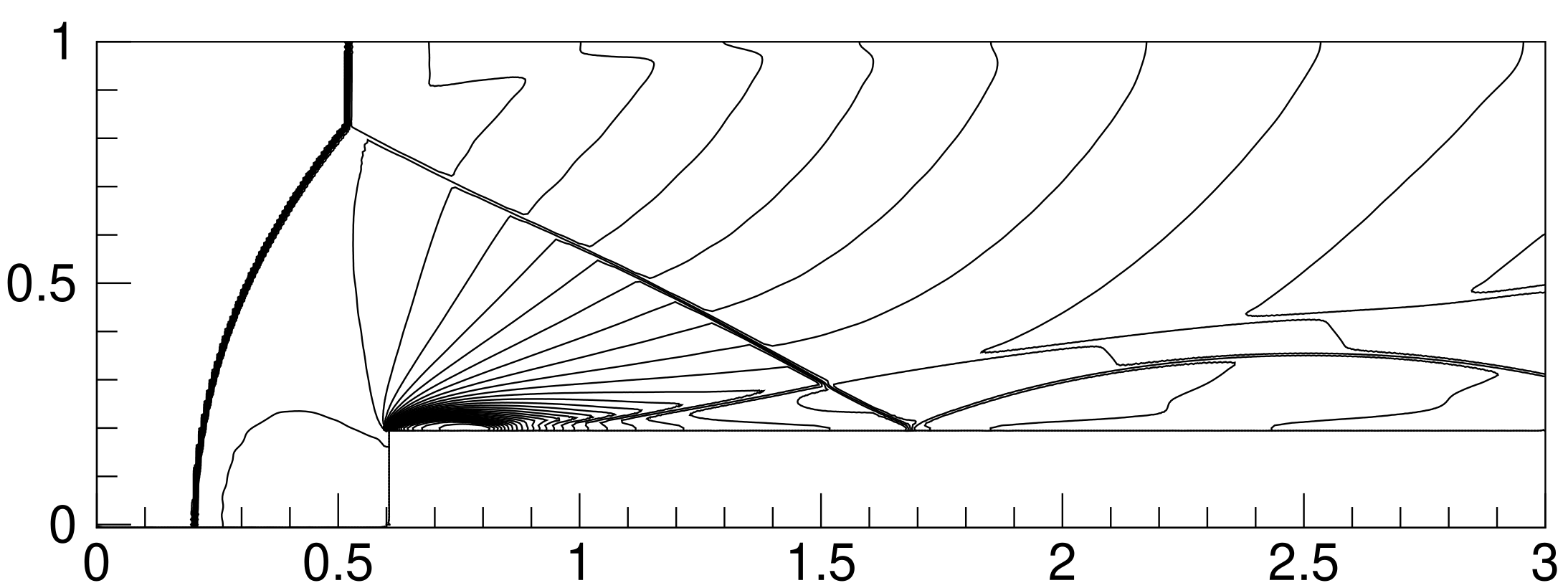}
		}
		\caption{Example \ref{forwardstep}: The logarithm of density with 30 equally spaced contours for three different configurations (namely $v_1=0.9,0.99$, and $0.999$) at the final time $t=6, 4.45$, and $4$, respectively. A triangular mesh with $h=1/160$ is used.}
		\label{fig:forwardstep}
	\end{figure}
\end{expl}

% \subsection{Example 6: Shock–vortex interaction problems}
\begin{expl}[Shock–vortex interaction problems]\label{shockvortex}\rm
	The interaction of shock and vortex  has been widely studied in classic hydrodynamics (e.g., \cite{balsara2000}) and extended to the special RHD (see, e.g., \cite{BalsaraKim2016, duan2019high}).  
   We take the velocity magnitude of the vortex as $w=0.9$, and the other parameters in the rest frame are the same as those in Example \ref{ex:accurary}. 
	The computational domain $[-17,3]\times[-5,5]$ is displayed in \figref{fig:shockvortexmesh}  with a sample triangular mesh ($h=1/2$). 
		In our computations, we divide the domain 
	into a finer mesh with $h=1/40$. We take $\Gamma = 1.4$ in this example. 
	Initially, a vortex is centered at the point $(0,0)$, and there is a standing shock at $x=-6$ far away from the vortex. 
	The pre-shock gas with constant state $(\rho_0, {\bm v}_0,p_0)=(1,-0.9,0,1)$ flows into the shock from its right side. The post-shock state 
	%can be calculated by the Rankine--Hugoniot jump condition and the Lax entropy condition,  and its value 
	is 
	\begin{equation*}
	(\rho_0, {\bm v}_0,p_0) = (4.891497310766981, -0.388882958251919, 0, 11.894863258311670).
	\end{equation*}
The outflow boundary condition is applied at the left $x=-17$, and the reflection boundary conditions are applied at both the bottom $y=-5$ and the top $y=5$ of the domain.  
We simulate two cases with different vortex intensities.

	\begin{figure}[htbp]
		\centering
		\subfigure{
			\includegraphics[width=0.7\textwidth]{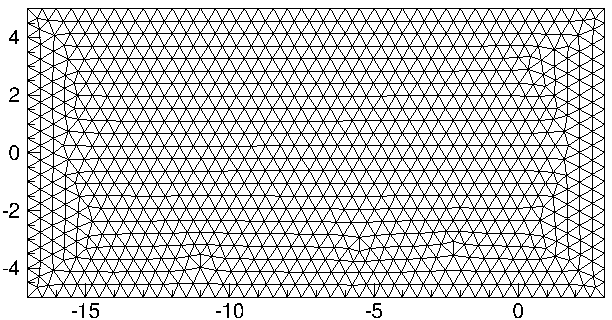}
		}
		\caption{A sample mesh with $h=1/2$. } 
		\label{fig:shockvortexmesh}
	\end{figure}	

	We first simulate the interaction of the shock with a 
	mild vortex with strength $\epsilon=5$, which is same as in \cite{duan2019high}. 
	 Fig.~\ref{fig:shockvortexcontourrho} shows 
	  the contour plots for 
	   $\log_{10}(1+|\nabla\rho|)$ and $|\nabla p|$.  
	   As seen from Figs.~\ref{fig:shockvortexcontourrho}\subref{fig:shockvortexrho0} and \ref{fig:shockvortexcontourrho}\subref{fig:shockvortexp0}, initially the vortex is elliptic due to the Lorentz contraction.  
	One can observe that when the vortex  passes through the standing shock wave, many complex wave structures are formed. 
	The proposed scheme is able to capture the shock-vortex interaction accurately, and our numerical results are consistent with those computed in \cite{duan2019high}.

In order to further demonstrate the PCP property and robustness of our method, we also investigate a more severe case with the vortex strength set as $\epsilon = 10.0828$. Our results are presented in Fig.~\ref{fig:shockvortexcontourlowrho}, from which we see that the wave structures are much more complicated than the first case with $\epsilon=5$. We remark that the proposed PCP limiter is highly desirable for this challenging test, because the pressure and density around the center of the vortex are very low. 
If the PCP limiter is turned off for this test, our high-order finite volume code would break down.
	
	\begin{figure}[htbp]
		\centering
		% \subfigure[$ \rho $]{
		%  \includegraphics[width=0.45\textwidth]{results/V3_2D_shock_vortex_lowrho/logrho_T19_01.2_50.jpeg}
		% }
		\subfigure[$\log_{10}(1+|\nabla\rho|)~ \text{at} ~t=0$]{
			\label{fig:shockvortexrho0}
			\includegraphics[width=0.45\textwidth]{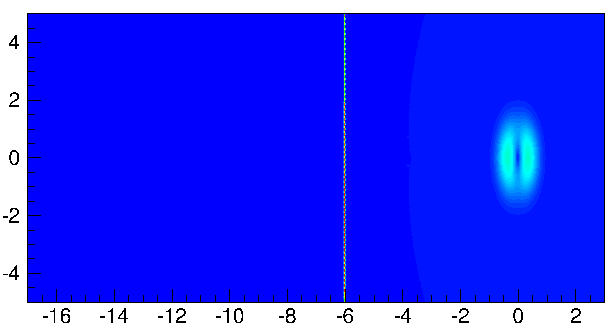}
		}
		\subfigure[$|\nabla p|~ \text{at} ~t=0$]{
			\label{fig:shockvortexp0}
			\includegraphics[width=0.45\textwidth]{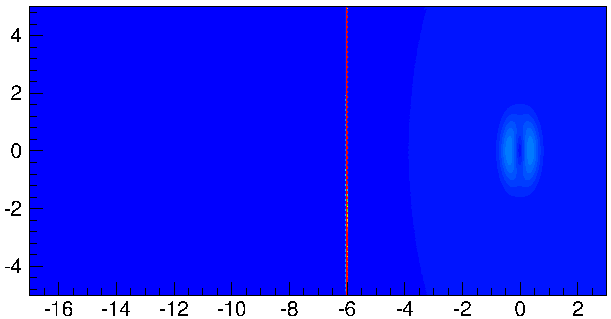}
		}
		\subfigure[$\log_{10}(1+|\nabla\rho|)~ \text{at} ~t=19$]{
			\label{fig:shockvortexrho19}
			\includegraphics[width=0.45\textwidth]{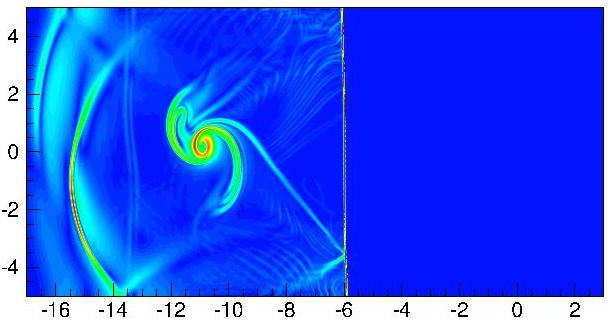}
		}
		\subfigure[$|\nabla p|~ \text{at} ~t=19$]{
			\label{fig:shockvortexp19}
			\includegraphics[width=0.45\textwidth]{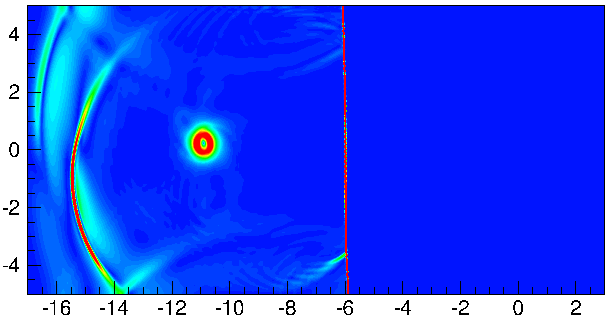}
		}
		\subfigure[Close-up~\text{of}~$\log_{10}(1+|\nabla\rho|)~ \text{at} ~t=19$]{
			\label{fig:shockvortexrhoup}
			\includegraphics[width=0.45\textwidth]{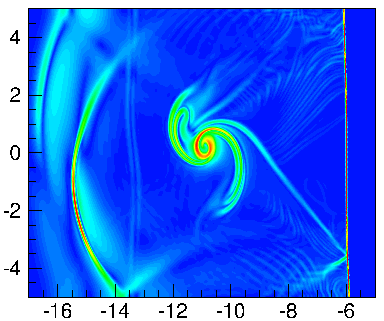}
		}
		\subfigure[Close-up~$\text{of}~|\nabla p|~ \text{at} ~t=19$]{
			\label{fig:shockvortexpup}
			\includegraphics[width=0.45\textwidth]{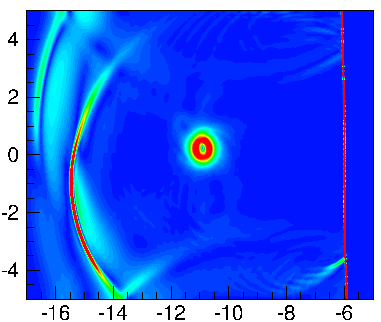}
		}
		\caption{Example \ref{shockvortex}: Snapshots has 50 contour lines equally distributed from 0 to 1 for $\log_{10}(1+|\nabla\rho|)$ (left) and  50 contour lines equally distributed from 0 to 20 for $|\nabla p|$  (right) with vortex strength $\epsilon=5$ at different time.  A triangular mesh with $h=1/40$ is used.
    	}
		\label{fig:shockvortexcontourrho}
	\end{figure}
	
	\begin{figure}[htbp]
		\centering
		% \subfigure[$ \rho $]{
		%  \includegraphics[width=0.45\textwidth]{results/V3_2D_shock_vortex_lowrho/logrho_T19_01.2_50.jpeg}
		% }
		
		\subfigure[$\log_{10}(1+|\nabla\rho|)~ \text{at} ~t=0$]{
			\label{fig:shockvortexlowrho0}
			\includegraphics[width=0.45\textwidth]{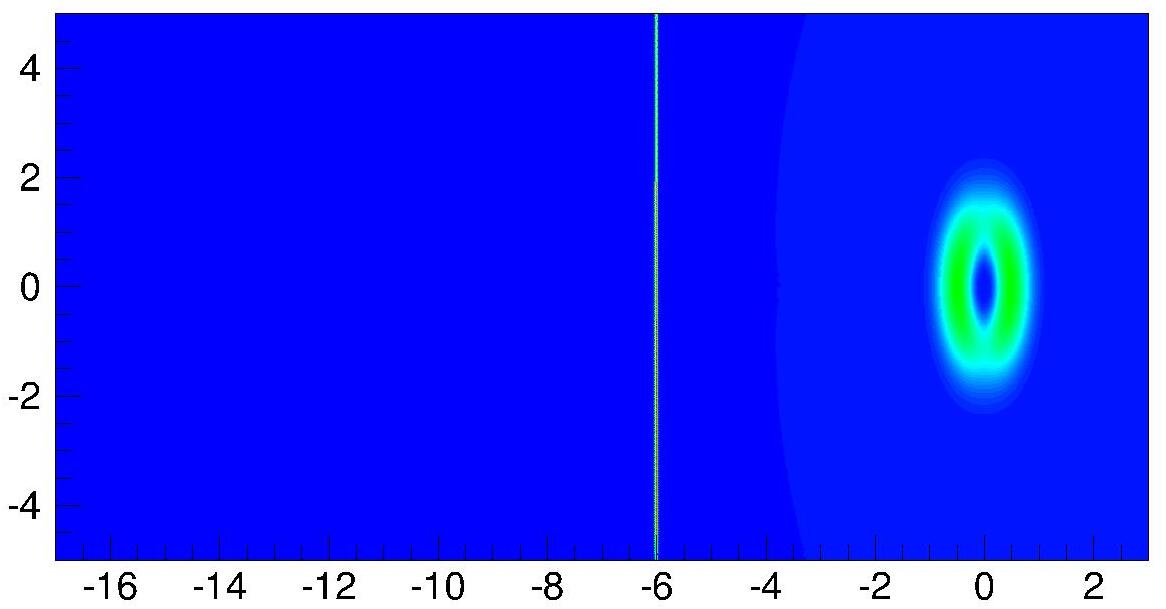}
		}
		\subfigure[$|\nabla p|~ \text{at} ~t=0$]{
			\label{fig:shockvortexlowp0}
			\includegraphics[width=0.45\textwidth]{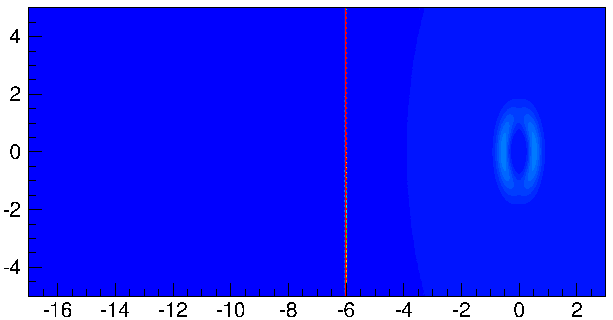}
		}
		\subfigure[$\log_{10}(1+|\nabla\rho|)~ \text{at} ~t=19$]{
			\label{fig:shockvortexlowrho19}
			\includegraphics[width=0.45\textwidth]{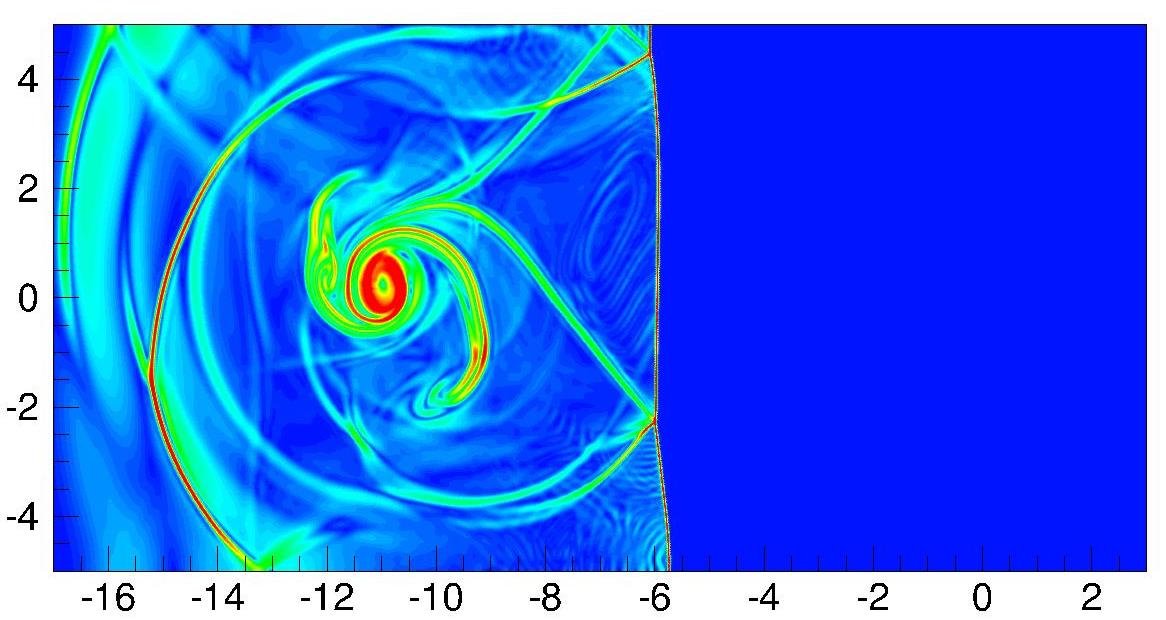}
		}
		\subfigure[$|\nabla p|~ \text{at} ~t=19$]{
			\label{fig:shockvortexlowp19}
			\includegraphics[width=0.45\textwidth]{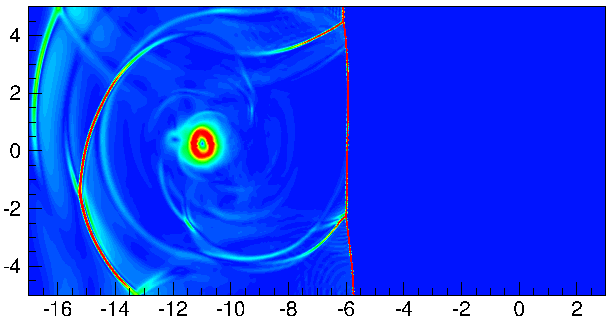}
		}
		\subfigure[Close-up~\text{of}~$\log_{10}(1+|\nabla\rho|)~ \text{at} ~t=19$]{
			\label{fig:shockvortexlowrhoup}
			\includegraphics[width=0.45\textwidth]{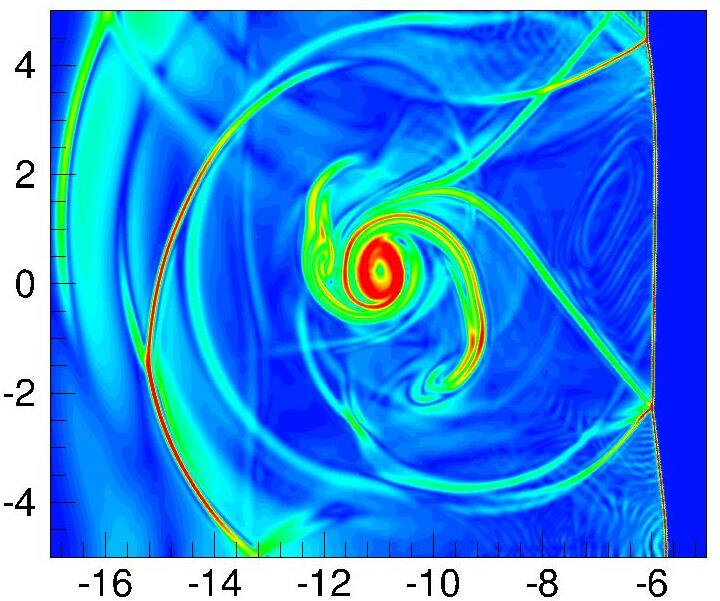}
		}
		\subfigure[Close-up~$\text{of}~|\nabla p|~ \text{at} ~t=19$]{
			\label{fig:shockvortexlowpup}
			\includegraphics[width=0.45\textwidth]{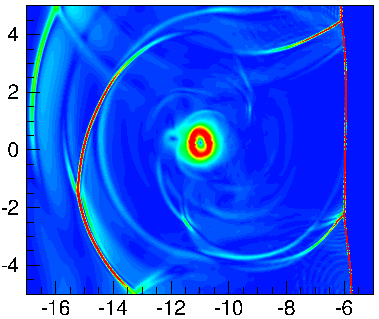}
		}
		\caption{Example \ref{shockvortex}: Snapshots has 50 contour lines equally distributed from 0 to 1 for $\log_{10}(1+|\nabla\rho|)$ (left) and  50 contour lines equally distributed from 0 to 20 for $|\nabla p|$  (right) with vortex strength $\epsilon=10.0828$ at different time.  A triangular mesh with $h=1/40$ is used.
		}
		\label{fig:shockvortexcontourlowrho}
	\end{figure}
	
\end{expl}

\begin{expl}[Shock diffraction problem]\label{shockdiffraction}\rm
	Shock wave diffraction at a sharp corner is a benchmark problem 
	to test numerical schemes on unstructured meshes in non-relativistic fluid dynamics; see, e.g., \cite{zhang2012,chen2017}. %chang2000,
	It is interesting to extend this problem to the RHD case. 
	The computational domain is displayed in \figref{fig:shockdiffraction}\subref{fig:shockdiffractionmesh} 
    with a sample mesh ($h=1/10$).  
     Initially, along the segment $6\leq y \leq11$ at $x=3.4$, there is a shock wave with velocity of $0.8$. For the pre-shock regime, the undisturbed air has the density of $1.4$ and the pressure of $1$. 
     The post-shock state, which can be calculated by  the Rankine-Hugoniot jump condition and the Lax entropy condition, is 
     \begin{equation*}
     (\rho_0, {\bm v}_0,p_0) = (2.58962919872684,0.40445979062926, 0, 2.865544850466692).
     \end{equation*}
     Along the walls of the wedge, reflection boundary conditions are  applied. The inflow boundary condition is specified at 
     $\{ x=0, 6 \le y \le 11 \}$, and the outflow conditions are used on the right, the top, and the bottom boundaries.

	Figs.~\ref{fig:shockdiffraction}\subref{fig:shockdiffraction3rd}, \ref{fig:shockdiffraction}\subref{fig:shockdiffraction1st40}, and \ref{fig:shockdiffraction}\subref{fig:shockdiffraction1st160} show the contour map of the rest-mass density at $t = 8$. 
	We observe that
	a diffracted shock is generated, 
	and a vortex is produced near the wedge corner.  
	\figref{fig:shockdiffraction}\subref{fig:shockdiffraction3rd} gives 
	the numerical results computed by using our third-order PCP method 
	on the mesh with $h=1/40$. 
	For validation purpose, we also present in 
	\figref{fig:shockdiffraction}\subref{fig:shockdiffraction1st40} and \figref{fig:shockdiffraction}\subref{fig:shockdiffraction1st160}  the results obtained by the first-order HLL scheme (without spatial reconstruction) on two different meshes (with $h=1/40$ and $1/160$, respectively). 
	 As expected, our third-order method 
	 has better resolution than the first-order HLL scheme. The flow structures are correctly resolved by the proposed method, and are 
	 very close to those by the first-order HLL scheme on a much refined mesh ($h=1/160$).   

	\begin{figure}[htbp]
		\centering
		\subfigure[A sample mesh with $h=1/10$]{
			\label{fig:shockdiffractionmesh}
			\includegraphics[width=0.43\textwidth]{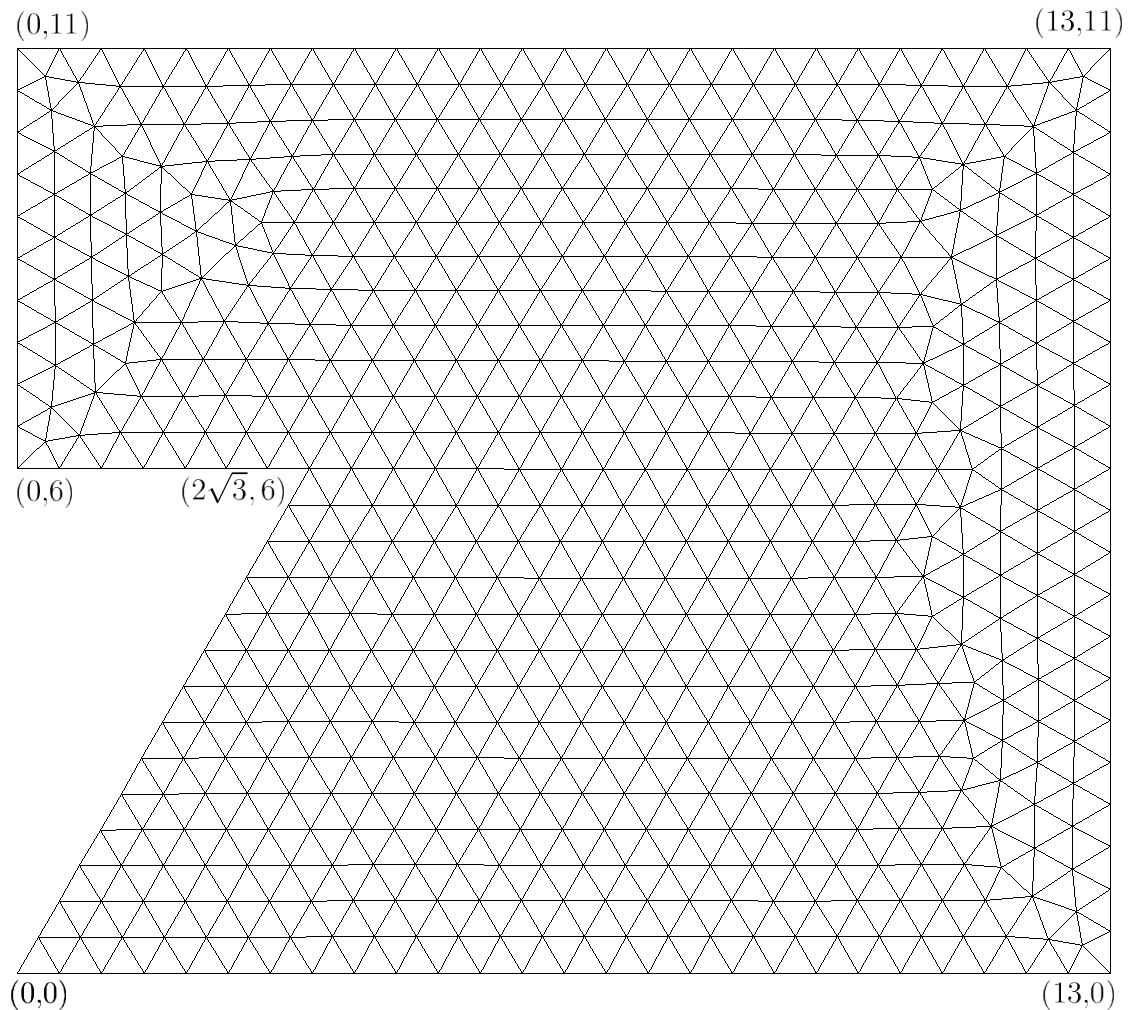}
		}
		\subfigure[Our scheme, $h=1/40$ ]{
			\label{fig:shockdiffraction3rd}
			\includegraphics[width=0.45\textwidth]{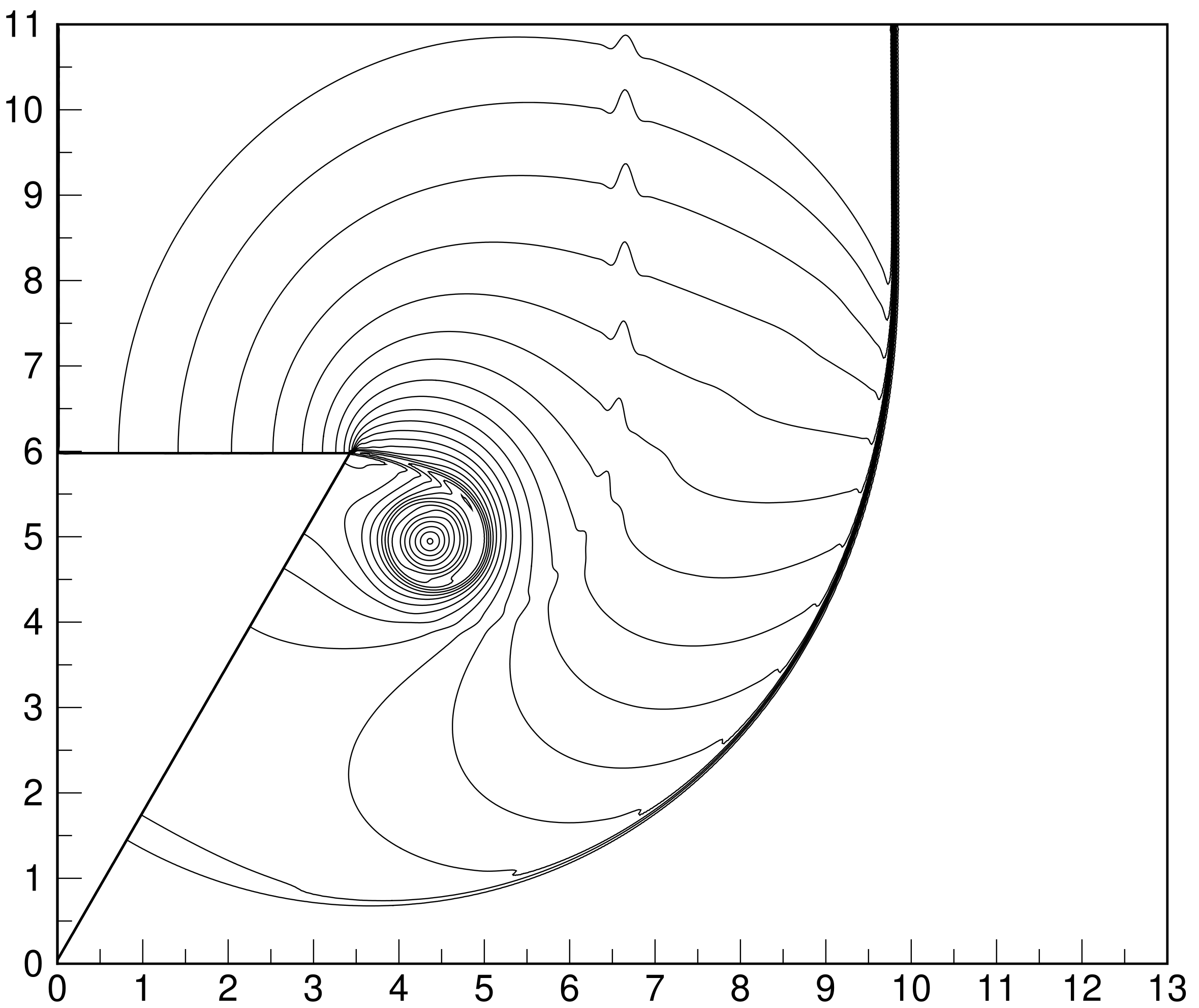}
		}
		\subfigure[First-order HLL scheme, $h=1/40$ ]{
			\label{fig:shockdiffraction1st40}
			\includegraphics[width=0.45\textwidth]{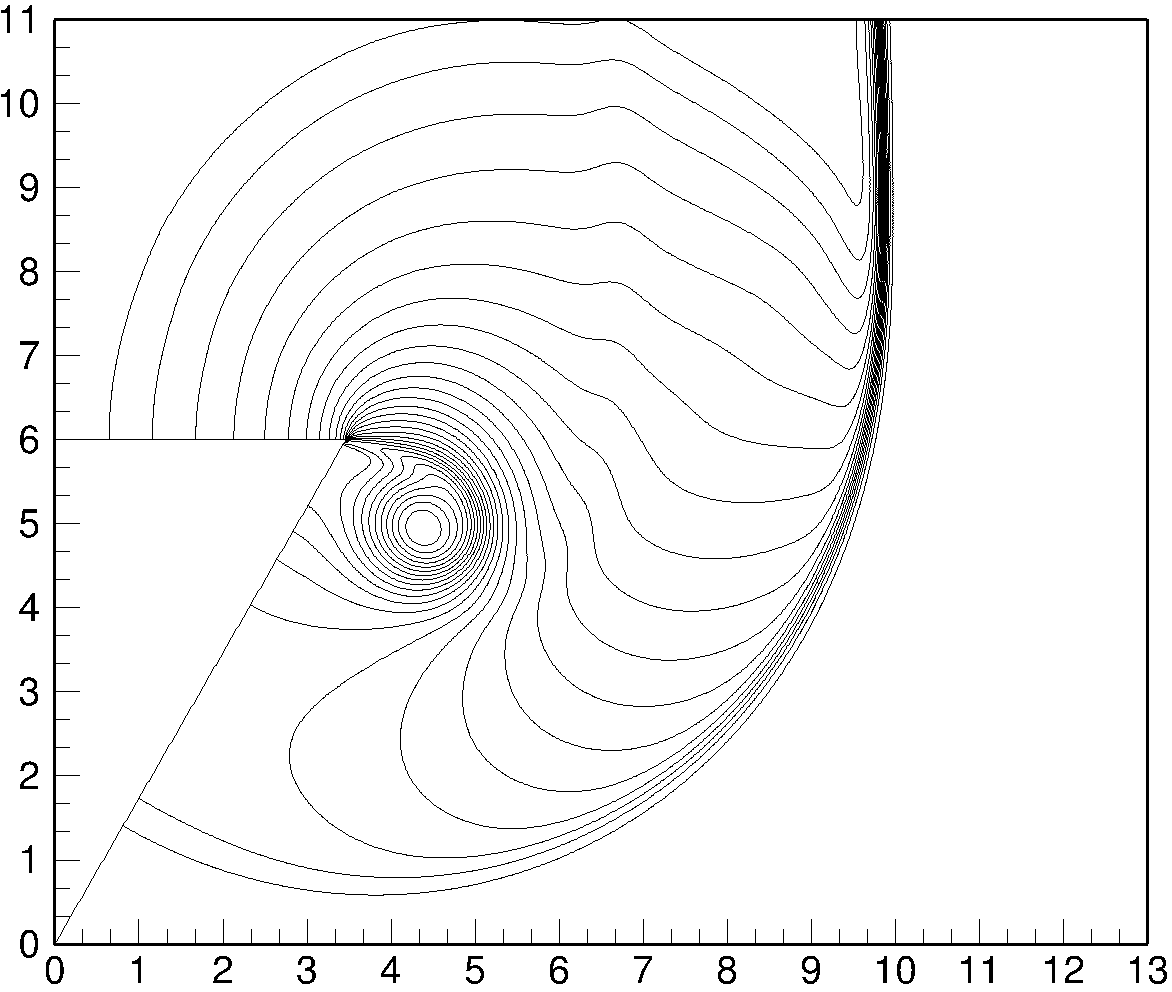}
		}
		% \subfigure[1st, $h=1/80$]{
		%   \includegraphics[width=0.45\textwidth]{results/V3_2D_shock_diffraction_1rd_1_80&160/h_1_80/rho_T8_30.png}
		% }
		\subfigure[First-order HLL scheme, $h=1/160$]{
			\label{fig:shockdiffraction1st160}
			\includegraphics[width=0.45\textwidth]{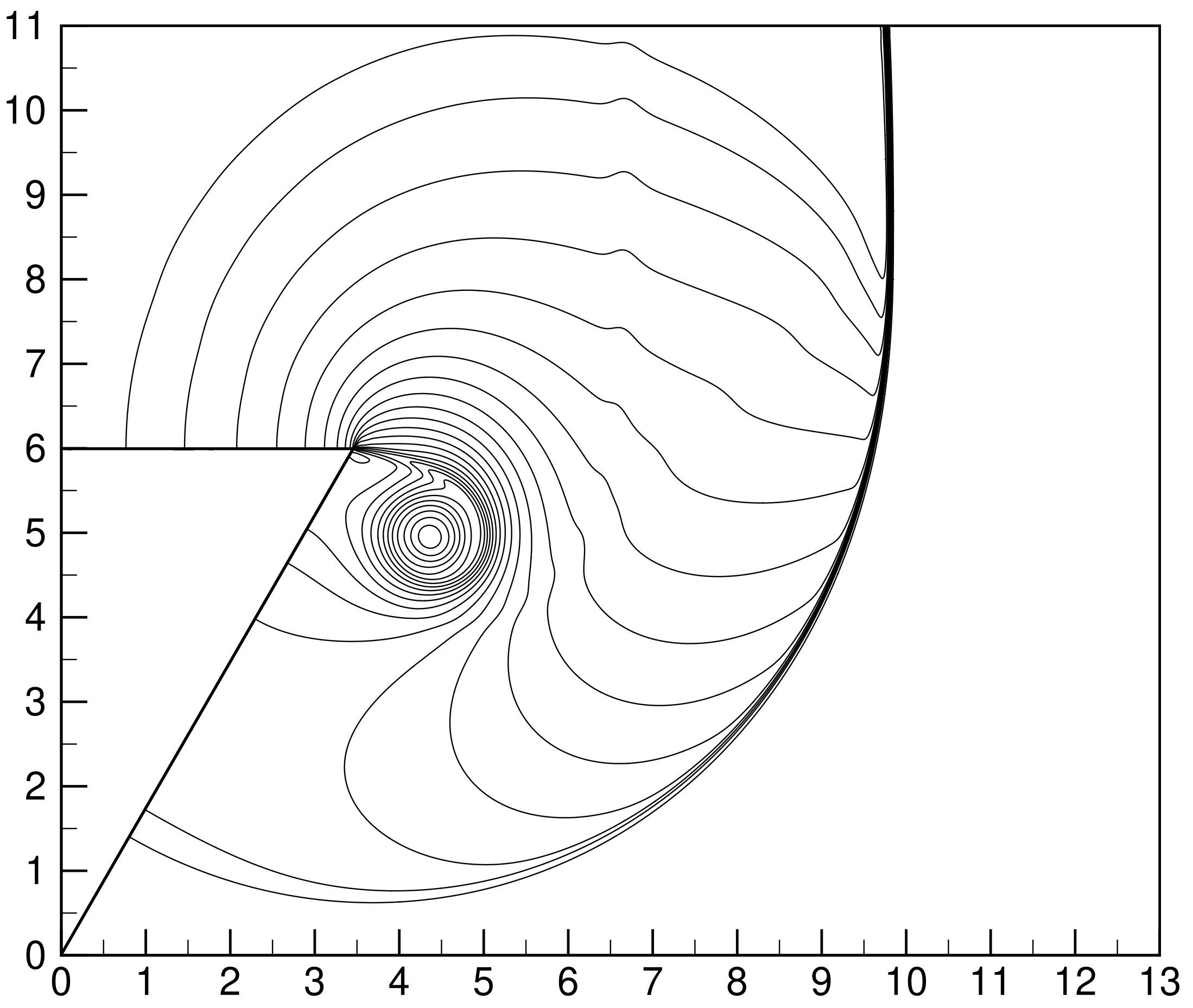}
		}
		\caption{Example \ref{shockdiffraction}: 30 equally spaced contour lines for density at $t=8$. 
			%The mesh points on the boundary are uniformly distributed with different $h$. 
		}
		\label{fig:shockdiffraction}
	\end{figure}
\end{expl}

	\begin{expl}[Axisymmetric relativistic jet]\label{jet}\rm
		In this example, we solve the axisymmetric RHD equations \eqref{eq:RHDcylind3D} 
		to simulate an axisymmetric relativistic jet, which is relevant to astrophysics and was well studied in, for example, \cite{marti1997morphology,2006Zhang,WuTang2015,QinShu2016}. The details of extending our scheme to the cylindrical RHD equations \eqref{eq:RHDcylind3D} have been discussed in Section \ref{sec:RHDAXIS}. We divide the computational domain $[0,15]\times[0,45]$ of the cylindrical coordinates $(r,z)$ into an unstructured triangular mesh with $h=1/25$. The initial states for the relativistic jet are
		\begin{eqnarray*}
			(\rho,v_z,v_r,p) = (1.0,0.0,0.0,1.70305\times 10^{-4}).
		\end{eqnarray*}
		A light jet beam is injected into the domain parallel to the axis of symmetry  (the $z$-axis) through the nozzle $(r\leq 1)$ of the bottom boundary $(z=0)$ with $\rho^b = 0.01, v_z^b = 0.99, v_r^b = 0$,  and $p^b = p$. Outflow boundary conditions are used on the domain boundaries, except at the symmetry axis ($r=0$ boundary) where the reflection conditions are imposed and at the nozzle where the inflow boundary conditions are imposed. The classic beam Mach number $M_b = v_z^b / c_s = 6$, and the corresponding relativistic Mach number $M_r:=M_bW_b/W_s$ is about $41.95$, 
		where $W_b=1/\sqrt{1-(v_z^b)^2}$ and $W_s=1/\sqrt{1-c^2_s}$ are respectively the Lorentz factors associated with the jet speed and the local sound speed. 
		
		\begin{figure}[h]\rm
			\centering
			\subfigure[$t=60$]{
				\label{fig:jet12}
				\includegraphics[width=0.3\textwidth]{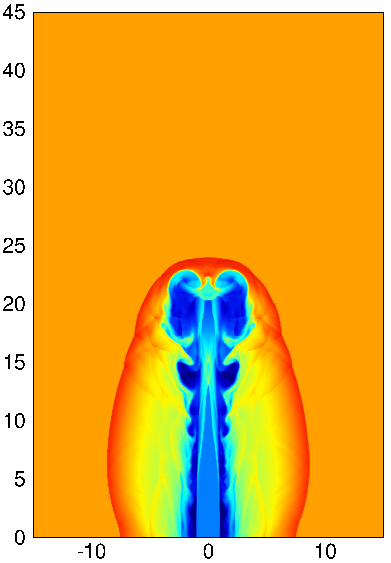}
			}
			\subfigure[$t=80$]{
				\label{fig:jet14}
				\includegraphics[width=0.3\textwidth]{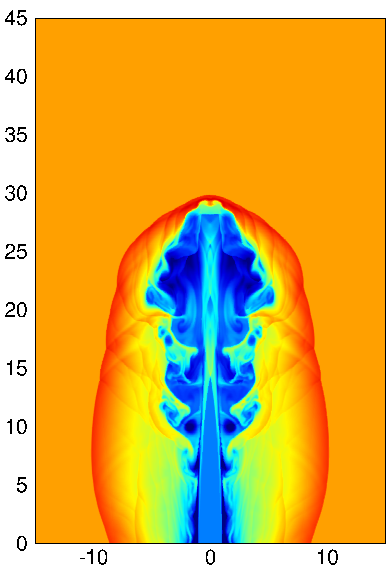}
			}
			\subfigure[$t=100$]{
				\label{fig:jet20}
				\includegraphics[width=0.3\textwidth]{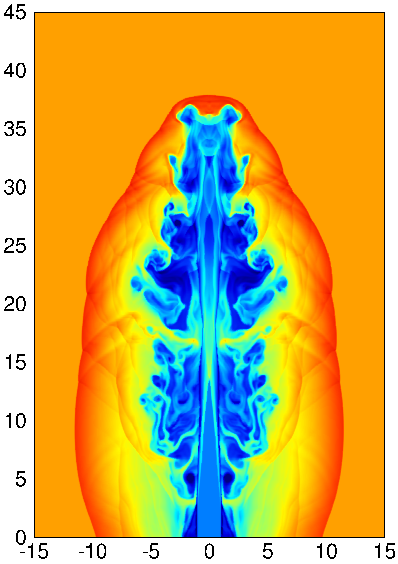}
			}
			\caption{Example \ref{jet}: The evolution of the rest-mass density logarithm $\ln \rho$ simulated by our scheme on a triangular mesh with $h=1/25$.}
			\label{3D_jet2}
		\end{figure}
		
		Fig.~\ref{3D_jet2} shows the schlieren images of the rest-mass density logarithm $\ln \rho$ at $t = 60, 80$, and $100$ obtained by our scheme. 
		%It is found that  the time evolution of a light, relativistic jet with large internal energy is well simulated by our schemes.  
		As expected, we observe a bow shock formed at the jet head, and the Kelvin--Helmholtz instabilities develop. These typical jet flow structures are correctly resolved by our scheme in comparison with \cite{marti1997morphology,2006Zhang,tchekhovskoy2007wham,WuTang2015,QinShu2016}.  Moreover, there is no carbuncle in our result. 
		The proposed PCP limiter is also necessary for this challenging test: If the limiter is turned off, the evolved cell averages would violate the physical constraints, and the high-order finite volume code would break down within a few time steps.
	\end{expl}

\section{Conclusions}\label{sec:conclusions}
		In this paper, we have developed a third-order robust finite volume WENO method for the RHD equations on unstructured triangular meshes. 
		The method has two distinctive features: the provably PCP property 
		and a scaling-invariant property (homogeneity). 
		%To the best of our knowledge, it 
		Due to the relativistic effects, the primitive quantities cannot be explicitly formulated in terms of the conservative variables, making the design and analysis of PCP schemes highly nontrivial. 
		Based on a novel quasilinear technique, 
		we have rigorously proven the PCP property of our method with the HLL flux. 
		In order to achieve high-order accuracy, we adopt the 
		efficient WENO reconstruction, recently proposed by Zhu and Qiu \cite{zhu2018new}. A modification to the nonlinear weights in the WENO method has been proposed, so that the modified nonlinear weights become scaling-invariant and that our method inherits the homogeneity of the exact evolution operator. Such scaling-invariance and homogeneity properties give our modified WENO method some advantages in resolving multi-scale wave structures. 
		We have also introduced three provable convergence-guaranteed iterative algorithms for the recovery of primitive quantities from admissible conservative variables. Extensive numerical experiments have demonstrated the robustness, accuracy, high resolution, the homogeneity, and the PCP property of the proposed method.

\begin{appendices}

\section{Eigenstructure of the rotated Jacobian matrix}\label{eigenstructure}
\begin{proposition}\label{prop111}
	If 	${\bf U}\in G_{u}$, then
	for any unit real vector ${\bf n}=(n_1,n_2)$, the Jacobian matrix $$\vec{A}_{\bf n}({\bf U}):= n_1 \frac{\partial\vec F_1(\vec{U})}{\partial {\bf U}} + n_2 \frac{\partial\vec F_2(\vec{U})}{\partial {\bf U} }$$
	has four real eigenvalues:
	\begin{align*}
		& \lambda_{\bf n}^{(1)} = \frac{v_{\bf n}(1-c_s^2)-c_s\gamma^{-1}\sqrt{1-v^2_{\bf n}-(\|{\bm v}\|^2-v^2_{\bf n})c^2_s}}{1-\|{\bm v}\|^2c^2_s},    \\
		& \lambda_{\bf n}^{(2)} = \lambda_{\bf n}^{(3)} = v_{\bf n},                                          \\
		& \lambda_{\bf n}^{(4)} = \frac{v_{\bf n}(1-c_s^2)+c_s\gamma^{-1}\sqrt{1-v^2_{\bf n}-(\|{\bm v}\|^2-v^2_{\bf n} )c^2_s}}{1-\|{\bm v}\|^2c^2_s},
	\end{align*}
	where $v_{\bf n} := \langle {\bm v}, {\bf n} \rangle = v_1 n_1 + v_2 n_2$ denotes the component of the fluid velocity in the direction of ${\bf n}$, and
	$c_s=\sqrt{\frac{\Gamma p}{\rho H}}$ denotes the sound speed for the ideal gas.
	The associated four real eigenvectors are
	\begin{align*}
		& {\bf R}_{\bf n}^{(1)}({\bf U})
		= \begin{pmatrix}
			\gamma ( v_{\bf n} \lambda_{\bf n}^{(1)} - 1 )
			\\
			H \gamma^2 \left(  \lambda_{\bf n}^{(1)} ( v_{\bf n} v_1 - n_1  ) + v_{\tau} n_2  \right)
			\\
			H \gamma^2 \left(  \lambda_{\bf n}^{(1)} ( v_{\bf n} v_2 - n_2  ) - v_{\tau} n_1  \right)
			\\
			H \gamma^2 ( v_{\bf n}^2 - 1 )
		\end{pmatrix},
		\qquad {\bf R}_{\bf n}^{(2)}({\bf U})
		= \begin{pmatrix}
			1
			\\
			\gamma v_1
			\\
			\gamma v_2
			\\
			\gamma
		\end{pmatrix}
		\\
		& {\bf R}_{\bf n}^{(4)}({\bf U})
		= \begin{pmatrix}
			\gamma ( v_{\bf n} \lambda_{\bf n}^{(4)} - 1 )
			\\
			H \gamma^2 \left(  \lambda_{\bf n}^{(4)} ( v_{\bf n} v_1 - n_1  ) + v_{\tau} n_2  \right)
			\\
			H \gamma^2 \left(  \lambda_{\bf n}^{(4)} ( v_{\bf n} v_2 - n_2  ) - v_{\tau} n_1  \right)
			\\
			H \gamma^2 ( v_{\bf n}^2 - 1 )
		\end{pmatrix},
		\qquad {\bf R}_{\bf n}^{(3)}({\bf U})
		= \begin{pmatrix}
			\gamma	v_\tau
			\\
			2H \gamma^2 v_\tau v_1 - H n_2
			\\
			2H \gamma^2 v_\tau v_2 + H n_1
			\\
			2H \gamma^2 v_\tau
		\end{pmatrix}
	\end{align*}
	with $v_\tau := v_2 n_1 - v_1 n_2$.
	The inverse of the right eigenvector matrix
	${\bf R}_{\bf n} ({\bf U}) := \left[ {\bf R}_{\bf n}^{(1)}, {\bf R}_{\bf n}^{(2)}, {\bf R}_{\bf n}^{(3)}, {\bf R}_{\bf n}^{(4)} \right]$ is given by
	\begin{equation*}
		{\bf R}_{\bf n}^{-1} ({\bf U}) =
		\begin{pmatrix}
			l_{1}^{(1)} & l_{2}^{(1)} & l_{3}^{(1)} & l_{4}^{(1)}
			\\
			\frac{1}{\rho c_s^2\eta}&
			\frac{v_{\bm n}(1+v_\tau^2\gamma^2)}{(1-v_{\bf n}^2)\gamma\rho Hc_s^2\eta} n_1
			- \frac{v_\tau\gamma}{\rho H c_s^2\eta} n_2
			&  	\frac{v_{\bm n}(1+v_\tau^2\gamma^2)}{(1-v_{\bf n}^2)\gamma\rho Hc_s^2\eta} n_2
			+ \frac{v_\tau\gamma}{\rho H c_s^2\eta} n_1
			&\frac{1+v_\tau^2\gamma^2}{(v_{\bf n}^2-1)\gamma\rho H c_s^2\eta}\\
			0
			&\frac{ -n_2 + v_{\bf n} v_2 }{H(1-v_{\bf n}^2)}
			&\frac{ n_1 - v_{\bf n} v_1 }{H(1-v_{\bf n}^2)}
			&\frac{v_\tau}{H(v_{\bf n}^2-1)}\\
			l_{1}^{(4)} & l_{2}^{(4)} & l_{3}^{(4)} & l_{4}^{(4)}
		\end{pmatrix},
	\end{equation*}
	where
	\begin{equation*}
		\begin{aligned}
			&l^{(z)}_1 = f_z \frac{v_{\bf n}  \lambda_{\bf n}^{(z)}-1}{\gamma(1+c^2_s\rho\eta)},\\
			&l^{(z)}_2 = f_z\left[\lambda_{\bf n}^{(z)}\left(\frac{ \|{\bm v}\|^2+\rho\eta}{1+c_s^2\rho\eta}n_1-v_\tau v_2 \right)-\frac{1+\rho\eta}{1+c^2_s\rho\eta} v_{\bf n} n_1 + v_\tau n_2  \right],\\
			&l^{(z)}_3 = f_z\left[\lambda_{\bf n}^{(z)}\left(\frac{ \|{\bm v}\|^2+\rho\eta}{1+c_s^2\rho\eta}n_2+v_\tau v_1 \right)-\frac{1+\rho\eta}{1+c^2_s\rho\eta} v_{\bf n} n_2 - v_\tau n_1  \right],\\
			&l^{(z)}_4 = f_z\left(v^2_\tau + \frac{1}{\gamma^2(1+c_s^2\rho\eta)}-v_{\bf n}(\lambda_{\bf n}^{(z)}-v_{\bf n})\frac{1+\rho\eta}{1+c_s^2\rho\eta}\right),
		\end{aligned}
	\end{equation*}
	with $\eta:=\frac{\partial e}{\partial p}=\frac{1}{ (\Gamma-1)\rho }$, and
	\begin{equation*}
		f_z = \frac{1+c_s^2\rho\eta}{2\rho \eta H(v_{\bf n}-\lambda_{\bf n}^{(z)})^2\gamma^2(v_{\bf n}^2+v_\tau^2c_s^2-1)},\qquad z=1, 4.
	\end{equation*}
\end{proposition}

\begin{proof}
	The eigenvalues and eigenvectors of $\vec{A}_{\bf n}({\bf U})$ can be derived from 
	those of the matrix $\frac{\partial\vec F_1(\vec{U})}{\partial {\bf U}}$
	and  
	the 
	rotational invariance of the RHD system \eqref{eq:RHD}, by following \cite{zhao2014steger}.
\end{proof}

\end{appendices}

	\bibliographystyle{siamplain}
	\bibliography{references}

\begin{thebibliography}{10}

\bibitem{abgrall1994}
{\sc R.~Abgrall}, {\em On essentially non-oscillatory schemes on unstructured
  meshes: analysis and implementation}, Journal of Computational Physics, 114
  (1994), pp.~45--58.

\bibitem{balsara2020}
{\sc D.~S. Balsara, S.~Garain, V.~Florinski, and W.~Boscheri}, {\em An
  efficient class of {WENO} schemes with adaptive order for unstructured
  meshes}, Journal of Computational Physics, 404 (2020), p.~109062.

\bibitem{BalsaraKim2016}
{\sc D.~S. Balsara and J.~Kim}, {\em A subluminal relativistic
  magnetohydrodynamics scheme with {ADER-WENO} predictor and multidimensional
  {Riemann} solver-based corrector}, Journal of Computational Physics, 312
  (2016), pp.~357--384.

\bibitem{balsara2000}
{\sc D.~S. Balsara and C.-W. Shu}, {\em Monotonicity preserving weighted
  essentially non-oscillatory schemes with increasingly high order of
  accuracy}, Journal of Computational Physics, 160 (2000), pp.~405--452.

\bibitem{2008chen}
{\sc G.~Chen, H.~Tang, and P.~Zhang}, {\em Second-order accurate {Godunov}
  scheme for multicomponent flows on moving triangular meshes.}, Journal of
  Scientific Computing, 34 (2008), pp.~64--86.

\bibitem{chen2017}
{\sc T.~Chen and C.-W. Shu}, {\em Entropy stable high order discontinuous
  {Galerkin} methods with suitable quadrature rules for hyperbolic conservation
  laws}, Journal of Computational Physics, 345 (2017), pp.~427--461.

\bibitem{chen2021second}
{\sc Y.~Chen, Y.~Kuang, and H.~Tang}, {\em Second-order accurate {BGK} schemes
  for the special relativistic hydrodynamics with the {S}ynge equation of
  state}, Journal of Computational Physics, 442 (2021), p.~110438.

\bibitem{del2002efficient}
{\sc L.~Del~Zanna and N.~Bucciantini}, {\em {An efficient shock-capturing
  central-type scheme for multidimensional relativistic flows-I.
  Hydrodynamics}}, Astronomy \& Astrophysics, 390 (2002), pp.~1177--1186.

\bibitem{dolezal1995}
{\sc A.~Dolezal and S.~Wong}, {\em Relativistic hydrodynamics and essentially
  non-oscillatory shock capturing schemes}, Journal of Computational Physics,
  120 (1995), pp.~266--277.

\bibitem{duan2019high}
{\sc J.~Duan and H.~Tang}, {\em High-order accurate entropy stable finite
  difference schemes for one-and two-dimensional special relativistic
  hydrodynamics}, Advances in Applied Mathematics and Mechanics, 12 (2020),
  pp.~1--29.

\bibitem{duffell2011tess}
{\sc P.~C. Duffell and A.~I. MacFadyen}, {\em {TESS: a relativistic
  hydrodynamics code on a moving Voronoi mesh}}, The Astrophysical Journal
  Supplement Series, 197 (2011), p.~15.

\bibitem{dumbser2007arbitrary}
{\sc M.~Dumbser and M.~K{\"a}ser}, {\em Arbitrary high order non-oscillatory
  finite volume schemes on unstructured meshes for linear hyperbolic systems},
  Journal of Computational Physics, 221 (2007), pp.~693--723.

\bibitem{Dumbser2007}
{\sc M.~Dumbser, M.~K\"{a}ser, V.~A. Titarev, and E.~F. Toro}, {\em
  Quadrature-free non-oscillatory finite volume schemes on unstructured meshes
  for nonlinear hyperbolic systems}, Journal of Computational Physics, 226
  (2007), pp.~204--243.

\bibitem{dumbser2009very}
{\sc M.~Dumbser and O.~Zanotti}, {\em Very high order {PNPM} schemes on
  unstructured meshes for the resistive relativistic {MHD} equations}, Journal
  of Computational Physics, 228 (2009), pp.~6991--7006.

\bibitem{endeve2019thornado}
{\sc E.~Endeve, J.~Buffaloe, S.~J. Dunham, N.~Roberts, K.~Andrew, B.~Barker,
  D.~Pochik, J.~Pulsinelli, and A.~Mezzacappa}, {\em thornado-hydro: towards
  discontinuous {Galerkin} methods for supernova hydrodynamics}, in Journal of
  Physics: Conference Series, vol.~1225, IOP Publishing, 2019, p.~012014.

\bibitem{friedrich1998weighted}
{\sc O.~Friedrich}, {\em Weighted essentially non-oscillatory schemes for the
  interpolation of mean values on unstructured grids}, Journal of computational
  physics, 144 (1998), pp.~194--212.

\bibitem{GottliebShuTadmor2001}
{\sc S.~Gottlieb, C.-W. Shu, and E.~Tadmor}, {\em Strong stability-preserving
  high-order time discretization methods}, SIAM Review, 43 (2001), pp.~89--112.

\bibitem{HARTEN1987231}
{\sc A.~Harten, B.~Engquist, S.~Osher, and S.~R. Chakravarthy}, {\em {Uniformly
  high order accurate essentially non-oscillatory schemes, III}}, Journal of
  Computational Physics, 71 (1987), pp.~231--303.

\bibitem{he2012adaptive1}
{\sc P.~He and H.~Tang}, {\em An adaptive moving mesh method for
  two-dimensional relativistic hydrodynamics}, Communications in Computational
  Physics, 11 (2012), pp.~114--146.

\bibitem{hu1999}
{\sc C.~Hu and C.-W. Shu}, {\em Weighted essentially non-oscillatory schemes on
  triangular meshes}, Journal of Computational Physics, 150 (1999),
  pp.~97--127.

\bibitem{Hu2013}
{\sc X.~Y. Hu, N.~A. Adams, and C.-W. Shu}, {\em Positivity-preserving method
  for high-order conservative schemes solving compressible {Euler} equations},
  Journal of Computational Physics, 242 (2013), pp.~169--180.

\bibitem{jiang1996}
{\sc G.-S. Jiang and C.-W. Shu}, {\em Efficient implementation of weighted
  {ENO} schemes}, Journal of computational physics, 126 (1996), pp.~202--228.

\bibitem{KIDDER201784}
{\sc L.~E. Kidder, S.~E. Field, F.~Foucart, and Erik}, {\em {SpECTRE: A
  task-based discontinuous Galerkin code for relativistic astrophysics}},
  Journal of Computational Physics, 335 (2017), pp.~84--114.

\bibitem{LingDuanTang2019}
{\sc D.~Ling, J.~Duan, and H.~Tang}, {\em Physical-constraints-preserving
  {Lagrangian} finite volume schemes for one- and two-dimensional special
  relativistic hydrodynamics}, Journal of Computational Physics, 396 (2019),
  pp.~507--543.

\bibitem{liu1994weighted}
{\sc X.-D. Liu, S.~Osher, and T.~Chan}, {\em Weighted essentially
  non-oscillatory schemes}, Journal of computational physics, 115 (1994),
  pp.~200--212.

\bibitem{liu2013robust}
{\sc Y.~Liu and Y.-T. Zhang}, {\em A robust reconstruction for unstructured
  {WENO} schemes}, Journal of Scientific Computing, 54 (2013), pp.~603--621.

\bibitem{Lucas2004}
{\sc A.~Lucas-Serrano, J.~A. Font, J.~M. Ib{\'a}nez, and J.~M. Marti}, {\em
  Assessment of a high-resolution central scheme for the solution of the
  relativistic hydrodynamics equations}, Astronomy and Astrophysics, 428
  (2004), pp.~703--715.

\bibitem{marquina2019capturing}
{\sc A.~Marquina, S.~Serna, and J.~M. Ib{\'a}{\~n}ez}, {\em Capturing composite
  waves in non-convex special relativistic hydrodynamics}, Journal of
  Scientific Computing, 81 (2019), pp.~2132--2161.

\bibitem{marti2003numerical}
{\sc J.~M. Mart{\'\i} and E.~M{\"u}ller}, {\em Numerical hydrodynamics in
  special relativity}, Living Reviews in Relativity, 6 (2003), p.~7.

\bibitem{Marti2015}
{\sc J.~M. Mart{\'\i} and E.~M{\"u}ller}, {\em Grid-based methods in
  relativistic hydrodynamics and magnetohydrodynamics}, Living Reviews in
  Computational Astrophysics, 1 (2015), p.~3.

\bibitem{marti1997morphology}
{\sc J.~M. Mart{\'\i}, E.~M{\"u}ller, J.~Font, J.~M.~Z. Ib{\'a}{\~n}ez, and
  A.~Marquina}, {\em Morphology and dynamics of relativistic jets}, The
  Astrophysical Journal, 479 (1997), p.~151.

\bibitem{mewes2020numerical}
{\sc V.~Mewes, Y.~Zlochower, M.~Campanelli, T.~W. Baumgarte, Z.~B. Etienne,
  F.~G.~L. Armengol, and F.~Cipolletta}, {\em Numerical relativity in spherical
  coordinates: {A} new dynamical spacetime and general relativistic {MHD}
  evolution framework for the {Einstein} {Toolkit}}, Physical Review D, 101
  (2020), p.~104007.

\bibitem{mignone2005hllc}
{\sc A.~Mignone and G.~Bodo}, {\em An {HLLC} {Riemann} solver for relativistic
  flows--{I}. {Hydrodynamics}}, Monthly Notices of the Royal Astronomical
  Society, 364 (2005), pp.~126--136.

\bibitem{easymesh}
{\sc B.~Niceno}, {\em Easymesh version 1.4: a two-dimensional quality mesh
  generator}, \url{http://web.mit.edu/easymesh_v1.4/www/easymesh.html}.

\bibitem{qin2018implicit}
{\sc T.~Qin and C.-W. Shu}, {\em Implicit positivity-preserving high-order
  discontinuous {G}alerkin methods for conservation laws}, SIAM Journal on
  Scientific Computing, 40 (2018), pp.~A81--A107.

\bibitem{QinShu2016}
{\sc T.~Qin, C.-W. Shu, and Y.~Yang}, {\em Bound-preserving discontinuous
  {Galerkin} methods for relativistic hydrodynamics}, Journal of Computational
  Physics, 315 (2016), pp.~323--347.

\bibitem{radice2011discontinuous}
{\sc D.~Radice and L.~Rezzolla}, {\em Discontinuous {Galerkin} methods for
  general-relativistic hydrodynamics: formulation and application to
  spherically symmetric spacetimes}, Physical Review D, 84 (2011), p.~024010.

\bibitem{radice2012thc}
{\sc D.~Radice and L.~Rezzolla}, {\em {THC: a new high-order finite-difference
  high-resolution shock-capturing code for special-relativistic
  hydrodynamics}}, Astronomy \& Astrophysics, 547 (2012), p.~A26.

\bibitem{radice2014high}
{\sc D.~Radice, L.~Rezzolla, and F.~Galeazzi}, {\em High-order fully
  general-relativistic hydrodynamics: new approaches and tests}, Classical and
  Quantum Gravity, 31 (2014), p.~075012.

\bibitem{rezzolla2013relativistic}
{\sc L.~Rezzolla and O.~Zanotti}, {\em Relativistic Hydrodynamics}, Oxford
  University Press, 2013.

\bibitem{shi2002technique}
{\sc J.~Shi, C.~Hu, and C.-W. Shu}, {\em {A technique of treating negative
  weights in WENO schemes}}, Journal of Computational Physics, 175 (2002),
  pp.~108--127.

\bibitem{Shu2018}
{\sc C.-W. Shu}, {\em Bound-preserving high-order schemes for hyperbolic
  equations: Recent developments}, in Theory, Numerics and Applications of
  Hyperbolic Problems II, C.~Klingenberg and M.~Westdickenberg, eds., Cham,
  2018, Springer International Publishing, pp.~591--603.

\bibitem{tchekhovskoy2007wham}
{\sc A.~Tchekhovskoy, J.~C. McKinney, and R.~Narayan}, {\em {WHAM: a WENO-based
  general relativistic numerical scheme--I. Hydrodynamics}}, Monthly Notices of
  the Royal Astronomical Society, 379 (2007), pp.~469--497.

\bibitem{teukolsky2016formulation}
{\sc S.~A. Teukolsky}, {\em Formulation of discontinuous {Galerkin} methods for
  relativistic astrophysics}, Journal of Computational Physics, 312 (2016),
  pp.~333--356.

\bibitem{wang2017compact}
{\sc Q.~Wang, Y.-X. Ren, J.~Pan, and W.~Li}, {\em {Compact high order finite
  volume method on unstructured grids III: Variational reconstruction}},
  Journal of Computational physics, 337 (2017), pp.~1--26.

\bibitem{woodward1984numerical}
{\sc P.~Woodward and P.~Colella}, {\em The numerical simulation of
  two-dimensional fluid flow with strong shocks}, Journal of computational
  physics, 54 (1984), pp.~115--173.

\bibitem{Wu2017}
{\sc K.~Wu}, {\em Design of provably physical-constraint-preserving methods for
  general relativistic hydrodynamics}, Physical Review D, 95 (2017), p.~103001.

\bibitem{Wu2017a}
{\sc K.~Wu}, {\em Positivity-preserving analysis of numerical schemes for ideal
  magnetohydrodynamics}, SIAM Journal on Numerical Analysis, 56 (2018),
  pp.~2124--2147.

\bibitem{WuMEP2021}
{\sc K.~Wu}, {\em Minimum principle on specific entropy and high-order accurate
  invariant region preserving numerical methods for relativistic
  hydrodynamics}, SIAM Journal on Scientific Computing, 43 (2021),
  pp.~B1164--B1197.

\bibitem{WuShu2018}
{\sc K.~Wu and C.-W. Shu}, {\em A provably positive discontinuous {Galerkin}
  method for multidimensional ideal magnetohydrodynamics}, SIAM Journal on
  Scientific Computing, 40 (2018), pp.~B1302--B1329.

\bibitem{WuShu2019}
{\sc K.~Wu and C.-W. Shu}, {\em Provably positive high-order schemes for ideal
  magnetohydrodynamics: analysis on general meshes}, Numerische Mathematik, 142
  (2019), pp.~995--1047.

\bibitem{WuShu2019SISC}
{\sc K.~Wu and C.-W. Shu}, {\em Entropy symmetrization and high-order accurate
  entropy stable numerical schemes for relativistic {MHD} equations}, SIAM
  Journal on Scientific Computing, 42 (2020), pp.~A2230--A2261.

\bibitem{Wu2021GQL}
{\sc K.~Wu and C.-W. Shu}, {\em Geometric quasilinearization framework for
  analysis and design of bound-preserving schemes}, arXiv preprint
  arXiv:2111.04722,  (2021).

\bibitem{wu2021provably}
{\sc K.~Wu and C.-W. Shu}, {\em Provably physical-constraint-preserving
  discontinuous {Galerkin} methods for multidimensional relativistic {MHD}
  equations}, Numerische Mathematik,  (2021), pp.~1--43.

\bibitem{WuTang2015}
{\sc K.~Wu and H.~Tang}, {\em High-order accurate
  physical-constraints-preserving finite difference {WENO} schemes for special
  relativistic hydrodynamics}, Journal of Computational Physics, 298 (2015),
  pp.~539--564.

\bibitem{WuTangM3AS}
{\sc K.~Wu and H.~Tang}, {\em Admissible states and
  physical-constraints-preserving schemes for relativistic magnetohydrodynamic
  equations}, Mathematical Models and Methods in Applied Sciences, 27 (2017),
  pp.~1871--1928.

\bibitem{WuTang2017ApJS}
{\sc K.~Wu and H.~Tang}, {\em Physical-constraint-preserving central
  discontinuous {Galerkin} methods for special relativistic hydrodynamics with
  a general equation of state}, The Astrophysical Journal Supplement Series,
  228 (2017), 3.

\bibitem{Xiong2016}
{\sc T.~Xiong, J.-M. Qiu, and Z.~Xu}, {\em Parametrized positivity preserving
  flux limiters for the high order finite difference {WENO} scheme solving
  compressible {Euler} equations}, Journal of Scientific Computing, 67 (2016),
  pp.~1066--1088.

\bibitem{xu2014parametrized}
{\sc Z.~Xu}, {\em Parametrized maximum principle preserving flux limiters for
  high order schemes solving hyperbolic conservation laws: one-dimensional
  scalar problem}, Mathematics of Computation, 83 (2014), pp.~2213--2238.

\bibitem{XuZhang2017}
{\sc Z.~Xu and X.~Zhang}, {\em Bound-preserving high order schemes}, in
  Handbook of Numerical Methods for Hyperbolic Problems: Applied and Modern
  Issues, edited by R. Abgrall and C.-W. Shu, vol.~18, North-Holland,
  Amsterdam, 2017, Elsevier.

\bibitem{2006Zhang}
{\sc W.~Zhang and A.~I. MacFadyen}, {\em {RAM: A relativistic adaptive mesh
  refinement hydrodynamics code}}, The Astrophysical Journal Supplement Series,
  164 (2006), p.~255.

\bibitem{ZHANG2017301}
{\sc X.~Zhang}, {\em On positivity-preserving high order discontinuous
  {Galerkin} schemes for compressible {Navier-Stokes} equations}, Journal of
  Computational Physics, 328 (2017), pp.~301--343.

\bibitem{zhang2010}
{\sc X.~Zhang and C.-W. Shu}, {\em On maximum-principle-satisfying high order
  schemes for scalar conservation laws}, Journal of Computational Physics, 229
  (2010), pp.~3091--3120.

\bibitem{zhang2010b}
{\sc X.~Zhang and C.-W. Shu}, {\em On positivity-preserving high order
  discontinuous {Galerkin} schemes for compressible {Euler} equations on
  rectangular meshes}, Journal of Computational Physics, 229 (2010),
  pp.~8918--8934.

\bibitem{ZhangShuReview}
{\sc X.~Zhang and C.-W. Shu}, {\em Maximum-principle-satisfying and
  positivity-preserving high-order schemes for conservation laws: survey and
  new developments}, Proceedings of the Royal Society A: Mathematical, Physical
  and Engineering Sciences, 467 (2011), pp.~2752--2776.

\bibitem{zhang2012}
{\sc X.~Zhang, Y.~Xia, and C.-W. Shu}, {\em Maximum-principle-satisfying and
  positivity-preserving high order discontinuous {Galerkin} schemes for
  conservation laws on triangular meshes}, Journal of Scientific Computing, 50
  (2012), pp.~29--62.

\bibitem{zhao2014steger}
{\sc J.~Zhao, P.~He, and H.~Tang}, {\em {Steger--Warming} flux vector splitting
  method for special relativistic hydrodynamics}, Mathematical Methods in the
  Applied Sciences, 37 (2014), pp.~1003--1018.

\bibitem{zhao2013runge}
{\sc J.~Zhao and H.~Tang}, {\em {Runge--Kutta} discontinuous {Galerkin} methods
  with {WENO} limiter for the special relativistic hydrodynamics}, Journal of
  Computational Physics, 242 (2013), pp.~138--168.

\bibitem{zhu2016new}
{\sc J.~Zhu and J.~Qiu}, {\em {A new fifth order finite difference WENO scheme
  for solving hyperbolic conservation laws}}, Journal of Computational Physics,
  318 (2016), pp.~110--121.

\bibitem{zhu2017new}
{\sc J.~Zhu and J.~Qiu}, {\em {A new type of finite volume WENO schemes for
  hyperbolic conservation laws}}, Journal of Scientific Computing, 73 (2017),
  pp.~1338--1359.

\bibitem{zhu2017tetrahedral}
{\sc J.~Zhu and J.~Qiu}, {\em A new third order finite volume weighted
  essentially non-oscillatory scheme on tetrahedral meshes}, Journal of
  Computational Physics, 349 (2017), pp.~220--232.

\bibitem{zhu2018new}
{\sc J.~Zhu and J.~Qiu}, {\em New finite volume weighted essentially
  nonoscillatory schemes on triangular meshes}, SIAM Journal on Scientific
  Computing, 40 (2018), pp.~A903--A928.

\bibitem{zhu2018MRWENO}
{\sc J.~Zhu and C.-W. Shu}, {\em A new type of multi-resolution {WENO} schemes
  with increasingly higher order of accuracy}, Journal of Computational
  Physics, 375 (2018), pp.~659--683.

\end{thebibliography}

\end{document}